\documentclass{article}
\usepackage[a4paper, margin=3cm, bottom=4cm]{geometry}
\usepackage[utf8]{inputenc}
\usepackage[T1]{fontenc}
\usepackage{lmodern}

\usepackage[
    backend=biber,
    style=alphabetic, maxalphanames=5, 
    maxbibnames=100, maxsortnames=100, 
    sorting=nyt,
    natbib=true,
    url=false,doi=false,isbn=false,eprint=false
]{biblatex}
\usepackage{graphicx}
\usepackage{subcaption}
\graphicspath{ {images/} }
\usepackage{enumitem}
\usepackage[normalem]{ulem}
\usepackage[toc,page]{appendix}

\usepackage[bookmarks, bookmarksnumbered,
			colorlinks=true,
            linkcolor = blue,
            urlcolor  = blue,
            citecolor = teal]{hyperref}

\newcommand\fnsurl[1]{{\footnotesize\url{#1}}}

\usepackage{myquicksetup}
\numberwithin{definition}{section}
\numberwithin{theorem}{section}
\numberwithin{corollary}{section}
\numberwithin{proposition}{section}
\numberwithin{lemma}{section}
\numberwithin{claim}{section}
\numberwithin{fact}{section}
\numberwithin{remark}{section}
\numberwithin{example}{section}
\numberwithin{equation}{section}
\mathtoolsset{showonlyrefs} 

\usepackage{mylivemacros}

\usepackage{local_convergence_MFL} 

\usepackage{subfiles} 

\newif\ifextended  
\extendedtrue

\title{Local convergence of mean-field Langevin dynamics: \\ from gradient flows to linearly monotone games}
\date{\today}

\author{Guillaume Wang and Lénaïc Chizat}

\begin{document}
\maketitle

\begin{abstract}
    We study the local convergence of diffusive mean-field systems, including Wasserstein gradient flows, min-max dynamics, and multi-species games. We establish exponential local convergence in $\chi^2$-divergence with sharp rates, under two main assumptions: (i) the stationary measures satisfy a Poincaré inequality, and (ii) the velocity field satisfies a monotonicity condition, which reduces to linear convexity of the objective in the gradient flow case. We do not assume any form of displacement convexity or displacement monotonicity.
    
    In the gradient flow case, global exponential convergence is already known under our linear convexity assumption, with an asymptotic rate governed by the log-Sobolev constant of the stationary measure. Our contribution in this setting is to identify the sharp rate near equilibrium governed instead by the Poincaré constant. This rate coincides with the one suggested by Otto calculus (i.e.~by a tight positivity estimate of the Wasserstein Hessian), and refines some results of Tamura (1984), extending them beyond quadratic objectives.
    
   More importantly, our proof technique extends to certain non-gradient systems, such as linearly monotone two-player and multi-player games. In this case, we obtain explicit local exponential convergence rates in $\chi^2$-divergence, thereby partially answering the open question raised by the authors at COLT 2024. While that question concerns global convergence (which remains open), even local convergence results were previously unavailable.

At the heart of our analysis is the design of a Lyapunov functional that mixes the $\chi^2$-divergence with weighted negative Sobolev norms of the density relative to equilibrium.
    

\end{abstract}

\section{Introduction} \label{sec:intro}

Wasserstein gradient flows (WGFs) are a class of PDEs over $\PPP(\RR^d)$, the space of probability measures on $\RR^d$, which exhibit a particularly appealing geometric structure.
They correspond to gradient flow curves for the 2-Wasserstein metric and, as shown in foundational works by Otto \cite{jordan1998variational,otto2000generalization,otto2001geometry}, they can be treated analogously to gradient flows on a Riemannian manifold.
Theoretical charm aside, the study of WGFs is motivated by the fact that they include a number of examples of interest across applied mathematics, such as the granular media equation in mathematical physics \cite{carrillo2003kinetic} or the training dynamics of two-layer neural networks in machine learning \cite{mei2018mean}. 

A common feature of many of the PDEs of WGF type
which are encountered in practice
is that they include a diffusion term. That is, they take the form
\begin{equation} \label{eq:intro:MFLD}
    \forall t > 0,~~
    \partial_t \mu_t = \nabla \cdot (\mu_t \nabla F'[\mu_t]) + \tau \Delta \mu_t
\end{equation}
for some $\tau>0$ and $F: \PPP(\RR^d) \to \RR$,
where
$F'[\mu]: \RR^d \to \RR$ denotes the first variation at $\mu$.
Equivalently, the above PDE is the WGF of the entropy-regularized functional
\begin{equation*}
    F_\tau(\mu) = F(\mu) + \tau H(\mu),
\end{equation*}
where $H(\mu) = \int_{\RR^d} \log \frac{\d\mu}{\d x} \,\d\mu$
denotes the (negative) differential entropy.

The PDE \eqref{eq:intro:MFLD} is called the mean-field Langevin dynamics (MFLD) of $F$ with temperature~$\tau$ \cite{hu2021mean}.
Its well-posedness is ensured as soon as $F$ is displacement-smooth and $\mu_0$ has finite second moments, which we will assume throughout.
Its long-time convergence behavior
has been the subject of extensive investigations.
Notably, \cite{chizat2022mean,nitanda2022convex} showed that
if $F$ is linearly convex,%
\footnote{We call a functional $F$ over $\PPP(\RR^d)$ \emph{linearly convex} if for any $\mu, \nu$, the map $t \mapsto F((1-t) \mu + t \nu)$ is convex over $[0,1]$. The qualifier ``linearly'' emphasizes the distinction with displacement convexity, which refers to convexity along Wasserstein geodesics.}
then the stationary measure $\nu$ is unique if it exists, and under an additional uniform log-Sobolev inequality assumption, $\mu_t$ converges exponentially to $\nu$ 
in Kullback-Leibler (KL) divergence
from any initialization~$\mu_0$.

In this work, we consider MFLDs $(\mu_t)_{t \geq 0}$ \emph{which are assumed to converge} to some $\nu$ as $t \to \infty$, and we study their long-time convergence rates in terms of properties of $F$, $\tau$, and $\nu$. 
We do this by analyzing the local convergence behavior of $(\mu_t)_t$ when
initialized
in a small $\chi^2$-divergence ball around~$\nu$.
We emphasize that our focus is on estimating the exact rate of exponential convergence;
in particular, ours is a different kind of ``local'' analysis than in the recent work of \cite{monmarche2025local}, who showed the convergence of MFLDs initialized in a Wasserstein neighborhood of a stationary measure 
under a
local displacement Polyak-Lojasiewicz inequality
(with potentially pessimistic estimates of the convergence rate).

In its most basic form,
the result of our analysis for MFLD \eqref{eq:intro:MFLD} is as follows.
We refer to \autoref{def:intro:LSI_T2_PI} for the definition of Poincar\'e inequality (PI).
\begin{theorem}[Local $\chi^2$ convergence of MFLD, informal] \label{thm:intro:L2_MFLD}
    Suppose that $F$ is displacement-smooth and linearly convex and that MFLD admits a stationary measure~$\nu$.
	If $\nu$ satisfies PI with a constant $c_{\PI}$, then MFLD converges locally at an exponential rate of at least $2 \tau c_{\PI}$ in $\chisq{\cdot}{\nu}$. The local convergence occurs for initializations lying in small sublevel sets of $\chisq{\cdot}{\nu}$.
\end{theorem}

More precise versions of the statement
and their proofs,
as well as a discussion of the implications, are placed in \autoref{sec:MFLD_L2}.
In brief,
\begin{itemize}
    \item The assumption that $F$ is globally linearly convex can be relaxed in two directions.
    It can be replaced by (1) a ``local'' condition involving only its second variation at stationarity, $F''[\nu]$,
    or by (2) global 
    weak convexity of $F$ relative to $H$,
    i.e., linear convexity of $F + \tau_0 H$ for some $0 \leq \tau_0 < \tau$. In the latter case, the local exponential rate lower bound becomes $2 (\tau-\tau_0) c_{\PI}$.
    These two relaxations can be combined to a condition involving $F''[\nu]$ and $\tau_0$ which also suffices for our analysis, stated as \autoref{assum:MFLD_A}.
    \item In the case where $F$ is quadratic, i.e., of the form 
    \begin{equation} \label{eq:intro:quadr_F}
        F(\mu) = \int V \,\d\mu + \frac12 \iint k(x,x') \,\d\mu(x) \d\mu(x'),
    \end{equation}
    explicit estimates for the size of the local
    neighborhood and for the constant prefactor in the exponential convergence bound can be obtained (\autoref{thm:MFLD_L2:localL2:quad}).
    We note that for quadratic $F$, our analysis is essentially equivalent to the one of \cite[Section~5]{tamura1984asymptotic}, 
    only with much more explicit constants;
    see Related works below for further discussion.
    \item Our estimate of the local rate is tight in the case where $F$ is linear. Indeed in this case, \eqref{eq:intro:MFLD} reduces to
    (the Fokker-Planck equation of)
    the overdamped Langevin dynamics,
    whose $\LLL^2_\nu$ convergence rate is precisely characterized by the PI constant of $\nu$ (\autoref{prop:MFLD_L2:localL2:OL}).
    \item Under additional qualitative assumptions on $F$,
    our result implies 
	a long-time convergence rate estimate for MFLD
    in Wasserstein distance, KL-divergence, and $\chi^2$-divergence to $\nu$, as well as in $F_\tau(\cdot) - F_\tau(\nu)$ (\autoref{coroll:MFLD_L2:longtime:allmetrics}).
    \item Our result is consistent with heuristic Otto calculus computations for the Wasserstein Hessian of $F_\tau$ at stationarity, developed in \autoref{sec:heuri_pfids}.
    This can be viewed as a non-linear extension of the fact that
    PI for $\nu$ is equivalent to positive-definiteness of the Wasserstein Hessian of $\KLdiv{\cdot}{\nu}$ at $\nu$ itself, with the same constant.
\end{itemize}

Besides the MFLD PDE proper,
we also obtain a corresponding result for the mean-field Langevin descent-ascent dynamics (MFL-DA), which is the dynamics over $\PPP(\XXX) \times \PPP(\YYY)$, for Riemannian manifolds $\XXX, \YYY$ and a payoff function $k: \XXX \times \YYY \to \RR$, defined by
\begin{equation} \label{eq:intro:MFL-DA}
	\begin{cases}
		\partial_t \mu^x_t = \nabla \cdot \left( \mu^x_t \nabla \int_\YYY k(\cdot, y) \d\mu^y_t(y) \right) + \tau \Delta \mu^x_t \\
		\partial_t \mu^y_t = -\nabla \cdot \left( \mu^y_t \nabla \int_\XXX k(x, \cdot) \d\mu^x_t(x) \right) + \tau \Delta \mu^y_t.
	\end{cases}
\end{equation}
We refer to 
Related works below
for the significance of this dynamics in game theory and for a discussion of recent works studying its global convergence behavior. In particular, under mild assumptions, there exists a unique equilibrium pair $(\nu^x, \nu^y)$ since MFL-DA is the Wasserstein gradient descent-ascent flow of $(\mu^x, \mu^y) \mapsto \iint_{\XXX \times \YYY} k \,\d(\mu^x \otimes \mu^y) + \tau H(\mu^x) - \tau H(\mu^y)$, which is strictly linearly convex-concave.
Yet, whether MFL-DA converges at all is an open problem \cite{wang2024open} 
(unless
$\tau$ is large or
$k$ is assumed convex-concave \cite{cai2024convergence}).
In this work we focus on the more modest goal of establishing its local convergence, and we obtain the following.

\begin{theorem}[Local $\chi^2$ convergence of MFL-DA, informal] \label{thm:intro:L2_MFL-DA}
	Suppose
    $\nabla_x \nabla_y k(x,y)$ is uniformly bounded
    and MFL-DA admits an equilibrium $(\nu^x, \nu^y)$.
	If $\nu^x$ and $\nu^y$ satisfy PIs with constants $c_{\PI}^x$ resp.\ $c_{\PI}^y$, then MFL-DA converges locally at an exponential rate of at least $2 \tau \min\{c_{\PI}^x, c_{\PI}^y\}$ in $\chisq{\cdot}{\nu^x} + \chisq{\cdot}{\nu^y}$. The local convergence occurs for initializations lying in small sublevel sets of $\chisq{\cdot}{\nu^x} + \chisq{\cdot}{\nu^y}$.
\end{theorem}

The precise version of this statement and its proof are placed in \autoref{sec:MFL-DA}.
The analysis can be generalized to the two-timescale version of the dynamics considered in \cite{lu2023two,an2025convergence}, as we discuss in the same section, and to a $N$-player and non-multi-linear setting, as we explain now.

The most general setting in which we apply our analysis is that of $N$-species flows with diffusion, i.e., dynamics over 
$(\mu^1_t, ..., \mu^N_t) \in 
\PPP_2(\RR^{d_1}) \times ... \times \PPP_2(\RR^{d_N})$ 
of the form
\begin{equation} \label{eq:intro:Nspecies_flow}
	\forall t \geq 0, \forall I \in \{1,...,N\},~
	\partial_t \mu^I_t = \nabla \cdot \left( \mu^I_t \nabla V_I[\mu^1_t, ..., \mu^N_t] \right) + \tau \Delta \mu^I_t,
\end{equation}
for appropriately regular potentials $V_I: \PPP_2(\RR^{d_1}) \times ... \times \PPP_2(\RR^{d_N}) \times \RR^{d_I} \to \RR$.
The global convergence of such dynamics was studied in \cite{conger2025monotone} assuming strong displacement monotonicity, the natural generalization of strong displacement convexity for multi-species contexts.
In this work, we study the local convergence of \eqref{eq:intro:Nspecies_flow}
to an equilibrium tuple under \emph{linear monotonicity}.
In particular, our analysis applies to the generalization of MFL-DA to multi-player pairwise-zero-sum polymatrix continuous games considered in \cite{lu2025convergence}.
Our result, 
loosely stated,
is as follows.

\begin{theorem}[Local $\chi^2$ convergence of multi-species flows with diffusion, informal] \label{thm:intro:MFG}
	Suppose the potentials $V_I$ in \eqref{eq:intro:Nspecies_flow} 
	have $\CCC^3$ first and second variations,
	and that they are linearly monotone, i.e.,
	\begin{equation*}
		\forall \mu^1, \tmu^1 \in \PPP_2(\RR^{d_1}), ..., \mu^N, \tmu^N \in \PPP_2(\RR^{d_N}),~~
		\sum_I \int_{\RR^{d_I}} 
		\left( V_I[\mu^1, ..., \mu^N] - V_I[\tmu^1, ..., \tmu^N] \right) \, \d(\mu^I - \tmu^I) \geq 0.
	\end{equation*}
	Further suppose that there exists an equilibrium $(\nu^1, ..., \nu^N)$.
	If $\nu^I$ satisfies PI with a constant $c_{\PI}^I$ for each $I$, then \eqref{eq:intro:Nspecies_flow} converges locally at an exponential rate of at least 
	$2 \tau \min_I c_{\PI}^I$
	in $\sum_I \chisq{\cdot}{\nu^I}$. The local convergence occurs for initializations lying in small sublevel sets of $\sum_I \chisq{\cdot}{\nu^I}$.
\end{theorem}

In fact the global linear monotonicity assumption can be relaxed, similar to the case of MFLD, to a ``local'' condition on the kernels $k_{IJ}(z^I, z^J) = \frac{\delta V_I[\nu^1, ..., \nu^N](z^I)}{\delta \nu^J(z^J)}$,
and/or to a weak linear monotonicity condition, in the sense of \autoref{assum:MFG_B}.
The full statement of our result 
and its proof 
are provided in \autoref{sec:MFG}.
In particular, this result illustrates that a gradient flow structure is not essential for our proof technique,
since the $V_I$ may not correspond to the Wasserstein gradients of fixed functionals.

\vspace{1em}
The remainder of the paper is organized as follows.
\autoref{subsec:intro:related} contains a discussion of related works.
In \autoref{sec:heuri_pfids} we present a heuristic argument supporting \autoref{thm:intro:L2_MFLD} based on Otto calculus, and introduce our main proof ideas.
In \autoref{sec:MFLD_L2} we present our local 
convergence analysis of MFLD.
We show how the analysis can be adapted to obtain our results on MFL-DA in \autoref{sec:MFL-DA}, and on multi-species flows in \autoref{sec:MFG}.
We conclude with directions for future work in \autoref{sec:ccl}.

\subsection{Related works} \label{subsec:intro:related}

\paragraph{Convergence analyses of MFLD.}
The convergence in time of MFLDs, i.e., of PDEs of the form \eqref{eq:intro:MFLD}, is a classical topic in mathematical physics in the context of interacting particle systems, and has recently also attracted some attention in the machine learning theory community.

The global convergence behavior of MFLD for displacement-convex $F$ is very well understood \cite{carrillo2003kinetic,ambrosio2008gradient}.
The convergence properties of overdamped Langevin dynamics,
corresponding precisely to MFLD with linear $F$,
are also very well studied \cite{bakry2014analysis}.
The global convergence of MFLD for linearly convex $F$ was proved by \cite{hu2021mean}, and a quantitative exponential rate under an additional uniform log-Sobolev inequality (LSI) condition was 
shown by \cite{chizat2022mean} and \cite{nitanda2022convex} (independently);
all three were motivated by applications to two-layer neural network training \cite{mei2018mean}.
The global guarantee in KL-divergence of \cite{chizat2022mean,nitanda2022convex} was improved to guarantees in $\LLL^p$ norm by \cite{chen2022uniform}, under the same assumptions on $F$.
Another quantitative global convergence result appeared in \cite[Thm.~1.1(b)]{carrillo2020long}, specific to the case of MFLD over a torus for linearly convex quadratic $F$, with $V = 0$ and $k$ of the form $k(x,y) = W(x-y)$ in \eqref{eq:intro:quadr_F}.

The local convergence of MFLDs,
i.e.,
the stability of mean-field systems around an equilibrium, is a classical subject in mathematical physics, with most works focusing on the case of quadratic~$F$.
The work most closely related to ours, in this direction, is \cite{tamura1984asymptotic} --- in fact our analysis specialized to quadratics is essentially equivalent to its Section~5 with more modern notations.
At the technical level, the only difference is that we do not require $\nabla_x k(x,x')$ uniformly bounded
in \eqref{eq:intro:quadr_F},
as we use a $H^{-1}$ instead of a $\LLL^2$ estimate at the step corresponding to its Lemma~5.7(i).
However the
constants were only very loosely tracked in that work, and no intuition on their significance was provided.

Other than \cite{tamura1984asymptotic}, the work most directly comparable to ours is \cite{cormier2022stability}.
It provides a general abstract criterion for MFLD with quadratic $F$ 
to be locally stable in 1-Wasserstein distance, i.e., for the existence of constants $C, \delta, \lambda$ such that 
$W_1(\mu_t, \nu) \leq C e^{-\lambda t} W_1(\mu_0, \nu)$
for any $\mu_0$ with $W_1(\mu_0, \nu) \leq \delta$.
(Our convergence bounds take the same form with $\chi^2$-divergence instead of $W_1$.)
The criterion is stated in terms of the spectra of 
a certain family of linear operators which depend on $\nu$ and on the kernel $k$ in \eqref{eq:intro:quadr_F}.
Similar to our work, the gradient flow structure of MFLD is not essential for their analysis, 
and their results also apply for
(in fact, they were stated for)
general McKean-Vlasov dynamics with pairwise interactions.
Contrary to our work, no effort was made in that paper to quantify the rate, i.e., to characterize the largest $\lambda$ such that the above holds.

From a different direction, 
the authors of \cite{monmarche2025local} investigated the local guarantees that can be obtained 
from the insights of \cite{chizat2022mean,nitanda2022convex} around uniform LSI, 
when $F_\tau$ is not linearly convex.
They showed how a ``local non-linear LSI'' condition can be verified in some settings --- that is, a displacement Polyak-Lojasiewicz inequality for $F_\tau$ on a Wasserstein neighborhood of a stationary measure $\nu$.
In such settings, 
MFLD is locally contracting in $F_\tau(\cdot) - F_\tau(\nu)$, and hence locally stable in Wasserstein distance.
Their proof method applies, notably, to the granular media equation with a double-well potential and squared-distance interaction kernel.

Finally, still in connection with the uniform LSI approach, we note that \cite[Prop.~5.1]{wang2024mean} proved local convergence in KL-divergence for the MFLD of a particular functional $F$ (on a sphere), with an exponential rate given by the LSI constant of the stationary measure. The proof can be generalized to general $F$ assuming 
uniform boundedness of $\nabla_x F''[\mu](x,x')$ or of $F''[\mu](x,x')$,
as shown in \autoref{prop:apx_MFLD_localKL:MFLD_localKL}.
This excludes, however, examples such as pairwise interacting particle systems on $\RR^d$ with squared-distance interaction kernel.

\paragraph{The mean-field Langevin descent-ascent (MFL-DA) dynamics.}
The MFL-DA PDE \eqref{eq:intro:MFL-DA} represents a natural min-max optimization dynamics for the computation of mixed Nash equilibria of two-player zero-sum games \cite{hsieh2019finding,domingo2020mean}. 
Despite its simplicity,
known global convergence guarantees 
are limited to cases where the entropy regularization $\tau$ is large, or where the payoff function $k(x,y)$ is convex-concave \cite{conger2024coupled,cai2024convergence} or additively separable. 
We refer to the open problem statement \cite{wang2024open} for a review of the associated literature.
In particular, \cite{lu2023two,an2025convergence} obtained global convergence guarantees for a two-timescale variant of the dynamics, reproduced as \eqref{eq:MFL-DA:2scale} below, which we also analyze.

\paragraph{Multi-species flows.}
The dynamics \eqref{eq:intro:Nspecies_flow} was put forward in the recent work of \cite{conger2025monotone} as a unifying framework for the multi-species mass-preserving flows arising in a variety of fields.
We refer to their Sections~1 and~5 for a full discussion of the related literature and for concrete examples.
Their main contribution was the analysis of the long-time behavior of \eqref{eq:intro:Nspecies_flow} assuming strong displacement monotonicity, a setting under which ideas from finite-dimensional monotone variational inequalities are readily applicable.
We note that their assumptions allow for both $\tau=0$ or $\tau>0$, since the (negative) differential entropy is displacement convex, and since their results are agnostic to the smoothness properties of the vector fields considered.

In the case $N=1$, the linear monotonicity condition displayed in \autoref{thm:intro:MFG} reduces to Lasry-Lions monotonicity, a classical condition in the mean-field game theory literature.
The long-time behavior of the mean-field game system (MFGS), a forward-backward PDE system describing the dynamics of games with an infinite number of indistinguishable players, was analyzed using this assumption by \cite{cardaliaguet2012longa,cardaliaguet2013longb,cardaliaguet2019longb,cirant2021longb}.
We emphasize, however, that our considered dynamics \eqref{eq:intro:Nspecies_flow} is structurally distinct from the MFGS: 
intuitively, \eqref{eq:intro:Nspecies_flow} corresponds to a generalization of gradient flow, whereas the MFGS (with a quadratic Hamiltonian) can be viewed as a coupled system of Hamiltonian flows \cite[Prop.~1 and Example~1]{chow2020wasserstein}.

It can be checked that the MFL-DA dynamics from the previous paragraph is an instance of linearly monotone multi-species flow with diffusion, with $N=2$.
Another example is given by the generalization of MFL-DA to $N$-player pairwise-zero-sum polymatrix continuous games, recently considered by \cite{lu2025convergence}, and which we discuss in \autoref{subsec:MFG:examples}.

\paragraph{Concurrent work.} This work is based on Chapter 2 of the first author's PhD thesis defended in December 2025. While finalizing this manuscript, we became aware of the preprint~\cite{seo2026local} which establishes closely related results using similar techniques. A detailed comparison will be included in a future revision.

\subsection{Notations} \label{subsec:intro:notations}

All integration symbols $\int$ are implicitly over $\RR^d$ unless specified.
We use $\norm{\cdot}$ to denote the $\ell^2$ norm of vectors in $\RR^d$ and
$\norm{\cdot}_{\mathrm{op}}$ for the corresponding operator norms of matrices and tensors over $\RR^d$.

The first variation of a functional $F: \PPP(\RR^d) \to \RR$ at $\mu$ is the function $F'[\mu]: \RR^d \to \RR$, unique up to an additive constant, if it exists, such that $F(\mu + \eps (\nu - \mu)) - F(\mu) = \eps \int F'[\mu] \d(\nu - \mu) + o(\eps)$ for any $\nu \in \PPP(\RR^d)$. The second variation $F''[\mu]: \RR^d \times \RR^d \to \RR$, as well as the $k$-th variation $F^{(k)}[\mu]: (\RR^d)^k \to \RR$ for $k \geq 3$, are defined similarly.
We may also write $\frac{\delta F(\mu)}{\delta \mu(x)}$ for $F'[\mu](x)$.

The pushforward of a measure $\mu$ by a mapping $T$ is the measure $T_\sharp \mu$ such that 
$\int \varphi \,\d(T_\sharp \mu) = \int (\varphi \circ T) \,\d\mu$ for any 
test 
function $\varphi$.
Equivalently, a random variable $X$ is distributed according to $\mu$ if and only if $T(X)$ is distributed according to $T_\sharp \mu$.


We denote by $\PPP_2(\RR^d)$ the set of probability measures on $\RR^d$ with finite second moments and by $W_2$ the $2$-Wasserstein distance,
that is, $W_2^2(\mu,\nu) = \min_\gamma \iint \norm{x-y}^2 \d\gamma(x,y)$ subject to $\gamma$ being a transport plan between $\mu$ and $\nu$. 
Moreover, we let $\PPP_2^{\AC}(\RR^d)$ be the subset of $\PPP_2(\RR^d)$ consisting of the measures which are absolutely continuous w.r.t.\ the Lebesgue measure. 

We say a functional $F: \PPP_2(\RR^d) \to \RR$ is $\beta$-displacement-smooth if $t \mapsto F(\mu_t)$ is $\beta$-smooth along any unit-speed Wasserstein geodesic $(\mu_t)_t$.
(A function $f: [0,1] \to \RR$ is called $\beta$-smooth if $| f(s) - f(t) - f'(t) (s-t) | \leq \frac \beta2 \abs{s-t}^2$ for all $s,t$.)

\paragraph{Function spaces.}
For any $\nu \in \PPP(\RR^d)$, we let 
$\LLL^2_\nu = \left\{ f: \RR^d \to \RR, \int \abs{f}^2 \d\nu <\infty \right\}$,
$\LLL^2_\nu \cap \{1\}^\perp = \left\{ f \in \LLL^2_\nu, \int f\, \d\nu = 0 \right\}$,
and $\bm\LLL^2_\nu = \left\{ \Phi: \RR^d \to \RR^d, \int \norm{\Phi}^2 \d\nu <\infty \right\}$.
We use $\innerprod{\cdot}{\cdot}_\nu$ to denote the inner product on all three of these Hilbert spaces.
Moreover, we denote by
$\CCC^\infty_c$ or $\CCC^\infty_c(\RR^d)$ (resp.\ $\CCC^\infty_c(\RR^d, \RR^d)$) the set of compactly supported $\CCC^\infty$-smooth real-valued (resp.\ vector-valued) functions,
and by $T_\nu \PPP_2(\RR^d) = \overline{\nabla \CCC^\infty_c}^{\,\bm\LLL^2_\nu}$ the subspace of $\bm\LLL^2_\nu$ consisting of gradient fields.

\paragraph{Functional inequalities.}
We write indifferently $\frac{\d\mu}{\d\nu}$ or $\frac{\mu}{\nu}$ for Radon-Nikodym derivatives. Expressions of the form $\nabla \log \mu$ should be interpreted as $\frac{\nabla \mu}{\mu}$. We may write $\log \mu$ instead of $\log \frac{\d\mu}{\d x}$ when it is clear from context that $\mu$ is absolutely continuous w.r.t.\ the Lebesgue measure.
We denote by $\KLdiv{\mu}{\nu} = \int \log \frac{\mu}{\nu} \,\d\mu$ the Kullback-Leibler divergence and by $\chisq{\mu}{\nu} = \int \left( \frac{\mu}{\nu}-1 \right)^2 \d\nu$ the $\chi^2$-divergence between two probability measures.

\begin{definition} \label{def:intro:LSI_T2_PI}
	We say a measure $\nu \in \PPP_2(\RR^d)$ satisfies the log-Sobolev inequality (LSI) with a constant $c_{\mathrm{LSI}}$ if
	\begin{equation*}
		\forall \mu \in \PPP_2(\RR^d),~ 
		\KLdiv{\mu}{\nu} \leq \frac{1}{2 c_{\mathrm{LSI}}} 
		\int \norm{\nabla \log \frac{\mu}{\nu}}^2 \d\mu.
	\end{equation*}
	We say $\nu$ satisfies the Talagrand inequality (T2) with a constant $c_{\mathrm{T2}}$ if
	\begin{equation*}
		\forall \mu \in \PPP_2(\RR^d),~ 
		\frac{c_{\mathrm{T2}}}{2} \, W_2^2(\mu, \nu) \leq \KLdiv{\mu}{\nu}.
	\end{equation*}
	Furthermore, we say $\nu$ satisfies Poincar\'e inequality (PI) with a constant $c_{\PI}$ if
	\begin{equation*}
		\forall f ~\text{s.t.} \int f \,\d\nu=0,~~ 
		\int \abs{f}^2 \d\nu \leq \frac{1}{c_{\mathrm{PI}}} \int \norm{\nabla f}^2 \d\nu.
	\end{equation*}
	It is classical that the optimal LSI, T2, resp.\ PI constants of any measure $\nu$ are ordered as
	$c_{\LSI} \leq c_{\mathrm{T2}} \leq c_{\PI}$ \cite{otto2000generalization}.
\end{definition}

\section{Heuristics and proof ideas} \label{sec:heuri_pfids}

In \autoref{subsec:heuri_pfids:heuri} we present an Otto calculus argument that heuristically suggests the result of \autoref{thm:intro:L2_MFLD}, 
and in \autoref{subsec:heuri_pfids:L2} we introduce the main ideas behind the theorem's proof, also used
in our analyses of MFL-DA and of multi-species flows.
The reader interested only in the results can safely skip ahead to the next section.

\subsection{Heuristics} \label{subsec:heuri_pfids:heuri}

Recall the following classical fact about gradient flows in finite dimension.
Note that it would be sufficient to assume $f$ is $\CCC^2$, but we state this weaker result for ease of comparison later on.

\begin{fact} \label{fact:findim}
    Consider a function $f \in \CCC^3(\RR^d)$ and a stationary point $x^*$ of its gradient flow, $\dot{x}_t = -\nabla f(x_t)$.
    Denote by 
    $\alpha_{\mathrm{loc}}$,
    $\alpha_{\mathrm{long}}$
    the rates of local resp.\ long-time exponential convergence to $x^*$,
    i.e., the largest constants such that
    \begin{itemize}
        \item For any $\eps>0$, there exist $C>0$ and a neighborhood $U$ of $x^*$ such that,
        for any initialization $x_0 \in U$,
        \begin{equation*}
            \forall t \geq 0,~ 
            x_t \in U
            ~~\text{and}~~
            \norm{x_t-x^*}^2 \leq C e^{-\alpha_{\mathrm{loc}} (1-\eps) t} \norm{x_0-x^*}^2.
        \end{equation*}
        %
        \item For any $\eps>0$ and any $x_0 \in \RR^d$ such that $\lim_{t \to \infty} x_t = x^*$, there exist $t_0, C>0$ such that
        \begin{equation*}
            \forall t \geq t_0,~
            \norm{x_t-x^*}^2 \leq C e^{-\alpha_{\mathrm{long}} (1-\eps) (t-t_0)}.
        \end{equation*}
    \end{itemize}
    If $\nabla^2 f(x^*)$ is positive-definite, then
    the local and the long-time rates coincide and are equal to $\alpha_{\mathrm{loc}} = \alpha_{\mathrm{long}} = 2 \sigma_{\min}(\nabla^2 f(x^*))$.%
\footnote{A more standard terminology would be to call ``local/long-time rates'' the quantities $\alpha_{\mathrm{loc}}/2$, $\alpha_{\mathrm{long}}/2$, i.e., to use the distance to $x^*$ as the error metric, instead of the squared distance. We chose to adjust the terminology for ease of comparison with our later results on convergence in $\chi^2$-divergence for MFLD.}
\end{fact}

\begin{proof}
	Fix $\eps>0$. Denote $M_3 \coloneqq \sup_{\norm{\xi - x^*} \leq 1} \norm{\nabla^3 f(\xi)}_{\mathrm{op}} < \infty$ since $f$ is $\CCC^3$. Let $r = \min\{1, \frac{2 \sigma \eps}{M_3}\}$ where $\sigma = \sigma_{\min}(\nabla^2 f(x^*))$.
	For any $x \in U \coloneqq B_{x^*,r}$ the ball centered at $x^*$ of radius $r$, we have $\nabla f(x) = \nabla f(x) - \nabla f(x^*) = \nabla^2 f(x^*) (x-x^*) + \frac12 \nabla^3 f(\xi) \left((x-x^*) \otimes (x-x^*)\right)$ for some $\xi \in [x^*, x] \subset B_{x^*, r} \subset B_{x^*,1}$, so that $\norm{\nabla^3 f(\xi)}_{\mathrm{op}} \leq M_3$. Thus for any $t$ such that $x_t \in U$,
	\begin{align*}
		\frac{d}{dt} \norm{x_t-x^*}^2 
		&= 2 \dot{x}_t^\top (x_t-x^*) 
		= -2 \nabla f(x_t)^\top (x_t-x^*) \\
		&\leq -2 (x_t-x^*)^\top \nabla^2 f(x^*) (x_t-x^*) + M_3 \norm{x_t-x^*}^3 \\
		&\leq -2 \Big( \sigma - \frac12 M_3 \norm{x_t-x^*} \Big) \norm{x_t-x^*}^2
		\leq -2 (1-\eps) \sigma \norm{x_t-x^*}^2,
	\end{align*}
	by definition of $U$.
	By Gr\"onwall's lemma, this proves a lower bound on the local convergence rate: $\alpha_{\mathrm{loc}} \geq 2 \sigma$.
	Here we can take $C=1$ for the constant prefactor.
	
    For the long-time convergence rate, fix $\eps>0$ and $x_0$ such that $\lim_{t \to \infty} x_t = x^*$. Let $C, U$ be as in the definition of $\alpha_{\mathrm{loc}}$. Since there exists $t_0$ such that $x_{t_0} \in U$, then by definition, $\alpha_{\mathrm{long}} \geq \alpha_{\mathrm{loc}}$.

    Thus $\alpha_{\mathrm{long}} \geq \alpha_{\mathrm{loc}} \geq 2\sigma$.
    Conversely, one can show that $\alpha_{\mathrm{long}} \leq 2 \sigma$ by choosing a unit eigenvector $v$ such that $\nabla^2 f(x^*) v = \sigma v$, an initialization $x_0 = x^* + \delta v$ with $0 \leq \delta \leq r = \min\{ 1, \frac{2\sigma \eps}{M_3} \}$, and
    lower-bounding
    \begin{align*}
    	\frac{d}{dt} v^\top (x_t-x^*)
    	= -v^\top \nabla^2 f(x_t)
    	&\geq -v^\top \nabla f(x^*) (x_t-x^*)
    	- \frac12 M_3 \norm{x_t-x^*}^2 \\
    	&\geq -\sigma v^\top (x_t-x^*)
    	- \frac12 M_3 \, \delta^2\, e^{-2\sigma (1-\eps) t}
   	\end{align*}
   	where we reused the results of the first paragraph.
   	By multiplying on both sides by $e^{\sigma t}$ and integrating, we obtain
   	\begin{align*}
   		\int_0^t \frac{d}{ds} \left( e^{\sigma s} v^\top (x_s-x^*) \right) \d s
   		&\geq -\int_0^t e^{\sigma s} \frac12 M_3 \, \delta^2 \, e^{-2 \sigma (1-\eps) s} \,\d s \\
   		v^\top (x_t - x^*)
   		&\geq e^{-\sigma t} \left( \delta - \frac12 M_3 \, \delta^2 \, \frac{1}{\sigma (1-2\eps)} \right),
   	\end{align*}
   	and the constant on the right-hand side is positive for $\delta$ small enough.
   	Then $\norm{x_t-x^*}^2 \geq \abs{v^\top (x_t-x^*)}^2 \geq \left( \delta - \frac12 M_3 \, \delta^2 \, \frac{1}{\sigma (1-2\eps)} \right)^2 e^{-2\sigma t}$,
   	hence $\alpha_{\mathrm{long}} \leq 2\sigma$ and so $\alpha_{\mathrm{long}} = \alpha_{\mathrm{loc}} = 2\sigma$.
\end{proof}

Since MFLD is the WGF of $F_\tau$,
it is natural to expect that its local and long-time convergence rates are analogously given by the Hessian of $F_\tau$, in the Wasserstein sense, at the stationary measure.
Otto calculus indeed provides such a notion of Wasserstein Hessian \cite[Chapter~15]{villani2008optimal}:
formally,
for a functional $\GGG: \PPP_2(\RR^d) \to \RR$ and $\mu \in \PPP_2^{\AC}(\RR^d)$, $\Hess_\mu \GGG$ is the symmetric bilinear operator over $T_\mu \PPP_2(\RR^d) = \overline{\nabla \CCC^\infty_c}^{\,\bm\LLL^2_\mu}$ such that $\forall \Phi, \Hess_\mu \GGG(\Phi, \Phi) = \restr{\frac{d^2}{ds^2}}{s=0} \GGG((\id+s\Phi)_\sharp \mu)$.
It is known that for displacement-smooth functionals~$F$ \cite[Proposition~19]{li2022transport}, 
\begin{equation*}
	\forall \mu, \forall \Phi,~
    \Hess_\mu F(\Phi,\Phi)
	=
	\iint F''[\mu](x,x') \,\d\left[ \nabla \cdot (\mu \Phi) \right](x) \,\d\left[ \nabla \cdot (\mu \Phi) \right](x')
	+ \int \Phi^\top \nabla^2 F'[\mu] \, \Phi \,\d\mu,
\end{equation*}
and that for differential entropy and for KL-divergence $\KLdiv{\cdot}{\nu}$ \cite{otto2000generalization},%
\footnote{To be clear, the formulas and computations appearing in this subsection should be understood as non-rigorous heuristics, as 
the expression $\int \trace((\nabla \Phi)^2) \,\d\mu$ is
not even well-defined for all $\Phi \in T_\mu \PPP_2(\RR^d)$.}
\begin{align*}
    \forall \mu, \forall \Phi,~ \quad
    \Hess_\mu H(\Phi,\Phi)
    &= \int \trace((\nabla \Phi)^2) \,\d\mu \\
	\Hess_\mu \KLdiv{\cdot}{\nu}(\Phi,\Phi)
	&= \int \trace((\nabla \Phi)^2) \,\d\mu + \int \Phi^\top (-\nabla^2 \log \nu) \, \Phi \,\d\mu.
\end{align*}

Now for MFLD, the first-order stationarity condition 
$\nabla F_\tau'[\nu] = 0$
rewrites 
$-\tau \log \nu = F'[\nu] + \cst$ 
(on $\support(\nu) = \RR^d$ \cite[Proposition~4.6]{chen2022uniform}). 
Thus the Hessian of $F_\tau = F + \tau H$ at the stationary measure $\nu$ is given by
\begin{align*}
	\forall \Phi \in T_\nu \PPP_2(\RR^d),~
    & \Hess_\nu F_\tau(\Phi,\Phi) 
	= \Hess_\nu F(\Phi,\Phi) 
	+ \tau \int \trace((\nabla \Phi)^2) \,\d\nu \\
	&~\qquad = \iint F''[\nu](x,x') \,\d\left[ \nabla \cdot (\nu \Phi) \right](x) \,\d\left[ \nabla \cdot (\nu \Phi) \right](x') \\
	&~\qquad \qquad
    + \int \Phi^\top \nabla^2 F'[\nu] \Phi \,\d\nu
    + \tau \int \trace((\nabla \Phi)^2) \,\d\nu \\
	&~\qquad = \iint F''[\nu](x,x') \,\d\left[ \nabla \cdot (\nu \Phi) \right](x) \,\d\left[ \nabla \cdot (\nu \Phi) \right](x')
	+ \tau \Hess_\nu \KLdiv{\cdot}{\nu}(\Phi,\Phi).
\end{align*}
If $F$ is linearly convex, then $\iint F''[\nu] \, \d(s \otimes s) \geq 0$ for any $s \in \MMM(\RR^d)$ with $\int \d s=0$, and so the first term on the last line is non-negative.
If additionally $\nu$ satisfies PI with some constant $c_{\PI}$, then we show in \autoref{sec:apx_OL} that
\begin{equation} \label{eq:heuri_pfids:heuri:Hess_ineqs}
	\forall \Phi \in T_\nu \PPP_2(\RR^d),~
	\Hess_\nu F_\tau(\Phi,\Phi) 
	\geq \tau \Hess_\nu \KLdiv{\cdot}{\nu}(\Phi,\Phi)
	\geq \tau \, c_{\PI} \int \norm{\Phi}^2 \d\nu.
\end{equation}
So heuristically one can expect MFLD to converge locally at a rate at least $2 \tau c_{\PI}$.

In order to formalize the above heuristic argument, a tempting approach is to mimick the proof of \autoref{fact:findim}: ``linearize'' the dynamics around $\nu$ and control the error.
However, 
WGFs generally do not satisfy the correct analog of the bound on $\nabla^3 f$,
and in the case of MFLD we do not even have the analog of $f$ being $\CCC^2$ because $H$ is not displacement-smooth.
Our proofs are instead inspired by the $\LLL^2$ convergence analysis of the overdamped Langevin dynamics.

\subsection{Proof idea: tracking the \texorpdfstring{$H_\nu^{-1}$}{H-1} norm} \label{subsec:heuri_pfids:L2}

Let us now explain the ideas behind our convergence analysis of MFLD in $\chi^2$-divergence, leading up to \autoref{thm:intro:L2_MFLD}.
Denote by $(\mu_t)_t$ the MFLD for a linearly convex and displacement-smooth functional $F$, and by $\nu$ the equilibrium measure, characterized by the stationarity condition $F'[\nu] + \tau \log \nu = \cst$.
Note that the MFLD PDE can be written as
\begin{equation*}
	\partial_t \mu_t 
	= \nabla \cdot (\mu_t \nabla F'[\mu_t]) + \tau \Delta \mu_t
	= \nabla \cdot \left( \mu_t \left( \tau \nabla \log \frac{\mu_t}{\nu} + \nabla F'[\mu_t] - \nabla F'[\nu] \right) \right).
\end{equation*}

To explain our proof strategy for the local $\chi^2$ convergence, first consider the following computation, which shows that local $\chi^2$ \emph{contraction} is unlikely unless a high temperature is assumed.
Denoting $\rho_t = \frac{\d\mu_t}{\d\nu}$, we have that for any $h: \RR_+ \to \RR$,
\begin{equation*}
	\frac{d}{dt} \int h(\rho_t) \, \d\nu 
	= \int h'(\rho_t) \, \d (\partial_t \mu_t)
	= -\int \nabla h'(\rho_t)^\top \left( \tau \nabla \log \frac{\mu_t}{\nu} + \nabla F'[\mu_t] - \nabla F'[\nu] \right) \d\mu_t
\end{equation*}
and so for $h(s) = (s-1)^2$,
\begin{align*}
	\frac{d}{dt} \chisq{\mu_t}{\nu} 
	&= -2 \tau \int \nabla \rho_t^\top \nabla \log \frac{\mu_t}{\nu} \,\d\mu_t
	- 2 \int \nabla \rho_t^\top \left( \nabla F'[\mu_t] - \nabla F'[\nu] \right) \d\mu_t \\
	&= -2 \tau \int \norm{\nabla \rho_t}^2 \d\nu
	- 2 \int \nabla \rho_t^\top \left( \nabla F'[\mu_t] - \nabla F'[\nu] \right) \rho_t \,\d\nu.
\end{align*}
The first term can be upper-bounded by $-2 \tau c_{\PI} \chisq{\mu_t}{\nu}$
by applying the definition of PI to $f = \rho_t-1$.
For the second term, 
if we additionally assume that $F$ has Lipschitz-continuous Wasserstein gradients, then
$\sup_{\RR^d} \norm{\nabla F'[\mu_t] - \nabla F'[\nu]} \leq \beta W_2(\mu_t, \nu)$
for some $\beta<\infty$,
and so 
\begin{align*}
    \abs{2 \int \nabla \rho_t^\top \left( \nabla F'[\mu_t] - \nabla F'[\nu] \right) \rho_t \,\d\nu} 
    &\leq
	2 \beta W_2(\mu_t, \nu) \int \norm{\nabla \rho_t} \rho_t \d\nu \\
	&\leq 2 \beta W_2(\mu_t, \nu) \sqrt{\int \norm{\nabla \rho_t}^2 \d\nu} 
	\underbrace{~\sqrt{\int \rho_t^2 \,\d\nu}~}_{\asymp~ 1+ \sqrt{\chisq{\mu_t}{\nu}}}.
\end{align*}
This bound on the second term is morally of order~2 in the distance between $\mu_t$ and $\nu$, so it fails to be locally negligible compared to the first term.
Thus, intuitively, the second term should instead be further expanded, so as to exploit the linear convexity of $F$. It turns out to be most convenient for this purpose to work directly in $\LLL^2_\nu$.

The case where $F$ is a quadratic functional, i.e., of the form \eqref{eq:intro:quadr_F}, essentially encapsulates all of the difficulty, so we focus on this case for the rest of the section. By the first-order stationarity condition, $F$ can be rewritten as
\begin{align*}
	F(\mu) &= F(\nu) + \int (-\tau \log \nu) \,\d(\mu-\nu) + \frac12 \iint k(x,y) \,\d(\mu-\nu)(x) \d(\mu-\nu)(y) \\
	\nabla F'[\mu](x) &= -\tau \nabla \log \nu + \int \nabla_x k(x, y) \,\d(\mu-\nu)(y).
\end{align*}
The MFLD can be re-expressed as a dynamics over $\LLL^2_\nu$ by posing $f_t = \frac{\d\mu_t}{\d\nu}-1$, as
\begin{align}
	\partial_t \mu_t &= \nabla \cdot \left( \mu_t \left( \tau \nabla \log \frac{\mu_t}{\nu} \right) \right) 
	+ \nabla \cdot \left( \mu_t \left( \nabla F'[\mu_t] - \nabla F'[\nu] \right) \right)   \nonumber \\
    &= \tau\, \nabla \cdot \left( \nu \nabla \frac{\mu_t}{\nu} \right)
    + \nabla \cdot \left( \mu_t \left( \int \nabla_x k(\cdot, y) \,\d(\mu_t-\nu)(y) \right) \right)   \nonumber \\
	\partial_t f_t &= - \tau L f_t + \frac1\nu \nabla \cdot (\nu (f_t+1) \nabla K f_t)   \nonumber \\
	&= -\tau L f_t - L K f_t + \frac1\nu \nabla \cdot (\nu f_t \nabla K f_t),
\label{eq:heuri_pfids:L2:MFLD_PDE_L2}
\end{align}
where $K, L$ are the operators
on $\LLL^2_\nu$
defined by $K f = \int k(\cdot, y) f(y) \d\nu(y)$ and $L f = -\frac1\nu \nabla \cdot (\nu \nabla f)$.%
\footnote{$L$ is precisely the generator of the overdamped Langevin dynamics associated to $\nu$. It is unbounded as an operator over $\LLL^2_\nu$, but in this section we do not discuss the associated technicalities explicitly, for ease of presentation.}
Note that $K$ and $L$ are both symmetric and positive-semi-definite, by linear convexity of $F$.
Treating the last term in the expression of $\partial_t f_t$ as an error term since it is of order~2 in $f_t$, we have
\begin{equation} \label{eq:heuri_pfids:L2:poorbound}
	\frac{d}{dt} \chisq{\mu_t}{\nu}
	= \frac{d}{dt} \norm{f_t}_\nu^2
	= 2 \innerprod{f_t}{\partial f_t}_\nu
	= -2 \tau \innerprod{f_t}{L f_t}_\nu - 2 \innerprod{f_t}{L K f_t}_\nu + \text{[error term of order~3]}.
\end{equation}
The term $-2 \tau \innerprod{f_t}{L f_t}_\nu$ can be upper-bounded thanks to PI, however it is not clear whether the term $-2 \innerprod{f_t}{L K f_t}_\nu$ has a sign.

Our main insight
is that the prefactor $L$ is cancelled out when we track, instead of $\norm{f_t}_\nu^2$, the $H^{-1}_\nu$ norm defined as
\begin{equation*}
	\forall f \in \LLL^2_\nu ~\text{s.t.}~ \int f \,\d\nu = 0,~
	\norm{f}_{H^{-1}_\nu}^2 = \innerprod{f}{L^{-1} f}_\nu,
\end{equation*}
which is well-defined since the PI for $\nu$ implies that $L$ is invertible on $\LLL^2_\nu \cap \{1\}^\perp$.
Indeed, we get
\begin{equation} \label{eq:heuri_pfids:L2:ddt_H-1}
	\frac{d}{dt} \norm{f_t}_{H^{-1}_\nu}^2
	= 2 \innerprod{L^{-1} f_t}{\partial f_t}_\nu
	= -2 \tau \innerprod{f_t}{f_t}_\nu - 2 \innerprod{f_t}{K f_t}_\nu + \text{[error term of order~3]}.
\end{equation}
Now $-2 \innerprod{f_t}{K f_t}_\nu \leq 0$ by linear convexity of $F$, and for the first term in $\tau$, we can still upper-bound it appropriately thanks to PI, as
\begin{equation*}
	\text{$\nu$ satisfies PI} ~~\iff~~ L \succeq c_{\PI} \id ~~\text{in $\LLL^2_\nu \cap \{1\}^\perp$} ~~\iff~~ L^{-1} \preceq c_{\PI}^{-1} \id ~~\text{in $\LLL^2_\nu \cap \{1\}^\perp$}
\end{equation*}
so that $\tau \innerprod{f_t}{f_t}_\nu \geq \tau c_{\PI} \innerprod{f_t}{L^{-1} f_t}_\nu = \tau c_{\PI} \norm{f_t}_{H_\nu^{-1}}^2$.
Hence we have
\begin{equation*}
	\frac{d}{dt} \norm{f_t}_{H^{-1}_\nu}^2
	\leq -2 \tau \, c_{\PI} \left( 1 - \text{[error term of order~1]} \right) \norm{f_t}_{H^{-1}_\nu}^2,
\end{equation*}
from which we can conclude to local contraction in  $\norm{\cdot}_{H^{-1}_\nu}^2$ by Gr\"onwall's lemma.
We can then deduce the local convergence in $\chi^2$ by restarting from \eqref{eq:heuri_pfids:L2:poorbound},
via an estimate $\abs{\innerprod{f}{LKf}_\nu} \lesssim \norm{\nabla f}_\nu \norm{f}_{H^{-1}_\nu}$ which holds under an appropriate regularity assumption on $k$ (\autoref{lm:MFLD_L2:localL2:estim_nablaKf_infty}).

Note that the linear convexity of $F$ only came into play to bound the second term on the right-hand side of \eqref{eq:heuri_pfids:L2:ddt_H-1}.
Accordingly, the linear convexity assumption can be relaxed to the condition
\begin{equation*}
    \forall f \in \LLL^2_\nu \cap \{1\}^\perp,~~
    -\innerprod{f}{K f}_\nu \leq \tau_0 \innerprod{f}{f}_\nu
\end{equation*}
for some $\tau_0<\tau$, which corresponds to a local form of relative weak convexity of $F$ w.r.t.\ $H$, as we discuss
below the statement of \autoref{assum:MFLD_A}.

\section[Local \texorpdfstring{$\LLL^2$}{L2} convergence of MFLD]{Local \texorpdfstring{$\LLL^2$}{L2} convergence of mean-field Langevin dynamics} \label{sec:MFLD_L2}

Throughout this section, we consider a functional $F$ and an equilibrium measure $\nu$ satisfying the following assumption.
Moreover, for the rest of this section, we use $(\mu_t)_t$ to denote the MFLD \eqref{eq:intro:MFLD}, i.e., the WGF of $F_\tau = F+\tau H$, and we set $f_t = \frac{\d\mu_t}{\d\nu} - 1$.
We recall that $\chisq{\mu_t}{\nu} = \norm{f_t}_\nu^2$, and for this reason we will also refer to convergence in $\chisq{\cdot}{\nu}$ as $\LLL^2$ convergence.

\begin{assumption} \label{assum:MFLD_A}
	The functional $F: \PPP_2(\RR^d) \to \RR$ is displacement-smooth and there exists a stationary measure $\nu$ for the MFLD, i.e., such that $F'[\nu] + \tau \log \nu = \cst$ on $\RR^d$. 
	Moreover, $F''[\nu]$ the second variation of $F$ at $\nu$ satisfies
    \begin{equation} \label{eq:MFLD_L2:assum_F''nu}
        \forall f \in \LLL^2_\nu \cap \{1\}^\perp,~~
        \iint F''[\nu](x,x') \, f(x) f(x') \,\d\nu(x) \d\nu(x')
        \geq -\tau_0 \int \abs{f}^2 \d \nu
    \end{equation}
    for some $0 \leq \tau_0 < \tau$.
	Furthermore, $\nu$ satisfies PI with a constant $c_{\PI}$, i.e.,
    $\int \norm{\nabla f}^2 \d\nu \geq c_{\PI} \int \abs{f}^2 \d\nu$ for all $f \in \LLL^2_\nu \cap \{1\}^\perp$.
\end{assumption}

To elucidate the assumption on $F''[\nu]$, let us mention two sufficient conditions for it.
In the simplified statement of our result in the introduction, \autoref{thm:intro:L2_MFLD}, we required global linear convexity of $F$, which corresponds to the second item with $\tau_0=0$.
\begin{itemize}
    \item When $\tau_0 = 0$, \eqref{eq:MFLD_L2:assum_F''nu} is equivalent to the function $(x,x') \mapsto F''[\nu](x,x')$ being a conditionally positive-semi-definite kernel, meaning that
    $\iint F''[\nu](x,x') \,\d s(x) \d s(x') \geq 0$ for all $s \in \MMM(\RR^d)$ such that $\int \d s = 0$.%
    \footnote{A sufficient condition for $F''[\nu]$ to be a conditionally positive-semi-definite kernel is if $\nu$ is a local minimizer of $F$ in the measure-space sense.
    However, in general, there is no reason for a stationary point of MFLD, which is the WGF of $F+\tau H$, to be a local minimizer for $F$ only.}
    \item
    The left-hand side of \eqref{eq:MFLD_L2:assum_F''nu} is precisely the measure-space Hessian of $F$ at $\nu$ applied to the tangent direction $f \nu$, i.e.,
    it is equal to $\restr{\frac{d^2}{d\theta^2}}{\theta=0} F(\nu + \theta f \nu)$.
    Likewise, the right-hand side is precisely $-\tau_0$ times the measure-space Hessian of $H$ at $\nu$ applied to the direction $f \nu$, as one can show by explicit computations.
    Thus \eqref{eq:MFLD_L2:assum_F''nu} is equivalent to asking that
    \begin{equation*}
        \forall s \in \MMM(\RR^d) ~\text{s.t.} \int \d s = 0 ~\text{and}~ \frac{\d s}{\d\nu} \in \LLL^2_\nu,~~
        \restr{\frac{d^2}{d\theta^2}}{\theta=0} (F+\tau_0 H)(\nu + \theta s) \geq 0.
    \end{equation*}
    This can be interpreted as a local form of linear convexity of $F+\tau_0 H$ ``at $\nu$''.
    In particular,
    it holds if $F+\tau_0 H$ is globally linearly convex.
\end{itemize}

\subsection{Local \texorpdfstring{$\LLL^2$}{L2} convergence} \label{subsec:MFLD_L2:localL2}

The main result of this section is the following. 
See also \autoref{thm:MFLD_L2:localL2:gen_full} for a more precise version.
\begin{theorem} \label{thm:MFLD_L2:localL2:gen_simple}
	Under \autoref{assum:MFLD_A}, additionally assume that
	\begin{equation*}
		M_{11} = \! \sup_{x,y \in \RR^d} \norm{\nabla_x \nabla_y F''[\nu]}_{\mathrm{op}},
		~~
		M_{111} = \!\!\! \sup_{\substack{\tmu \in \PPP_2(\RR^d)\\x,y,z \in \RR^d}} \!\! \norm{\nabla_x \nabla_y \nabla_z F'''[\tmu]}_{\mathrm{op}},
		~~
		M_{12} = \! \sup_{\substack{\tmu \in \PPP_2(\RR^d)\\x,y \in \RR^d}} \! \norm{\nabla_x \nabla_y^2 F''[\tmu]}_{\mathrm{op}}
	\end{equation*}
	are finite.
	Then for any $\eps>0$ small enough,
	there exist
	$r^{-1}, C = \mathrm{poly}\big( \eps^{-1}, (\tau-\tau_0)^{-1}, c_{\PI}^{-1}, M_{11}, \allowbreak M_{111}, \allowbreak M_{12} \big)$
	such that
	if $\chisq{\mu_0}{\nu} \leq r$ then
	\begin{equation*}
		\forall t \geq 0,~
		\chisq{\mu_t}{\nu} \leq C e^{-2 (\tau-\tau_0) c_{\PI} (1-\eps) t} \chisq{\mu_0}{\nu}.
	\end{equation*}
\end{theorem}

This theorem shows local convergence in $\chi^2$-divergence, but not directly via a local contraction, since we explained in \autoref{subsec:heuri_pfids:L2} that that is unlikely to hold except for large $\tau$.
Rather, we show local contraction in a mixed $H^{-1}_\nu$ - $\LLL^2_\nu$ metric as displayed in \eqref{eq:MFLD_L2:localL2:genthm_Wt} below.

The case where $F$ is a quadratic functional allows for more explicit constants, so we present this result separately.

\begin{theorem} \label{thm:MFLD_L2:localL2:quad}
	Under \autoref{assum:MFLD_A}, additionally suppose $F$ is quadratic, i.e., of the form \eqref{eq:intro:quadr_F}. In particular $F''[\mu] = k$ for all $\mu$, and $F$ can be rewritten as
	\begin{equation*}
		\forall \mu,~
		F(\mu) = F(\nu) + \int (-\tau \log \nu) \,\d(\mu-\nu) + \frac12 \iint k(x,x') \,\d(\mu-\nu)(x) \d(\mu-\nu)(x').
	\end{equation*}
	Further assume $M_{11} \coloneqq \sup_{x,x' \in \RR^d} \norm{\nabla_x \nabla_{x'} k}_{\mathrm{op}} < \infty$.
	Then for any $0<\eps < 1$,
	if $\chisq{\mu_0}{\nu} \leq \frac{(\tau-\tau_0)^2 c_{\PI}^2 \eps^2}{4 M_{11}^2}$ then
	\begin{equation*}
		\forall t \geq 0,~
		\chisq{\mu_t}{\nu} \leq
		\left( 1 + \frac{M_{11}^2}{\tau (\tau-\tau_0) c_{\PI}^2 \eps^2} \right)
		e^{-2 (\tau-\tau_0) c_{\PI} (1-\eps) t} \chisq{\mu_t}{\nu}.
	\end{equation*}
\end{theorem}

\begin{remark}
	In the case of a quadratic functional $F$, as a by-product of the proof, we also have local contraction in $H^{-1}_\nu$ norm.
	Specifically we have that, in the setting of \autoref{thm:MFLD_L2:localL2:quad}, for any $0<\eps < 1$,
	if $\norm{f_0}_{H^{-1}_\nu}^2 \leq \frac{(\tau-\tau_0)^2 c_{\PI} \, \eps^2}{M_{11}^2}$ then
	$\forall t \geq 0, \norm{f_t}_{H^{-1}_\nu}^2 \leq e^{-2 (\tau-\tau_0) c_{\PI} (1-\eps) t} \norm{f_0}_{H^{-1}_\nu}^2$.
\end{remark}

\begin{remark}
    For functionals $F$ such that
    $\tau_c = \inf \left\{ \tau' \geq 0 ; F + \tau' H ~\text{is globally linearly convex} \right\} < \infty$,
    our local exponential convergence rate estimate of $2 (\tau-\tau_c) c_{\PI}$
    complements the recent global convergence bounds of \cite{chizat2025convergence},
    whose exponential rate also scaled linearly with $\tau-\tau_c$ (for MFLDs over a torus).
\end{remark}

The proofs of \autoref{thm:MFLD_L2:localL2:gen_simple} and \autoref{thm:MFLD_L2:localL2:quad} are delayed to \autoref{subsec:MFLD_L2:localL2_proof}.
Let us now comment on three aspects of these results, to help delineate the directions in which they can or cannot be improved.

\subsubsection{Tightness of the rate for overdamped Langevin dynamics}

For linear functionals $F$, i.e., of the form $F(\mu) = \int V\, \d\mu$ for some $V: \RR^d \to \RR$, MFLD reduces to the overdamped Langevin dynamics associated to $\nu \propto e^{-V(x)/\tau} \d x$ (with time rescaled by $\tau$):
\begin{equation} \label{eq:MFLD_L2:localL2:OL_PDE}
    \partial_t \mu_t = \nabla \cdot (\mu_t \nabla V) + \tau \Delta \mu_t.
\end{equation}
It is well-known that,
under mild regularity assumptions on $V$, the long-time convergence rate of this dynamics is precisely given by the PI constant of $\nu$,
as the following proposition formalizes.

\begin{proposition} \label{prop:MFLD_L2:localL2:OL}
    Let $\nu \propto e^{-V(x)/\tau} \d x$ be a probability measure that satisfies PI and denote by $c_{\PI}$ its optimal PI constant. 
    Suppose that the embedding $H^1_\nu \subset \LLL^2_\nu$ is compact, where $H^1_\nu$ is the Sobolev space $\left\{ f \in \LLL^2_\nu ; \int \norm{\nabla f}^2 \d\nu < \infty \right\}$.
    Then for any $\mu_0 \in \PPP_2(\RR^d)$ such that $\frac{\d\mu_0}{\d\nu} \in \LLL^2_\nu$, the solution $(\mu_t)_t$ of \eqref{eq:MFLD_L2:localL2:OL_PDE} satisfies 
    $
        \forall t \geq 0,~ \chisq{\mu_t}{\nu} \leq e^{-2 \tau c_{\PI} t} \chisq{\mu_0}{\nu}
    $.
    Conversely, there exists $\mu_0 \in \PPP_2(\RR^d)$ with $\frac{\d\mu_0}{\d\nu} \in \LLL^2_\nu$ such that this inequality holds with equality for all $t$.
\end{proposition}

\begin{proof}
    Consider the unbounded symmetric operator $L = -\frac1\nu \nabla \cdot (\nu \nabla \bullet)$ over $\LLL^2_\nu$.
    The PI and compact embedding assumptions ensure that 
    $\LLL^2_\nu$ has an orthonormal basis $(g_0=1, g_1, g_2, ...)$ of eigenfunctions of $L$ with associated eigenvalues $\alpha_0 = 0 < \alpha_1 = c_{\PI} \leq \alpha_2 \leq ...$
    \cite[Remark~1.1]{cao2023explicit}. 
    Moreover \eqref{eq:MFLD_L2:localL2:OL_PDE} is equivalent to $\partial_t f_t = -\tau L f_t$ with the change of variables $f_t = \frac{\mu_t}{\nu} - 1$.
    The first part of the proposition follows by noting that
    $\frac{d}{dt} \chisq{\mu_t}{\nu} = \frac{d}{dt} \norm{f_t}_\nu^2 = -2 \tau \innerprod{f_t}{L f_t}_\nu \leq 2 \tau c_{\PI} \norm{f_t}_\nu^2$.
    The second part of the proposition follows by the same computation applied to $f_0 = g_1$, 
	since $f_t$ is then proportional to $g_1$ for all $t$.
\end{proof}

Thus, our estimate of the local convergence rate of MFLD is tight in the special case of the overdamped Langevin dynamics.
Note that in this case, the assumption \eqref{eq:MFLD_L2:assum_F''nu} holds with $\tau_0 = 0$.
We leave open the question of determining examples
of $F$ such that \eqref{eq:MFLD_L2:assum_F''nu} holds only for some $\tau_0>0$ and for which our rate estimate is tight.

\subsubsection{Range of \texorpdfstring{$\tau_0$}{tau0} and near-necessity of (\ref{eq:MFLD_L2:assum_F''nu})} \label{subsubsec:MFLD_localL2:range_of_tau0}

Our assumption that \eqref{eq:MFLD_L2:assum_F''nu} holds for some $\tau_0<\tau$ is ``nearly'' necessary for a convergence guarantee such as ours to exist,
in the sense that \eqref{eq:MFLD_L2:assum_F''nu} with $\tau_0 = \tau$ is necessary for local stability of MFLD.
This is the content of the next proposition, 
whose proof is delayed to \autoref{subsec:MFLD_L2:delayed_pfs}.

\begin{proposition} \label{prop:MFLD_L2:localL2:tau0=tau_unstable}
    Consider $F: \PPP_2(\RR^d) \to \RR$ displacement-smooth, $\tau>0$, and $\nu$ a stationary measure of the associated MFLD assumed to satisfy PI.
    Further assume that the quantities $M_{11}$ and $M_{111}$ in the statement of \autoref{thm:MFLD_L2:localL2:gen_simple} are finite.
    Suppose that there exists $f \in \LLL^2_\nu \cap \{1\}^\perp$ such that
    \begin{equation*}
        \iint F''[\nu](x,x') \, f(x) f(x') \,\d\nu(x) \d\nu(x')
        < -\tau \int \abs{f}^2 \d \nu.
    \end{equation*}
    Then for any $r>0$, there exists $\mu_0 \in \PPP_2(\RR^d)$ with $\chisq{\mu_0}{\nu} \leq r$ such that
    the MFLD $(\mu_t)_t$ initialized at $\mu_0$ does not converge to $\nu$ in Wasserstein distance.
\end{proposition}

We suspect that local convergence bounds with an algebraic rate can still be obtained assuming only that \eqref{eq:MFLD_L2:assum_F''nu} holds with $\tau_0 = \tau$, 
similar to \cite[part~2 of Prop.~1]{monmarche2025local} or \cite[part~2 of Thm.~1.5]{chizat2025convergence},
but that is out of the scope of this paper.

On the other hand, the fact that we only consider non-negative values for $\tau_0$ in \autoref{assum:MFLD_A} is simply due to the fact that \eqref{eq:MFLD_L2:assum_F''nu} provably cannot hold for $\tau_0 < 0$ as soon as $F''[\nu]$ is regular enough. We show this fact in the next proposition, whose proof is also delayed to \autoref{subsec:MFLD_L2:delayed_pfs}.
A similar result appeared in \cite[Lemma~39]{duncan2023geometry}. 

\begin{proposition} \label{prop:MFLD_L2:localL2:tau0>=0}
    Consider any $\nu \in \PPP_2(\RR^d)$ satisfying PI and any symmetric $k: \RR^d \times \RR^d \to \RR$ such that $\iint \norm{\nabla_x \nabla_{x'} k(x,x')}_{\mathrm{op}}^2 \allowbreak \d\nu(x) \d\nu(x') < \infty$.
    Then the condition
    \begin{equation*}
    	\forall f \in \LLL^2_\nu \cap \{1\}^\perp,~~
    	\iint k(x,x') \, f(x) f(x') \,\d\nu(x) \d\nu(x')
    	\geq \alpha \int \abs{f}^2 \d \nu
    \end{equation*}
    can hold only for $\alpha \leq 0$.
\end{proposition}

\subsubsection{Conjectured spectral characterization of the exact rate} \label{subsubsec:MFLD_L2:localL2:exact_rate}

\autoref{thm:MFLD_L2:localL2:gen_simple} only implies a lower bound on the local convergence rate of MFLD.
Given the $\LLL^2$ formulation of the dynamics (see \eqref{eq:heuri_pfids:L2:MFLD_PDE_L2} for the quadratic case or \eqref{eq:MFLD_L2:localL2:MFLD_PDE_L2} below for the general case), we intuitively expect the exact local rate to be equal to two times
\begin{equation*}
	\sigma = \min \{ \Re(\lambda); \lambda \in \spectrum(\tau L + LK) \}
\end{equation*}
where $\spectrum(\cdot)$ denotes the spectrum and
$K, L$ are the operators over $\LLL^2_\nu$ given by
$K f = \int F''[\nu](\cdot, y) \allowbreak f(y) \allowbreak \d\nu(y)$ and $L f = -\frac1\nu \nabla \cdot (\nu \nabla f)$.
However, proving this rate formally would be technically involved, because $L$ is an unbounded operator, so we leave this direction for future work.

Still at the intuitive level,
a slightly friendlier expression for our conjectured exact rate can be obtained by expressing the dynamics in terms of $\Phi_t = \nabla L^{-1} f_t$. 
Indeed denoting $\nabla^* = -\frac1\nu \nabla \cdot (\nu \bullet)$, we have $L = \nabla^* \nabla$ and $f_t = \nabla^* \Phi_t$, so%
\footnote{This intuition could also be formalized by arguing that $\nabla^*$ is ``essentially invertible'' as an operator from $\overline{\nabla \CCC^\infty_c(\RR^d)}^{\,\bm\LLL^2_\nu}$ 
to $\LLL^2_\nu \cap \{1\}^\perp$ by the PI for $\nu$,
and that $\tau \nabla \nabla^* + \nabla K \nabla^*$ is similar to $\tau L + LK = \tau \nabla^* \nabla + \nabla^* \nabla K$ via conjugation by $\nabla^*$.}
\begin{align*}
	\partial_t f_t &\approx -\tau L f_t - L K f_t \\
	\text{becomes}~~~~
	\partial_t \Phi_t &\approx -\tau \nabla \nabla^* \Phi_t - \nabla K \nabla^* \Phi_t.
\end{align*}
So the quantity $\sigma$ from the previous paragraph is also equal to the smallest eigenvalue of 
$\tau \nabla \nabla^* + \nabla K \nabla^*$, viewed as a symmetric operator over $\overline{\nabla \CCC^\infty_c}^{\LLL^2_\nu}$, the subspace of $\bm\LLL^2_\nu$ consisting of gradient fields. (Viewed as an operator over $\bm\LLL^2_\nu$, this operator would have a non-trivial kernel, although it would have the same non-zero eigenvalues.)

In other words, $\sigma$ is precisely the largest constant $\sigma'$ such that
\begin{multline*}
	\forall \Phi \in \overline{\nabla \CCC^\infty_c}^{\LLL^2_\nu},~~
	\innerprod{\Phi}{(\tau \nabla \nabla^* + \nabla K \nabla^*) \Phi}_\nu 
	= \innerprod{\Phi}{\nabla K \nabla^* \Phi}_\nu
	+ \tau \norm{\nabla^* \Phi}_\nu^2 \\
	= \iint \Phi(x)^\top \nabla_x \nabla_{x'} F''[\nu](x,x') \, \Phi(x') \, \d\nu(x) \d\nu(x')
	+ \tau \int \abs{\frac1\nu \nabla \cdot (\nu \Phi)}^2 \d\nu 
	\geq \sigma' \int \norm{\Phi}^2 \d\nu.
\end{multline*}
As we show in \autoref{sec:apx_OL}, the second term on the left-hand side is precisely equal to $\tau \Hess_\nu \KLdiv{\cdot}{\nu}(\Phi, \Phi)$, so the left-hand side is equal to $\Hess_\nu F_\tau(\Phi, \Phi)$ as defined in \autoref{subsec:heuri_pfids:heuri}.
So we can also interpret $\sigma$ as the smallest eigenvalue of the Wasserstein Hessian of $F_\tau$ at the stationary measure, in line with the intuition from \autoref{fact:findim}.

As a potential application of this conjectured rate, note that if $M_{11} = \sup_{x,x'} \norm{\nabla_x \nabla_{x'} F''[\nu]}_{\mathrm{op}} < \infty$,
then $-M_{11} \id_{\bm\LLL^2_\nu} \preceq \nabla K \nabla^* \preceq M_{11} \id_{\bm\LLL^2_\nu}$, 
and so $\sigma \leq \tau c_{\PI} + M_{11}$ by Weyl's perturbation theorem \cite[Coroll.~III.2.6]{bhatia1997matrix}.
In particular, we expect that for any $\eps>0$, there should exist an initialization $\mu_0$ such that $\forall t, \chisq{\mu_t}{\nu} \geq C e^{-2 (\tau c_{\PI} + M_{11}) (1+\eps) t}$ for some constant $C$.

\subsection{Long-time convergence rate in all metrics} \label{subsec:MFLD_L2:longtime}

The local convergence result presented in \autoref{thm:MFLD_L2:localL2:gen_simple} can be used to deduce an estimate of the long-time convergence rate of MFLD,
provided that the dynamics eventually passes through a small enough $\chi^2$-divergence neighborhood of $\nu$.
This can be guaranteed using stronger assumptions on $F$,
such as global (weak) linear convexity and the uniform LSI property introduced by \cite{chizat2022mean,nitanda2022convex}.
Moreover, $\chi^2$-divergence is a ``sensitive'' enough error metric that exponential convergence in this metric implies exponential convergence for other metrics of interest, with the same rate.
Our long-time convergence result is as follows.

\begin{corollary} \label{coroll:MFLD_L2:longtime:allmetrics}
	Under \autoref{assum:MFLD_A}, additionally suppose that
	$F+\tau_0 H$ is globally linearly convex for some $0 \leq \tau_0 < \tau$ and that
    there exists $c_{u}>0$ such that for any $\mu \in \PPP_2(\RR^d)$, the \emph{proximal Gibbs measure}
	$
		\hmu \propto \exp\left( -\frac1\tau F'[\mu](x) \right) \d x
	$
	satisfies LSI with a constant at least $c_{u}$.
	Also suppose that $F$ has Lipschitz-continuous Wasserstein gradients, i.e., that
	\begin{equation} \label{eq:MFLD_L2:longtime:lipschitz_wass_grads}
		\forall \mu, \tilde\mu \in \PPP_2(\RR^d),
		\forall x, \tx \in \RR^d,~~
		\norm{\nabla F'[\mu](x) - \nabla F'[\tilde\mu](\tx)} 
		\leq \beta \left( \norm{x-\tx} + W_2(\mu, \tilde\mu) \right)
	\end{equation}
	for some $\beta<\infty$.%
	\footnote{It is known that Lipschitz continuity of the Wasserstein gradients implies displacement smoothness \cite[App.~A]{chizat2022mean}. It is unclear at present whether the converse implication holds.}
    Moreover, suppose that
    $\inf_{q>1} \int \big| \frac{\d\mu_{t_0}}{\d\nu} \big|^q \d\nu < \infty$ for some $t_0 < \infty$.
    Then for any $\eps>0$, there exist $C, t_1<\infty$ such that for all $t \geq t_1$,
    \begin{equation*}
        \max\left\{
            W_2^2(\mu_t, \nu), 
            \KLdiv{\mu_t}{\nu},
            \chisq{\mu_t}{\nu},
            F_\tau(\mu_t) - F_\tau(\nu)
        \right\}
        \leq C e^{-2 (\tau-\tau_0) c_{\PI} (1-\eps) t}.
    \end{equation*}
\end{corollary}

\begin{proof}
    Under the conditions of the theorem, by \cite[Prop.~2.3]{chen2022uniform},
    we have that for any $\delta>0$ and any $\mu_0$ such that 
    $\inf_{q>1} \int \big| \frac{\d\mu_{t_0}}{\d\nu} \big|^q \d\nu < \infty$ for some $t_0 < \infty$, 
    there exists $t_1 < \infty$ such that $\chisq{\mu_{t_1}}{\nu} \leq \delta$.
    Combining this fact with \autoref{thm:MFLD_L2:localL2:gen_simple}, 
    we directly obtain the announced long-time convergence bound in $\chi^2$-divergence.
    
    The bounds in Wasserstein distance and in KL-divergence then follow from the fact that
    \begin{equation*}
		\forall \mu \in \PPP_2(\RR^d),~
		\frac{c_u}{2} \, W_2^2(\mu, \nu) \leq \KLdiv{\mu}{\nu} \leq \chisq{\mu}{\nu}.
	\end{equation*}
    Here the first inequality follows from the assumption that $\nu$, which is equal to its own proximal Gibbs measure by the stationarity condition, satisfies LSI with a constant at least $c_u$, and from \cite[Theorem~1]{otto2000generalization}.
    The second inequality can be proved by applying Jensen's inequality on the $\log$ inside the definition of KL-divergence and by using that $\log(1+x) \leq x$.
    
    For the long-time convergence bound in $F_\tau(\cdot) - F_\tau(\nu)$, it then suffices to apply \autoref{lm:MFLD_L2:longtime:Ftau} below. Its proof is delayed to \autoref{subsec:MFLD_L2:delayed_pfs}.
\end{proof}

\begin{lemma} \label{lm:MFLD_L2:longtime:Ftau}
	For any $F: \PPP_2(\RR^d) \to \RR$ that has $\beta$-Lipschitz-continuous Wasserstein gradients (in the sense of \eqref{eq:MFLD_L2:longtime:lipschitz_wass_grads}),
	for any stationary measure $\nu$ of $F_\tau$ that satisfies LSI with a constant $c$, we have
	\begin{equation*}
		\forall \mu,~
		F_\tau(\mu) - F_\tau(\nu)
		\leq \left( \beta + \frac{\beta^2}{c\, \tau} \right) W_2^2(\mu, \nu)
		+ \tau \KLdiv{\mu}{\nu}
		+ \frac\tau 2 \chisq{\mu}{\nu}.
	\end{equation*}
\end{lemma}

\subsection{Proof of \autoref{thm:MFLD_L2:localL2:gen_simple} and \autoref{thm:MFLD_L2:localL2:quad}} \label{subsec:MFLD_L2:localL2_proof}

We now provide the full proofs of our main results \autoref{thm:MFLD_L2:localL2:gen_simple} and \autoref{thm:MFLD_L2:localL2:quad}.
The proofs involve 
the elliptic operator
$L = -\frac1\nu \nabla \cdot (\nu \nabla \bullet)
= -\Delta - \nabla \log \nu \cdot \nabla$,
viewed as an unbounded operator over $\LLL^2_\nu$.
Note that $L$ is essentially self-adjoint and that
for any regular enough $f, g$, by integration by parts,
\begin{equation*}
	\innerprod{f}{Lg}_\nu = \innerprod{\nabla f}{\nabla g}_\nu
\end{equation*}
where the inner product on the right-hand side is the one of $\bm\LLL^2_\nu$.
So we may write $L = \nabla^* \nabla$, where $\nabla^* = -\frac1\nu \nabla \cdot (\nu \bullet)$ is the adjoint of $\nabla$ in the sense of unbounded operators between $\LLL^2_\nu$ and $\bm\LLL^2_\nu$.
Moreover, note that by the PI for $\nu$, any $f \in \LLL^2_\nu \cap \{1\}^\perp$ has an inverse by $L$ belonging to 
$\LLL^2_\nu$
and
\begin{equation*}
	\forall f \in \LLL^2_\nu \cap \{1\}^\perp,~~
	\norm{f}^2_{H^{-1}_\nu} \coloneqq \innerprod{f}{L^{-1} f}_\nu
	\leq c_{\PI}^{-1} \norm{f}_\nu^2.
\end{equation*}
(In other words, the weighted Sobolev space inclusions
$(H^1_\nu, \norm{\nabla \,\cdot\,}_\nu) \subset (\LLL^2_\nu, \norm{\cdot}_\nu) \subset (H^{-1}_\nu, \norm{\cdot}_{H^{-1}_\nu})$ are bounded, where 
$H^1_\nu = \left\{ f \in \LLL^2_\nu ; ~\text{$f$ admits a weak derivative $\nabla f \in \bm\LLL^2_\nu$} \right\}$
and $H^{-1}_\nu$ is its dual space w.r.t.\ the inner product of $\LLL^2_\nu$.)

It is straightforward to check that \autoref{thm:MFLD_L2:localL2:gen_simple} is implied by the following statement.
Here we use the shorthands $a \vee b = \max\{a,b\}$ and $a \wedge b = \min\{a, b\}$.

\begin{theorem} \label{thm:MFLD_L2:localL2:gen_full}
	Under \autoref{assum:MFLD_A}, denote $k(x,x') = F''[\nu](x,x')$ and 
	let $R: \PPP_2(\RR^d) \to \RR$ be the functional such that
	\begin{equation*}
		\forall \mu,~
		F(\mu) = F(\nu) + \int (-\tau \log \nu) \,\d(\mu-\nu) + \frac12 \iint k(x,x') \,\d(\mu-\nu)(x) \d(\mu-\nu)(x') + R(\mu).
	\end{equation*}
	Note that
	$R'[\nu], R''[\nu] = \cst$.
	Additionally assume that
	\begin{equation*}
		M_{11} = \! \sup_{x,y \in \RR^d} \norm{\nabla_x \nabla_y F''[\nu]}_{\mathrm{op}},
		~~
		M_{111} = \!\!\! \sup_{\substack{\tmu \in \PPP_2(\RR^d)\\x,y,z \in \RR^d}} \!\! \norm{\nabla_x \nabla_y \nabla_z R'''[\tmu]}_{\mathrm{op}},
		~~
		M_{12} = \! \sup_{\substack{\tmu \in \PPP_2(\RR^d)\\x,y \in \RR^d}} \! \norm{\nabla_x \nabla_y^2 R''[\tmu]}_{\mathrm{op}}
	\end{equation*}
	are finite.
	Then for any $0<\eps\leq \frac18$, denoting
	$M = M_{11} \vee \left( (M_{12} + M_{111}) / c_{\PI} \right)$ and
	\begin{equation} \label{eq:MFLD_L2:localL2:genthm_Wt}
		W_t = \norm{f_t}_{H^{-1}_\nu}^2 + \gamma \norm{f_t}_\nu^2
		~~\text{where}~~
		\gamma = \frac{\tau (\tau-\tau_0) c_{\PI} \, \eps^2}{128 \, M^2},
	\end{equation}
	if $W_0 \leq \olW_0 \coloneqq
	2^{-20} (\tau-\tau_0)^4 c_{\PI}^3 \, \eps^4 M^{-4} \wedge 1$
	then
	\begin{equation} \label{eq:MFLD_L2:localL2:genthm_Wtcontracts}
		\frac{d}{dt} W_t \leq -2 (\tau-\tau_0) c_{\PI} (1-\eps) W_t.
	\end{equation}
	In particular, if $\norm{f_0}_\nu^2 \leq (\gamma+1/c_{\PI})^{-1} \olW_0$
	then
	\begin{equation*}
		\norm{f_t}_\nu^2 \leq C e^{-2 (\tau-\tau_0) c_{\PI} (1-\eps) t} \norm{f_0}_\nu^2
	\end{equation*}
	where $C 
	= \gamma^{-1} (\gamma + 1/c_{\PI}) 
	= 1 + \frac{128 M^2}{\tau (\tau-\tau_0) c_{\PI}^2 \eps^2}$.
\end{theorem}

\begin{proof}
	Denote by $\tK: \MMM(\RR^d) \to \CCC(\RR^d)$ the operator given by $(\tK \mu)(x) = \int k(x,y) \,\d\mu(y)$, and by $K: \LLL^2_\nu \to \LLL^2_\nu$ the operator given by $(K f)(x) = \int k(x,y) f(y) \d\nu(y)$.
	Given our decomposition of $F$, the MFLD rewrites
	\begin{equation*}
		\partial_t \mu_t = \tau \nabla \cdot \left( \mu_t \nabla \log \frac{\mu_t}{\nu} \right) + \nabla \cdot \left( \mu_t \nabla \tK (\mu_t - \nu) \right)
		+ \nabla \cdot (\mu_t \nabla R'[\mu_t]).
	\end{equation*}
	In terms of $f_t = \frac{\d\mu_t}{\d\nu}-1$, this writes
	\begin{align}
		\partial_t f_t &= -\tau L f_t + \frac1\nu \nabla \cdot \left( \nu (f_t+1) \nabla K f_t \right)
		+ \frac1\nu \nabla \cdot \left( \nu (f_t+1) \nabla R'[\mu_t] \right)   \nonumber \\
		&= -\tau L f_t - L K f_t - \nabla^* (f_t \nabla K f_t)
		- \nabla^* ((f_t+1) \nabla R'[\mu_t])
	\label{eq:MFLD_L2:localL2:MFLD_PDE_L2}
	\end{align}
	where 
    $\nabla^* = -\frac1\nu \nabla \cdot (\nu \,\bullet)$
    and $L = \nabla^* \nabla$
    as defined at the beginning of this subsection.
    
	Let us estimate the time-derivative of $\norm{f_t}^2_{H^{-1}_\nu} = \innerprod{f_t}{L^{-1} f_t}_\nu$:
	\begin{multline}
		\frac{d}{dt} \innerprod{f_t}{L^{-1} f_t}_\nu
		= 2 \innerprod{f_t}{L^{-1} \partial_t f_t}_\nu   \nonumber \\
		= -2 \tau \norm{f_t}_\nu^2 - 2 \innerprod{f_t}{K f_t}_\nu
		- 2 \, \underbrace{\innerprod{f_t}{L^{-1} \nabla^* (f_t \nabla K f_t)}_\nu}_{\eqqcolon~ \mathrm{err}_1} \,
		- 2 \, \underbrace{\innerprod{f_t}{L^{-1} \nabla^* ((f_t+1) \nabla R'[\mu_t])}}_{\eqqcolon~ \mathrm{err}_2}.
	\label{eq:MFLD_L2:localL2:err1err2}
	\end{multline}
	\begin{itemize}
		\item The first two terms are upper-bounded by
	    $-2 (\tau-\tau_0) \norm{f_t}_\nu^2$, 
	    by the assumption \eqref{eq:MFLD_L2:assum_F''nu} on $k = F''[\nu]$.
		\item The third term can bounded absolutely  as follows: by Cauchy-Schwarz inequality w.r.t.\ the inner product $\innerprod{\cdot}{\cdot}_{H^{-1}_\nu} = \innerprod{\cdot}{L^{-1} \, \cdot}_\nu$,
		\begin{align*}
			\abs{\mathrm{err}_1} 
			= \abs{ \innerprod{f_t}{\nabla^* (f_t \nabla K f_t)}_{H^{-1}_\nu} }
			&\leq \norm{f_t}_{H^{-1}_\nu}
			\sqrt{\innerprod{\nabla^* (f_t \nabla K f_t)}{L^{-1}\, \nabla^* (f_t \nabla K f_t)}_\nu} \\
			&= \norm{f_t}_{H^{-1}_\nu}
			\sqrt{\innerprod{f_t \nabla K f_t}{~\underbrace{\nabla L^{-1} \nabla^*}~ (f_t \nabla K f_t)}_\nu}.
		\end{align*}
		Now $\nabla L^{-1} \nabla^*$ is precisely the orthogonal projector, in $\bm\LLL^2_\nu$, onto the subspace consisting of gradient fields, so $0 \preceq \nabla L^{-1} \nabla^* \preceq \id_{\bm\LLL^2_\nu}$. Hence
		\begin{equation*}
			\abs{\mathrm{err}_1} / \norm{f_t}_{H^{-1}_\nu}
			\leq \norm{f_t \nabla K f_t}_\nu
			= \sqrt{\int f_t^2 \norm{\nabla K f_t}^2 \d\nu}
			\leq \norm{f_t}_\nu ~ \sup_{\RR^d} \norm{\nabla K f_t}.
		\end{equation*}
		Moreover, we show in \autoref{lm:MFLD_L2:localL2:estim_nablaKf_infty} below that
		$
			\sup_{\RR^d} \norm{\nabla K f}
			\leq M_{11} \norm{f}_{H^{-1}_\nu}
		$,
		so finally
		\begin{equation*}
			\abs{\mathrm{err}_1}
			\leq \norm{f_t}_{H^{-1}_\nu} \cdot \norm{f_t}_\nu \cdot M_{11} \norm{f}_{H^{-1}_\nu}.
		\end{equation*}
		\item Similarly, we can bound the fourth term as
		\begin{equation*}
			\abs{\mathrm{err}_2} / \norm{f_t}_{H^{-1}_\nu}
			\leq \norm{(f_t+1) \nabla R'[\mu_t]}_\nu
			\leq (1 + \norm{f_t}_\nu) ~ \sup_{\RR^d} \norm{\nabla R'[\mu_t]}.
		\end{equation*}
		Furthermore, we show in \autoref{lm:MFLD_L2:localL2:estim_nablaR'_infty} below that
		\begin{equation*}
			\sup_{\RR^d} \norm{\nabla R'[\mu_t]}
			\leq \frac12 (M_{12} + M_{111}) W_2^2(\mu_t, \nu)
			\leq ~\underbrace{~
				(M_{12} + M_{111}) c_{\PI}^{-1}
				~}_{\eqqcolon \, M_R ~\text{for concision}}	~
			\norm{f_t}_\nu^2,
		\end{equation*}
		where in the second inequality we used the estimate
		$W_2^2(\mu, \nu) \leq 2 c_{\PI}^{-1} \chisq{\mu}{\nu}$
		by \cite{liu2020poincare}; so finally
		\begin{equation*}
			\abs{\mathrm{err}_2}
			\leq \norm{f_t}_{H^{-1}_\nu} \cdot (1+\norm{f_t}_\nu) \cdot
			M_R \norm{f_t}_\nu^2.
		\end{equation*}
	\end{itemize}
	In summary, we have
	\begin{align*}
		\frac{d}{dt} \norm{f_t}_{H^{-1}_\nu}^2 &\leq -2 (\tau-\tau_0) \norm{f_t}_\nu^2
		+ 2 \norm{f_t}_{H^{-1}_\nu} \left[
		\norm{f_t}_\nu \cdot M_{11} \norm{f_t}_{H^{-1}_\nu} 
		+ (1 + \norm{f_t}_\nu) \cdot M_R  \norm{f_t}_\nu^2
		\right].
	\end{align*}
	
	The inequality above is not yet enough to conclude to any local convergence, due to the term in $M_R \norm{f_t}_{H_\nu^{-1}} \norm{f_t}_\nu^3$, as it is not yet guaranteed that $\norm{f_t}_\nu$ itself remains small for all $t$ if small at initialization.
	Nonetheless, a rough bound on $\frac{d}{dt} \norm{f_t}_\nu^2$ will now be enough to conclude. 
	Namely, 
	\begin{align*}
		\frac{d}{dt} \norm{f_t}_\nu^2 
		&= 2 \innerprod{f_t}{\partial f_t}_\nu \\
		&= -2 \tau \innerprod{f_t}{L f_t}_\nu
		- 2 \innerprod{f_t}{\nabla^* \left( (f_t+1) \nabla K f_t \right)}_\nu 
		- 2 \innerprod{f_t}{\nabla^* \left( (f_t+1) \nabla R'[\mu_t] \right)} \\
		&\leq -2 \tau \norm{\nabla f_t}^2_\nu
		+ 2 \norm{\nabla f_t}_\nu 
		\cdot \left[ \sup_{\RR^d} \norm{\nabla K f_t} + \sup_{\RR^d} \norm{\nabla R'[\mu_t]} \right]
		\cdot \left( 1 + \norm{f_t}_\nu \right) \\
		&\leq -2 \tau \norm{\nabla f_t}^2_\nu
		+ 2 \norm{\nabla f_t}_\nu 
		\left( M_{11} \norm{f_t}_{H^{-1}_\nu} + M_R \norm{f_t}_\nu^2 \right) 
		\left( 1 + \norm{f_t}_\nu \right)
	\end{align*}
	where in the last line we used \autoref{lm:MFLD_L2:localL2:estim_nablaKf_infty} and \autoref{lm:MFLD_L2:localL2:estim_nablaR'_infty} again.
	
	Thus we have shown that
	\begin{align*}
		\frac{d}{dt} \norm{f_t}_{H^{-1}_\nu}^2 
		&\leq -2 (\tau-\tau_0) \norm{f_t}_\nu^2 
		+ 2 M_{11} \norm{f_t}_{H^{-1}_\nu}^2 \norm{f_t}_\nu 
		+ 2 M_R \norm{f_t}_{H^{-1}_\nu} \norm{f_t}_\nu^2 (1 + \norm{f_t}_\nu) \\
		\frac{d}{dt} \norm{f_t}^2_\nu
		&\leq -2 \tau \norm{\nabla f_t}_\nu^2
		+ 2 \norm{\nabla f_t}_\nu 
		\left( M_{11} \norm{f_t}_{H^{-1}_\nu} + M_R \norm{f_t}_\nu^2 \right) 
		\left( 1 + \norm{f_t}_\nu \right)
	\end{align*}
	and it remains to combine these bounds appropriately, knowing that $c_{\PI} \norm{\cdot}_{H^{-1}_\nu}^2 \leq \norm{\cdot}_\nu^2 \leq c_{\PI}^{-1} \norm{\nabla \,\cdot\,}_\nu^2$.
	That is, denoting
	\begin{equation*}
		z_t = \norm{f_t}_{H^{-1}_\nu}^2,
		\qquad\quad
		a_t = \norm{f_t}_\nu^2,
		\qquad\quad
		b_t = \norm{\nabla f_t}_\nu^2,
        \qquad\quad
        c = c_{\PI},
        \qquad\quad
        \ol\tau = \tau - \tau_0
	\end{equation*}
    for concision, we have
	\begin{equation} \label{eq:MFLD_L2:localL2:zab_system_gen}
		\begin{aligned}
			\dot{z}_t &\leq -2 \ol\tau a_t + 2 M_{11} z_t \sqrt{a_t} + 2 M_R \sqrt{z_t}\, a_t (1+\sqrt{a_t}) \\
			\dot{a}_t &\leq -2 \tau b_t + 2 \sqrt{b_t} \left( M_{11} \sqrt{z_t} + M_R\, a_t \right) (1+\sqrt{a_t})
		\end{aligned}
		\qquad\qquad \text{and} \qquad\qquad
		c z_t \leq a_t \leq c^{-1} b_t
	\end{equation}
	and we want to show that $a_t$ converges to $0$ exponentially.
	
	Here is one possible way to proceed. (Refinements at this step may improve the values of $\olW_0$ and of $C$ in the theorem statement, but the rate we obtain is tight.)
	For simplicity, let $M \coloneqq M_{11} \vee M_R$.
	Fix a small $\delta>0$ to be chosen later.
	Since $\forall A, B \in \RR,~ 2 A B \leq A^2 + B^2$ and $(A+B)^2 \leq 2 (A^2+B^2)$,
	\begin{align*}
        \dot{a}_t
        &\leq 
        - 2 \tau b_t 
        + 2 \sqrt{b_t} \cdot M \left( \sqrt{z_t} + a_t \right) (1+\sqrt{a_t}) \\
		&\leq - 2 \tau b_t 
		+ 2 \delta \tau \, b_t + \frac{M^2}{2\delta \tau} \left( \sqrt{z_t} + a_t \right)^2 (1+\sqrt{a_t})^2 \\
		&\leq -2 \tau (1-\delta) b_t
		+ \frac{M^2}{2\delta \tau} \cdot 2 \left( z_t + a_t^2 \right) \cdot 2 (1 + a_t) \\
		&\leq -2 \ol\tau (1-\delta) b_t
		+ \frac{2 M^2}{\delta \tau} \left( z_t + a_t^2 \right) (1+a_t).
	\end{align*}
	Now fix a small $\gamma>0$ to be chosen later. We have, denoting $W_t = z_t + \gamma a_t$,
	\begin{align*}
		&~ \dot{W}_t
		= \dot{z}_t + \gamma \dot{a}_t \\
		&\leq -2 \ol\tau a_t - 2 \ol\tau (1-\delta) \gamma b_t
		+ 2 M z_t \sqrt{a}_t
		+ 2 M \sqrt{z_t}~ a_t (1 + \sqrt{a_t})
		+ \frac{2 \gamma}{\delta \tau} M^2 \left( z_t + a_t^2 \right) (1+a_t) \\
		&\leq -2 \ol\tau c (1-\delta) (z_t + \gamma a_t)
		+ 2 M W_t \sqrt{\frac{W_t}{\gamma}} 
		+ 2 M \sqrt{W_t} \frac{W_t}{\gamma} \left( 1 + \sqrt{\frac{W_t}{\gamma}} \right)
		+ \frac{2\gamma}{\delta \tau} M^2 \left( W_t + \frac{W_t^2}{\gamma^2} \right) \left( 1 + \frac{W_t}{\gamma} \right) \\
		&= -2 \ol\tau c (1-\delta) \left[ 1 - \frac{1}{\ol\tau c (1-\delta)} \left( 
		M \sqrt{\frac{W_t}{\gamma}} 
		+ M \frac{\sqrt{W_t}}{\gamma} \left( 1 + \sqrt{\frac{W_t}{\gamma}} \right)
		+ \frac{\gamma}{\delta \tau} M^2 \left( 1 + \frac{W_t}{\gamma^2} \right) \left( 1 + \frac{W_t}{\gamma} \right)
		\right) \right] W_t.
	\end{align*}
	
	We can now conclude the proof: for any $0<\eps \leq \frac{1}{8}$, 
	by applying the above inequality with 
	$\delta = \eps/8$ and 
	$\gamma = \frac{\ol\tau c \,\cdot\, \tau \delta \,\cdot\, \eps}{16 M^2}$
	we get
	\begin{equation*}
		\dot{W}_t \leq -2 \ol\tau c (1-\eps/8) \left[ 1 - \frac{64}{63} \left( 
		\frac{M}{\ol\tau c} \cdot 3~ \sqrt{
			\frac{W_t}{\gamma} \vee \frac{W_t}{\gamma^2} \vee \frac{W_t^2}{\gamma^3}
		}
		+ \frac{\eps}{16}
		\left( 
		1 + 3
		\left( \frac{W_t}{\gamma} \vee \frac{W_t}{\gamma^2} \vee \frac{W_t^2}{\gamma^3} \right)
		\right)
		\right) \right] W_t. 
	\end{equation*}
	By assuming $W_0 \leq \olW_0$ for some $\olW_0$ such that
	$\frac{\olW_0}{\gamma} \vee \frac{\olW_0}{\gamma^2} \vee \frac{\olW_0^2}{\gamma^3}
	\leq \left(\frac{\ol\tau c \eps}{8M} \right)^2 \wedge 1$,
	we obtain that $t \mapsto W_t$ is decreasing and that
	\begin{equation*}
		\dot{W_t}
		\leq -2 \ol\tau c (1-\eps/8) \left[ 1 - \frac{64}{63} \Big( (\eps/8) \cdot 3 + (\eps/16) \cdot 4 \Big) \right] W_t
		\leq -2 \ol\tau c (1-\eps) W_t.
	\end{equation*}
	One can check that the value
	$\olW_0 = 2^{-20} \ol\tau^4 c^3 \eps^4 M^{-4} \wedge 1$
	used in the theorem statement is suitable.
	Hence the local contraction in $W_t$,
	and the local convergence in $\norm{f_t}_\nu^2$ follows by Gr\"onwall's lemma and by using that $\norm{\cdot}_{H^{-1}_\nu}^2 \leq c_{\PI}^{-1} \norm{\cdot}_\nu^2$.
\end{proof}

\begin{lemma} \label{lm:MFLD_L2:localL2:estim_nablaKf_infty}
	For any $\nu \in \PPP_2(\RR^d)$ satisfying a PI and any $k: \RR^d \times \RR^d \to \RR$,
	denoting by $K$ the operator given by $(K f)(x) = \int k(x,y) f(y) \d\nu(y)$,
	we have, for any $f \in \LLL^2_\nu$,
	\begin{equation*}
		\sup_{\RR^d} \norm{\nabla K f}
		\leq M_{11} \norm{f}_{H^{-1}_\nu}
	\end{equation*}
	where $M_{11} = \sup_{x,y} \norm{\nabla_x \nabla_y k}_{\mathrm{op}}$.
	In particular, 
	$K$ is a bounded operator over $\LLL^2_\nu$ with
	$\forall f, 
	\norm{K f}_\nu \leq c_{\PI}^{-1} M_{11} \norm{f}_\nu$.
\end{lemma}

\begin{proof}	
	For any $x \in \RR^d$ and $w \in \RR^d$ such that $\norm{w}=1$, denoting $h(y) = w^\top \nabla_x k(x, y)$,
	\begin{align*}
		w^\top \nabla (K f)(x)
		&= w^\top \int \nabla_x k(x, y) f(y) \d\nu(y)
		= \int fh ~\d\nu
		\leq \norm{f}_{H^{-1}_\nu} \norm{h}_{H^1_\nu} \\
		&\leq \norm{f}_{H^{-1}_\nu} ~
		\sqrt{\int \d\nu(y) \norm{\nabla_y \left(w^\top \nabla_x k(x,y) \right)}^2} \\
		&\leq \norm{f}_{H^{-1}_\nu} ~
		\sup_{y \in \RR^d} \norm{\nabla_x \nabla_y k}_{\mathrm{op}}
	\end{align*}
	where $\norm{h}_{H^1_\nu}^2 \coloneqq \innerprod{h}{L h}_\nu = \innerprod{\nabla h}{\nabla h}_\nu$.
	The announced inequality follows by taking the supremum w.r.t.\ $x$ and $w$.
	The second part of the lemma follows by noting that
	\begin{align*}
		\norm{K f}_\nu^2
		\leq c_{\PI}^{-1} \norm{\nabla K f}_\nu^2
		\leq c_{\PI}^{-1} M_{11}^2 	\norm{f}_{H^{-1}_\nu}^2
		\leq c_{\PI}^{-2} M_{11}^2 \norm{f}_\nu^2
	\end{align*}
	by the PI for $\nu$.
\end{proof}

\begin{lemma} \label{lm:MFLD_L2:localL2:estim_nablaR'_infty}
	Let $\nu \in \PPP_2^{\AC}(\RR^d)$ and $R: \PPP_2(\RR^d) \to \RR$ such that $R'[\nu], R''[\nu] = \cst$.
	Then for any $\mu \in \PPP_2(\RR^d)$,
	\begin{equation*}
		\sup_{\RR^d} \norm{\nabla R'[\mu]}
		\leq \frac12 (M_{12} + M_{111}) W_2^2(\mu, \nu)
	\end{equation*}
	where 
	$M_{12} = \sup_{x,y,\tmu} \norm{\nabla_x \nabla_y^2 R''[\tmu]}_{\mathrm{op}}$ and
	$M_{111} = \sup_{x,y,z,\tmu} \norm{\nabla_x \nabla_y \nabla_z R'''[\tmu]}_{\mathrm{op}}$.
\end{lemma}

\begin{proof}
	By Brenier's theorem, there exists an optimal transport map $T$ such that $\mu = T_\sharp \nu$.
	Let $\Psi = T - \id$ and $\mu^s = (\id + s \Psi)_\sharp \nu$ for all $0 \leq s \leq 1$, so that $\mu^0 = \nu$ and $\mu^1 = \mu$.
	Also denote $\Psi^s = \Psi \circ (\id + s \Psi)^{-1}$ for all $0 \leq s < 1$, so that $\partial_s \mu^s = -\nabla \cdot (\mu^s \Psi^s)$,
	and note that $\int \norm{\Psi^s}^2 \d\mu^s = W_2^2(\mu, \nu)$ for all $s$.
	
	Fix $x \in \RR^d$.
	By explicit computations, one has that
	\begin{align*}
		\frac{d}{ds} \nabla R'[\mu^s](x) &= \nabla_x \int \d\mu^s(y)~ \Psi^s(y)^\top \nabla_y R''[\mu^s](x,y) \\
		\frac{d^2}{ds^2} \nabla R'[\mu^s](x) &= \nabla_x \int d\mu^s(y)~ \Psi^s(y)^\top \nabla_y^2 R''[\mu^s](x,y) \, \Psi^s(y) \\
		&~~~~ + \nabla_x \iint \d\mu^s(y) \d\mu^s(y')~ \Psi^s(y)^\top \nabla_y \nabla_{y'} R'''[\mu^s](x,y,y') \, \Psi^s(y').
	\end{align*}
	So by a second-order Taylor expansion of $s \mapsto \nabla R'[\mu^s](x)$,
	since $R'[\nu]=\cst$ and $R''[\nu]=\cst$,
	\begingroup \allowdisplaybreaks
	\begin{align*}
		\nabla R'[\mu](x) &= 0 + 0 + \int_0^1 \d s ~(1-s)
		\Bigg(
			\nabla_x \int d\mu^s(y)~ \Psi^s(y)^\top \nabla_y^2 R''[\mu^s](x,y) \, \Psi^s(y) \\*
		&\qquad\qquad\qquad + 
			\nabla_x \iint \d\mu^s(y) \d\mu^s(y')~ \Psi^s(y)^\top \nabla_y \nabla_{y'} R'''[\mu^s](x,y,y') \, \Psi^s(y')
		\Bigg) \\
		\norm{\nabla R'[\mu](x)}
		&\leq \frac12 \sup_{s \in [0,1)}
		\Bigg( 
			\int \d\mu^s(y) ~\norm{\nabla_x \nabla_y^2 R''[\mu^s](x,y)}_{\mathrm{op}} \norm{\Psi^s(y)}^2 \\*
		&\qquad\qquad~~ 
			+ \iint \d\mu^s(y) \d\mu^s(y')~ \norm{\Psi^s(y)} \norm{\Psi^s(y')} ~\norm{\nabla_x \nabla_y \nabla_{y'} R'''[\mu^s](x,y,y')}_{\mathrm{op}}
		\Bigg) \\*
		&\leq \frac12 M_{12} W_2^2(\mu, \nu)
		+ \frac12 M_{111} \left( \sup_s \int \d\mu^s \norm{\Psi^s} \right)^2 
		\leq \frac12 (M_{12} + M_{111}) W_2^2(\mu, \nu),
	\end{align*}
	\endgroup
	as announced,
	where in the last inequality we used Jensen's inequality.
\end{proof}

In the case of a quadratic functional $F$, the proof can be simplified, leading to \autoref{thm:MFLD_L2:localL2:quad}, as we now present.

\begin{proof}[Proof of \autoref{thm:MFLD_L2:localL2:quad}]
	The beginning of the proof is identical to the one for the general case, specialized to $R(\mu)=0$. 
	%
	One can check that we obtain
	\begin{align*}
		\frac{d}{dt} \norm{f_t}_{H^{-1}_\nu}^2 
		&\leq -2 (\tau-\tau_0) \norm{f_t}_\nu^2 
		+ 2 M_{11} \norm{f_t}_{H^{-1}_\nu}^2 \norm{f_t}_\nu \\
		\frac{d}{dt} \norm{f_t}^2_\nu
		&\leq -2 \tau \norm{\nabla f_t}_\nu^2
		+ 2 M_{11} \norm{\nabla f_t}_\nu \norm{f_t}_{H^{-1}_\nu} (1 + \norm{f_t}_\nu)
	\end{align*}
	and it remains to combine these bounds appropriately.
	Denoting
	\begin{equation*}
		z_t = \norm{f_t}_{H^{-1}_\nu}^2,
		\qquad
		a_t = \norm{f_t}_\nu^2,
		\qquad
		b_t = \norm{\nabla f_t}_\nu^2,
		\qquad
		c = c_{\PI},
        \qquad
        \ol\tau = \tau-\tau_0,
		\qquad
		M = M_{11}
	\end{equation*}
	for concision, we have
	\begin{equation} \label{eq:MFLD_L2:localL2:zab_system_quad}
		\begin{aligned}
			\dot{z}_t &\leq -2 \ol\tau a_t + 2 M z_t \sqrt{a_t} \\
			\dot{a}_t &\leq -2 \tau b_t + 2 M \sqrt{b_t} \sqrt{z_t} (1+\sqrt{a_t})
		\end{aligned}
		\qquad\qquad \text{and} \qquad\qquad
		c z_t \leq a_t \leq c^{-1} b_t.
	\end{equation}
	
	Contrary to the proof for the general case, here we can directly establish local contraction of $z_t$. Indeed, since $z_t \leq \sqrt{z_t} \cdot \sqrt{c^{-1} a_t}$,
	\begin{align*}
		\dot{z}_t &\leq -2 \ol\tau a_t + 2 M \sqrt{z_t} \cdot c^{-1/2} a_t \\
		&= -2 \ol\tau a_t (1 - M \ol\tau^{-1} c^{-1/2} \sqrt{z_t}).
	\end{align*}
	Assuming $z_0 \leq \frac{\ol\tau^2 c (\eps/2)^2}{M^2}$, which we note can be ensured by assuming $a_0 \leq \frac{\ol\tau^2 c^2 (\eps/2)^2}{M^2}$, we obtain that $t \mapsto z_t$ is decreasing and that
	\begin{align}
		\dot{z}_t &\leq -2 \ol\tau a_t (1 - \eps/2)
		\leq -2 \ol\tau c (1-\eps/2) z_t   \nonumber \\
		\text{and so}~~
		z_t &\leq e^{-2 \ol\tau c (1-\eps/2) t} z_0
	\label{eq:MFLD_L2:localL2:zt_contracts}
	\end{align}
	by Gr\"onwall's lemma.
	In turn, we can use this control on $z_t$ to show the convergence of $a_t$ as follows. We have, since $\sqrt{b_t} \sqrt{z_t} \sqrt{a_t} \leq c^{-1/2} b_t \sqrt{z_t} \leq c^{-1/2} b_t \sqrt{z_0}$ since $t \mapsto z_t$ is decreasing, and since $\forall A,B \in \RR,~ 2AB \leq A^2+B^2$,
	\begin{align*}
		\dot{a}_t &\leq -2 \tau b_t + 2 M \sqrt{b_t} \sqrt{z_t} \sqrt{a_t} + 2 M \sqrt{b_t} \sqrt{z_t} \\
		&\leq -2 \tau b_t + 2 M c^{-1/2} \sqrt{z_0}\, b_t + \eps \tau\, b_t + \frac{1}{\eps \tau} M^2 z_t \\
		&= -2 \tau \left( 1 - M \tau^{-1} c^{-1/2} \sqrt{z_0} - \eps/2 \right) b_t + \frac{1}{\eps \tau} M^2 z_t.
	\end{align*}
	Since we assume
    $z_0 \leq \frac{\ol\tau^2 c (\eps/2)^2}{M^2}
    \leq \frac{\tau^2 c (\eps/2)^2}{M^2}$,
    then the first term is non-negative and bounded by $-2 \tau (1-\eps) b_t \leq -2 \tau c (1-\eps) a_t$. Then, using \eqref{eq:MFLD_L2:localL2:zt_contracts} to bound the second term, we get
	\begin{align*}
	\label{eq:MFLD_L2:localL2:defective_gronwall}
		\dot{a}_t &\leq -2 \tau c (1-\eps) a_t + \frac{1}{\eps \tau} M^2 z_0 \, e^{-2 \ol\tau c (1-\eps/2) t} \\
		&\leq -2 \ol\tau c (1-\eps) a_t + \frac{1}{\eps \tau} M^2 z_0 \, e^{-2 \ol\tau c (1-\eps/2) t} \\
		\frac{d}{dt} \left[ e^{2 \ol\tau c (1-\eps) t} a_t \right]
		&= e^{2 \ol\tau c (1-\eps) t} \left( \dot{a}_t + 2 \ol\tau c (1-\eps) a_t \right)
		\leq \frac{1}{\eps \tau} M^2 z_0 \, e^{-\ol\tau c \eps t} \\
		e^{2 \ol\tau c (1-\eps) t} a_t 
		&\leq a_0 + \frac{1}{\eps \tau} M^2 z_0
		\underbrace{~\int_0^t e^{-\ol\tau c \eps s} \d s~}_{\frac{1}{\ol\tau c \eps} (1 - e^{-\ol\tau c \eps t})} \\
		&\leq a_0 + \frac{1}{\eps \tau} M^2 z_0 \cdot \frac{1}{\ol\tau c \eps}
		\leq a_0 + \frac{M^2}{\tau \ol\tau c \eps^2} z_0 
		\leq \left( 1 + \frac{M^2}{\tau \ol\tau c^2 \eps^2} \right) a_0,
	\end{align*}
	as announced.
\end{proof}

\subsection{Delayed proofs of complementary Propositions} \label{subsec:MFLD_L2:delayed_pfs}

\begin{proof}[Proof of \autoref{prop:MFLD_L2:localL2:tau0=tau_unstable}]
	By finiteness of $M_{11} = \sup_{x,x'} \norm{\nabla_x \nabla_{x'} F''[\nu]}_{\mathrm{op}}$, the mapping
	$g \mapsto \iint F''[\nu](x,x') \allowbreak \, g(x) g(x') \allowbreak \,\d\nu(x) \d\nu(x') + \tau \int \abs{g}^2 \d\nu$ is continuous w.r.t.\ the Hilbertian norm on $\LLL^2_\nu$. 
	Indeed, 
	denoting by $K$ the operator such that $(Kg)(x) = \int F''[\nu](x,x') g(x') \d\nu(x')$,
	we have for any $g, h \in \LLL^2_\nu$
	\begin{align*}
		& \abs{\iint F''[\nu](x,x') \, g(x) g(x') \,\d\nu(x) \d\nu(x') 
			- \iint F''[\nu](x,x') \, h(x) h(x') \,\d\nu(x) \d\nu(x')} \\
		&= \abs{\innerprod{g}{Kg}_\nu - \innerprod{h}{Kh}_\nu}
		\leq \abs{\innerprod{g}{K(g-h)}_\nu} + \abs{\innerprod{h}{K(h-g)}_\nu}
		\leq (\norm{g}_\nu + \norm{h}_\nu) 
		\underbrace{ \norm{K(g-h)}_\nu }_{\leq c_{\PI}^{-1} M_{11} \norm{g-h}_\nu}
	\end{align*}
	by \autoref{lm:MFLD_L2:localL2:estim_nablaKf_infty}.
	Moreover, the set $\CCC^\infty_c$ of compactly supported and infinitely smooth functions on $\RR^d$ is dense in $\LLL^2_\nu$. So without loss of generality, we may assume that
	the function $f$ satisfying
	$\iint F''[\nu](x,x') \, f(x) f(x') \,\d\nu(x) \d\nu(x') + \tau \int \abs{f}^2 \d \nu < 0$
	belongs to $\CCC^\infty_c$.
	
	Let $\mu^\delta = (1 + \delta f) \nu \in \PPP(\RR^d)$ for any $\abs{\delta} < \norm{f}_\infty^{-1}$.
	Note that
	\begin{equation*}
		\int \norm{x}^2 \d\mu^\delta(x)
		= \int \norm{x}^2 \d\nu(x)
		+ \delta \int \norm{x}^2 f(x) \,\d\nu(x)
		\leq (1 + \abs{\delta} \, \norm{f}_\infty) \int \norm{x}^2 \d\nu(x) < \infty,
	\end{equation*}
	so $\mu^\delta \in \PPP_2(\RR^d)$.
	Further denote $g(\delta) = F_\tau(\mu^\delta)$, 
	then
	\begin{align*}
		g'(\delta) &= \int F'_\tau[\mu^\delta](x) \, f(x) \,\d\nu(x)
		= \int F'[\mu^\delta](x) \, f(x) \,\d\nu(x)
		+ \tau \int \log\left[ \left( 1 + \delta f(x) \right) \nu(x) \right] f(x) \,\d\nu(x) \\
		g''(\delta) &= \iint F''[\mu^\delta](x,x') \, f(x) f(x') \, \d\nu(x) \d\nu(x')
		+ \int \frac{\abs{f(x)}^2}{1 + \delta f(x)} \d\nu(x).
	\end{align*}
	In particular $g(0) = F_\tau(\nu)$, $g'(0) = 0$ by stationarity of $\nu$, and $g''(0) < 0$ by assumption.
	
	Furthermore, we claim that $g$ is $\CCC^2$ in a neighborhood of zero. 
	Indeed, the second term in the expression of $g''(\delta)$ is continuous on $(-\norm{f}_\infty^{-1}, \norm{f}_\infty^{-1})$ since $f \in \CCC^\infty_c$.
	For the first term, denoting $L = -\frac1\nu \nabla \cdot (\nu \nabla \bullet)$ and $\phi = L^{-1} f$ since $L$ is invertible on $\LLL^2_\nu \cap \{1\}^\perp$ by the PI assumption, we have
	\begin{align*}
		G_1(\delta) 
		&\coloneqq \iint F''[\mu^\delta](x,x') \, f(x) f(x') \, \d\nu(x) \d\nu(x') \\
		&= \iint \nabla \phi(x)^\top \nabla_x \nabla_{x'} F''[\mu^\delta](x,x') \, \nabla \phi(x') \, \d\nu(x) \d\nu(x') 
	\end{align*}
	and so
	\begin{align*}
		\abs{G_1(\delta) - G_1(0)}
		&= \abs{\iint \nabla \phi(x)^\top
			\Big( \nabla_x \nabla_{x'} F''[\mu^\delta](x,x') - \nabla_x \nabla_{x'} F''[\nu](x,x') \Big)
			\, \nabla \phi(x') \, \d\nu(x) \d\nu(x')} \\
		&\leq \sup_{x,x' \in \RR^d}
		\norm{\nabla_x \nabla_{x'} F''[\mu^\delta] - \nabla_x \nabla_{x'} F''[\nu]}_{\mathrm{op}}
		\left( \int \norm{\nabla \phi} \d\nu \right)^2.
	\end{align*}
	Now by Jensen's inequality and by PI, the second factor is bounded by
	$\left( \int \norm{\nabla \phi} \d\nu \right)^2
	\leq \norm{\nabla \phi}_\nu^2 = \innerprod{\phi}{L \phi}_\nu = \innerprod{L^{-1} f}{f}_\nu \leq c_{PI}^{-1} \norm{f}_\nu^2 < \infty$, and by finiteness of 
	$M_{111} = \sup_{\substack{\tmu \in \PPP_2(\RR^d)\\x,y,z \in \RR^d}} \norm{\nabla_x \nabla_y \nabla_z F'''[\tmu]}_{\mathrm{op}}$, the first factor is bounded by
	\begin{equation*}
		\sup_{x,x'}
		\norm{\nabla_x \nabla_{x'} F''[\mu^\delta] - \nabla_x \nabla_{x'} F''[\nu]}_{\mathrm{op}}
		\leq M_{111} W_1(\mu^\delta, \nu)
		\leq M_{111} W_2(\mu^\delta, \nu) 
		\leq M_{111} \sqrt{ 2 c_{\PI}^{-1} \chisq{\mu^\delta}{\nu} }
	\end{equation*}
	where the first inequality follows from 
	\cite[Lemma~D.8]{wang2024mean}
	and in the last inequality we used the estimate proved in \cite{liu2020poincare}.
	Finally, note that by definition of $\mu^\delta$,
	\begin{equation*}
		\chisq{\mu^\delta}{\nu} = \delta^2 \norm{f}_\nu^2 \to 0 ~~\text{as}~~ \delta \to 0,
	\end{equation*}
	which proves our claim.
	
	Since $g'(0) = 0$, $g''(0)<0$ and $g$ is $\CCC^2$ in a neighborhood of zero, then for any small enough non-zero $\delta$, we have by a Taylor's expansion that
	$g(\delta) < g(0)$, i.e., $F_\tau(\mu^\delta) < F_\tau(\nu)$.
	So along the MFLD $(\mu_t)_t$ initialized at $\mu_0 = \mu^\delta$, since $t \mapsto F_\tau(\mu_t)$ is non-increasing, then
	$\lim\inf_t F_\tau(\mu_t) \leq F_\tau(\mu^\delta) < F_\tau(\nu)$.
	On the other hand, $F_\tau$ is lower semi-continuous w.r.t.\ convergence in Wasserstein distance, i.e., $\lim\inf_n F_\tau(\mu_n) \geq F_\tau(\mu_\infty)$ for any sequence such that $(\mu_n)_n$ that converges narrowly to $\mu_\infty$ and $\int \norm{x}^2 \d\mu_n \to \int \norm{x}^2 \d\mu_\infty$
	\cite[Remark~7.1.11, Lemma~9.4.3]{ambrosio2008gradient}.
	So MFLD initialized at $\mu^\delta$ cannot converge to $\nu$ in Wasserstein distance.
	To finish the proof of the proposition, it suffices to check that $\chisq{\mu^\delta}{\nu}$ can be made arbitrarily small by choosing a small $\delta$, which is indeed the case since $\chisq{\mu^\delta}{\nu} \to 0$ as $\delta \to 0$ as we showed above.
\end{proof}

\begin{proof}[Proof of \autoref{prop:MFLD_L2:localL2:tau0>=0}]
	Assume by contradiction that the condition holds for some $\alpha >0$.
	Let $L = -\frac1\nu \nabla \cdot (\nu \nabla \bullet)$ and denote by $G = \overline{\nabla \CCC^\infty_c}^{\bm\LLL^2_\nu}$ the subspace of $\bm\LLL^2_\nu$ consisting of gradient fields. For any $\phi \in \CCC^\infty_c$, applying the condition to $f = L \phi$ and performing integrations by parts on the left-hand side yields
	\begin{equation*}
		\forall \phi \in \CCC^\infty_c,~
		\iint \nabla \phi(x)^\top \nabla_x \nabla_{x'} k(x,x') \, \nabla \phi(x') \,\d\nu(x) \d\nu(x')
		\geq \alpha \norm{L \phi}_\nu^2
		\geq \alpha \, c_{\PI} \int \norm{\nabla \phi}^2 \d\nu
	\end{equation*}
	where for the second inequality we used the PI for $\nu$ and the assumption that $\alpha>0$.
	Thus, we have
	\begin{equation*}
		\forall \Phi \in G,~
		\iint \Phi(x)^\top \nabla_x \nabla_{x'} k(x,x') ~\Phi(x') \,\d\nu(x) \d\nu(x')
		\geq \alpha \, c_{\PI} \int \norm{\Phi}^2 \d\nu.
	\end{equation*}
	Equivalently, $\tK \succeq \alpha \, c_{\PI} \id_G$ where $(\tK \Phi)(x) = \int \nabla_x \nabla_{x'} k(x,x') \, \Phi(x') \d\nu(x')$.
	Now the integrability assumption on $k$ is equivalent to the integral operator $\tK$ being Hilbert-Schmidt, which implies it is compact as an operator over $\bm\LLL^2_\nu$, and so also as an operator over $G$ since $\tK$ takes values in $G$ and vanishes over $G^\perp$. But by the spectral theorem, the eigenvalues of compact operators accumulate at $0$. So by choosing a sequence $(\lambda_n)_n$ of eigenvalues of $\tK$ which converges to $0$, with associated eigenfunctions $\Phi_n$, and by evaluating the above inequality at $\Phi = \Phi_n$, we obtain that $\forall n, \lambda_n \geq \alpha \, c_{\PI}$. Taking a limit $n \to \infty$ shows that $0 \geq \alpha \,c_{\PI}$, contradicting our assumption that $\alpha>0$.
\end{proof}

\begin{proof}[Proof of \autoref{lm:MFLD_L2:longtime:Ftau}]
	Let $\mu \in \PPP_2(\RR^d)$ and $\hmu \propto e^{-\frac1\tau F'[\mu](x)} \d x$.
	We have
	\begin{equation} \label{eq:MFLD_L2:longtime:upper_entropy_sandwich_idty}
		F_\tau(\mu) - F_\tau(\nu)
		= \tau \KLdiv{\mu}{\hat\mu} - \tau \KLdiv{\nu}{\hat\mu} - B_F(\nu | \mu)
	\end{equation}
	where $B_F(\mu_1 | \mu_0) \coloneqq F(\mu_1) - F(\mu_0) - \int F'[\mu_0] \,\d(\mu_1-\mu_0)$ denotes the Bregman divergence of $F$.
	This can be seen by inspecting the proof of the upper entropy sandwich \cite[Lemma~3.4]{chizat2022mean}, or it can be verified by explicitly computing, starting from the right-hand side:
	\begin{align*}
		\tau \KLdiv{\mu}{\hat\mu} - \tau \KLdiv{\nu}{\hat\mu} - B_F(\nu | \mu)
		&= \tau H(\mu) - \tau \int \d\mu \log \hat\mu - \tau H(\nu) + \tau \int \d\nu \log \hat\mu \\
		&~~~~ - F(\nu) + F(\mu) + \int F'[\mu] \,\d(\nu - \mu) \\
		&= F_\tau(\mu) - F_\tau(\nu)
		+ \tau \int \d(\nu - \mu) \log \hat\mu + \int F'[\mu] \,\d(\nu - \mu) \\
		&= F_\tau(\mu) - F_\tau(\nu).
	\end{align*}
	Now let us bound each term on the right-hand side of \eqref{eq:MFLD_L2:longtime:upper_entropy_sandwich_idty} separately.
	\begin{itemize}
		\item 
		The second term, $-\tau \KLdiv{\nu}{\hat\mu}$, is upper-bounded by $0$.
		\item 
		To bound the third term, $-B_F(\nu | \mu)$, we use that $\beta$-Lipschitz continuity of the Wasserstein gradients of $F$ implies that
		$\abs{ B_F(\mu_1 | \mu_0) } \leq \beta W_2^2(\mu_1, \mu_0)$ for all $\mu_1, \mu_0 \in \PPP_2(\RR^d)$.
		Indeed, denoting by $\pi$ the optimal transport plan with $X_\sharp \pi = \mu_1, Y_\sharp \pi = \mu_0$ and by $\mu_s = (s X + (1-s) Y)_\sharp \pi$ the displacement interpolation between $\mu_1$ and $\mu_0$, we have
		\begin{align*}
			& B_F(\mu_1 | \mu_0) = F(\mu_1) - F(\mu_0) - \int F'[\mu_0] \,\d(\mu_1-\mu_0) \\
			&= F(\mu_1) - F(\mu_0)
			- \iint \left( F'[\mu_0](x) - F'[\mu_0](y) \right) \d\pi(x,y) \\
			&= \int_0^1 \d s \frac{d}{ds} F(\mu_s)
			- \iint \d\pi \int_0^1 \d s \frac{d}{ds} F'[\mu_0](sx + (1-s)y) \\
			&= \int_0^1 \d s \iint \nabla F'[\mu_s](s x + (1-s) y)^\top (x-y) \,\d\pi 
			- \iint \d\pi \int_0^1 \d s\, \nabla F'[\mu_0](sx + (1-s) y)^\top (x-y) \\
			&= \int_0^1 \d s \iint \d\pi \left\{ \nabla F'[\mu_s](s x + (1-s) y) -  \nabla F'[\mu_0](s x + (1-s) y) \right\}^\top (x-y)
		\end{align*}
		and so
		\begin{align*}
			\abs{B_F(\mu_1 | \mu_0)} 
			&\leq \int_0^1 \d s \iint \d\pi \norm{ \nabla F'[\mu_s](s x + (1-s) y) -  \nabla F'[\mu_0](s x + (1-s) y) } \norm{x-y} \\
			&\leq \int_0^1 \d s 
			~\underbrace{ \iint \d\pi \norm{x-y} }_{\leq W_2(\mu_1, \mu_0)}~
			\cdot
			~\underbrace{ \sup_{\RR^d} \norm{\nabla F'[\mu_s] -  \nabla F'[\mu_0]} }_{\leq \beta W_2(\mu_s, \mu_0) \leq \beta W_2(\mu_1, \mu_0)}~
			\leq \beta W_2^2(\mu_1, \mu_0).
		\end{align*}
		\item 
		To bound the first term, we further decompose it as follows:
		\begin{equation*}
			\tau \KLdiv{\mu}{\hmu} = \tau \KLdiv{\mu}{\nu} + \tau \KLdiv{\nu}{\hmu}
			- \tau \int \log \frac{\hmu}{\nu} ~\d(\mu-\nu).
		\end{equation*}
		For the second term in this new expression, $\KLdiv{\nu}{\hmu}$, since $F'[\nu] + \tau \log \nu = \cst$ by stationarity and since $\nu$ satisfies LSI with a constant $c$,
		\begin{equation*}
			\KLdiv{\nu}{\hmu} 
			\leq \frac{1}{2 c} \int \d\nu \norm{\nabla \log \frac{\nu}{\hmu}}^2
			= \frac{1}{2 c \tau^2} \int \d\nu \norm{\nabla F'[\mu] - \nabla F'[\nu]}^2
			\leq \frac{\beta^2}{2 c \tau^2} W_2^2(\mu, \nu).
		\end{equation*}
		For the final term, $-\int \log \frac{\hmu}{\nu} ~\d(\mu-\nu)$, 
		denoting $C = \int \d\nu \log \frac{\hmu}{\nu}$,
		\begin{align*}
			\abs{\int \d(\mu-\nu)\, \log \frac{\hmu}{\nu}}
			&= \abs{\int \d\nu\, \left( \frac{\mu}{\nu}-1 \right) \left( \log \frac{\hmu}{\nu} + C \right)} \\
			&\leq \sqrt{\chisq{\mu}{\nu}}~
			\sqrt{\int \d\nu \abs{\log \frac{\hmu}{\nu} + C}^2}
			\leq \frac12 \chisq{\mu}{\nu}
			+ \frac12 \int \d\nu \abs{\log \frac{\hmu}{\nu} + C}^2
		\end{align*}
		by Cauchy-Schwarz inequality.
		Now LSI implies PI with the same constant \cite{otto2000generalization}, so $\nu$ satisfies PI with constant $c$, and
		since $\int \d\nu \left( \log \frac{\hmu}{\nu} + C \right) = 0$ by definition,
		\begin{equation*}
			\int \d\nu \abs{\log \frac{\hmu}{\nu} + C}^2
			\leq \frac{1}{c}
			\int \d\nu \norm{\nabla \log \frac{\hmu}{\nu}}^2
			= \frac{1}{c \tau^2} \int \d\nu \norm{\nabla F'[\mu] - \nabla F'[\nu]}^2 \leq \frac{\beta^2}{c \tau^2} W_2^2(\mu, \nu).
		\end{equation*}
	\end{itemize}
	Gathering the terms yields the announced bound on $F_\tau(\mu) - F_\tau(\nu)$.
\end{proof}

\section[Results for MFL-DA]{Results for mean-field Langevin descent-ascent} \label{sec:MFL-DA}

In this section, we analyze the convergence of the mean-field Langevin descent-ascent (MFL-DA) dynamics, which is the dynamics over $(\mu^x_t, \mu^y_t) \in \PPP_2(\XXX) \times \PPP_2(\YYY)$ given by,
for a fixed function $k(x,y)$ and temperature parameter $\tau > 0$,
\begin{equation} \label{eq:MFL-DA:MFL-DA}
	\begin{cases}
		\partial_t \mu^x_t = \nabla \cdot \left( \mu^x_t \nabla \int_\YYY k(\cdot, y) \d\mu^y_t(y) \right) + \tau \Delta \mu^x_t \\
		\partial_t \mu^y_t = -\nabla \cdot \left( \mu^y_t \nabla \int_\XXX k(x, \cdot) \d\mu^x_t(x) \right) + \tau \Delta \mu^y_t.
	\end{cases}
\end{equation}
Here $\XXX$ and $\YYY$ could be any Riemannian manifold without boundaries, but for simplicity we will assume henceforth that they are tori: $\XXX = \TT^{d_x}$, $\YYY= \TT^{d_y}$. We stress that our results can be extended straightforwardly to more general settings, at the cost of heavier notations.

Observe that \eqref{eq:MFL-DA:MFL-DA} can be interpreted as the Wasserstein gradient descent-ascent flow for the min-max objective functional
\begin{equation*}
	\min_{\mu^x \in \PPP_2(\XXX)}~
	\max_{\mu^y \in \PPP_2(\YYY)}~ 
	\iint_{\XXX \times \YYY} k(x,y) \,\d\mu^x(x) \d\mu^y(y) + \tau H(\mu^x) - \tau H(\mu^y),
\end{equation*}
which is linearly strictly convex-concave
and thus has a unique saddle point, denoted $(\nu^x, \nu^y)$.
Moreover, as soon as $k$ is bounded, the proximal Gibbs distribution pairs $(\hmu^x, \hmu^y)$ defined by
$
	\hmu^x \propto \exp\left( -\frac1\tau \int_\YYY k(\cdot, y) \d\mu^y(y) \right) \d x,
	~
	\hmu^y \propto \exp\left( \frac1\tau \int_\XXX k(x, \cdot) \d\mu^x(x) \right) \d y
$
satisfy LSI uniformly for any pair $(\mu^x, \mu^y)$
\cite{lu2023two}.
One could thus
expect that \eqref{eq:MFL-DA:MFL-DA} should converge globally towards $(\nu^x, \nu^y)$ as $t \to \infty$.
This is known to indeed be the case under additional assumptions, such as $k$ being strongly convex-concave%
\footnote{There is no strongly convex function on a torus; what we refer to here are the instances of \eqref{eq:MFL-DA:MFL-DA} on $\XXX = \RR^{d_x}, \YYY= \RR^{d_y}$ where $k$ is strongly convex-concave. They are studied in detail in \cite{cai2024convergence}.}
or $\tau$ being large enough, but unknown for the general setting.

\subsection{Local \texorpdfstring{$\LLL^2$}{L2} convergence} \label{subsec:MFL-DA:localL2}

Whether MFL-DA converges globally is an open problem asked in \cite{wang2024open}.
In this work we focus on the more modest goal of studying the \emph{local} convergence properties of MFL-DA, and we show the following. Recall that $(\nu^x, \nu^y)$ is the unique equilibrium pair of \eqref{eq:MFL-DA:MFL-DA}.

\begin{theorem} \label{thm:MFL-DA:localL2}
	Suppose $k \in \CCC^2(\XXX \times \YYY)$ and let $M_{11} = \sup_{x, y} \norm{\nabla_x \nabla_y k}_{\mathrm{op}}$.
	Further suppose $\nu^x$ and $\nu^y$ satisfy PI with constants $c_{\PI}^x$ resp.\ $c_{\PI}^y$
	and let $c_{\PI} = \min\{ c_{\PI}^x, c_{\PI}^y \}$.
	Then for any $0< \eps < 1$,
	if $\chisq{\mu^x_0}{\nu^x} + \chisq{\mu^y_0}{\nu^y} \leq \frac{\tau^2 c_{\PI}^2 \eps^2}{4 M_{11}^2}$ then the MFL-DA dynamics \eqref{eq:MFL-DA:MFL-DA} satisfies
	\begin{equation*}
		\forall t \geq 0,~
		\chisq{\mu^x_t}{\nu^x} + \chisq{\mu^y_t}{\nu^y}
		\leq \left( 1 + \frac{M_{11}^2}{\tau^2 c_{\PI}^2 \eps^2} \right) e^{-2 \tau c_{\PI} (1-\eps) t}
		\left( \chisq{\mu^x_0}{\nu^x} + \chisq{\mu^y_0}{\nu^y} \right).
	\end{equation*}
\end{theorem}

This theorem implies that the long-time convergence rates of MFL-DA in $\chi^2$-divergence, KL-divergence, and squared Wasserstein distance are all lower-bounded by $2 \tau \min\{c_{\PI}^x, c_{\PI}^y\}$, for the same reason as for MFLD (\autoref{coroll:MFLD_L2:longtime:allmetrics}), \emph{provided that convergence to $(\nu^x, \nu^y)$ does occur}.
We note that this lower bound on the long-time rate is tight in the case where $k$ is additively separable: $k(x,y) = V(x) - W(y)$, since MFL-DA then separates into two independent overdamped Langevin dynamics with stationary measures $\nu^x \propto e^{-V/\tau} \d x$ resp.\ $\nu^y \propto e^{-W/\tau} \d y$.

The proof follows from similar ideas as for MFLD in the quadratic case with $\tau_0=0$.
There, the conditional positive-semi-definiteness of the second variation at optimum played a crucial role. Here, it is replaced by a crucial cancellation of terms in $k$, occurring at step \eqref{eq:MFL-DA:localL2:cancellation} of the proof.

\begin{proof}
	Let $K: \LLL^2_{\nu^x} \times \LLL^2_{\nu^y} \to \RR$ be the bilinear operator given by $K(f, g) = \iint_{\XXX \times \YYY} k(x,y) \allowbreak f(x) \d\nu^x(x) \, g(y) \d\nu^y(y)$.
	With abuse of notation, we will also write $f^\top K g = K(f, g)$ and denote by $K: \LLL^2_{\nu^y} \to \LLL^2_{\nu^x}$ and $K^\top: \LLL^2_{\nu^x} \to \LLL^2_{\nu^y}$ the operators given by
	\begin{equation*}
		(K g)(x) = \int_\YYY k(x, y) \,g(y) \d\nu^y(y),
		\qquad\qquad
		(K^\top f)(y) = \int_\XXX k(x, y) \,f(x) \d\nu^x(x).
	\end{equation*}
	Furthermore, let 
	$\nabla_x^* = -\frac{1}{\nu^x} \nabla \cdot (\nu^x \,\bullet)$
	and
	$L_x = \nabla_x^* \nabla_x$ in $\LLL^2_{\nu^x}$, and likewise for $\nabla_y^*$ and $L_y$.
	Since $\nu^x, \nu^y$ satisfy PI then $L_x, L_y$ are invertible in $\LLL^2_{\nu^x} \cap \{1\}^\perp$ resp.\ $\LLL^2_{\nu^y} \cap \{1\}^\perp$ and
	\begin{equation*}
		\forall f \in \LLL^2_{\nu^x} ~\text{s.t.} \int_\XXX f \,\d\nu^x = 0,~~
		\norm{f}_{H^{-1}_x}^2 \coloneqq \innerprod{f}{L_x^{-1} f}_{\nu^x}
		\leq (c_{\PI}^x)^{-1} \norm{f}_{\nu^x}^2
	\end{equation*}
    and likewise for $\norm{g}_{H^{-1}_y}^2
    \coloneqq \innerprod{g}{L_y^{-1} g}_{\nu^y}$.
	
	From $(\mu^x_t, \mu^y_t)_t$ the MFL-DA dynamics given by \eqref{eq:MFL-DA:MFL-DA}, let $(f_t, g_t)_t$ in $\LLL^2_{\nu^x} \times \LLL^2_{\nu^y}$ be defined by 
	$f_t = \frac{\d\mu^x_t}{\d\nu^x}-1$ and
	$g_t = \frac{\d\mu^y_t}{\d\nu^y}-1$.
	We claim their time-evolution is given by
	\begin{equation} \label{eq:MFL-DA:ftgt}
		\begin{cases}
			\partial_t f_t = -\tau L_x f_t - L_x K g_t - \nabla_x^* (f_t \nabla K g_t) \\
			\partial_t g_t = -\tau L_y g_t + L_y K^\top f_t + \nabla_y^* (g_t \nabla K^\top f_t).
		\end{cases}
	\end{equation}
	Indeed,
	denoting $(\tK \mu^y)(x) = \int_\YYY k(x, y) \d\mu^y(y)$
	and $(\tK^\top \mu^x)(y) = \int_\XXX k(x, y) \d\mu^x(x)$,
	the first-order stationarity condition for the equilibrium pair $(\nu^x, \nu^y)$ translates to
	\begin{equation*}
		\tau \log \nu^x = -\tK \nu^y + \cst,
		\qquad\qquad
		\tau \log \nu^y = \tK^\top \nu^x + \cst
	\end{equation*}
	(equivalently, $(\nu^x, \nu^y)$ is equal to its own proximal Gibbs pair).
	Hence,
	\begin{align*}
		\partial_t \mu^x_t 
		&= \nabla \cdot (\mu^x_t \nabla \tK \mu^y_t) + \tau \Delta \mu^x_t \\
		&= \nabla \cdot (\mu^x_t \nabla \tK (\mu^y_t-\nu^y))
		+ \tau \nabla \cdot (\mu^x_t \nabla (-\log \nu^x))
		+ \tau \Delta \mu^x_t \\
		&= \tau \nabla \cdot \left( \mu^x_t \nabla \log \frac{\mu^x_t}{\nu^x} \right) 
		+ \nabla \cdot (\mu^x_t \nabla \tK (\mu^y_t - \nu^y)) \\
		\partial_t f_t 
		&= \tau \frac1{\nu^x} \nabla \cdot (\nu^x \nabla f_t)
		+ \frac{1}{\nu^x} \nabla \cdot \left( \nu^x (f_t+1) \nabla K g_t \right) \\
		&= -\tau L_x f_t - \nabla_x^* \left( (f_t+1) \nabla K g_t \right)
	\end{align*}
	and likewise for $\partial_t g_t$.
	
	Let us compute the time-derivatives of 
	\begin{equation*}
		Z_t \coloneqq \norm{f_t}_{H^{-1}_x}^2 + \norm{g_t}_{H^{-1}_y}^2
		~~~~\text{and}~~~~
		A_t \coloneqq \norm{f_t}_{\nu^x}^2 + \norm{g_t}_{\nu^y}^2.
	\end{equation*}
	Note that $Z_t \leq c_{\PI}^{-1} A_t$.
	We have
	\begin{align*}
		\frac{d}{dt} \norm{f_t}_{H^{-1}_x}^2
		&= 2 \innerprod{f_t}{L_x^{-1} \partial_t f_t}_{\nu^x}
		= -2 \tau \norm{f_t}_{\nu^x}^2
		- 2 f_t^\top K g_t
		- 2 \innerprod{f_t}{L_x^{-1} \nabla_x^* (f_t \nabla K g_t)}_{\nu^x} \\
		\frac{d}{dt} \norm{g_t}_{H^{-1}_y}^2
		&= 2 \innerprod{g_t}{L_y^{-1} \partial_t g_t}_{\nu^y}
		= -2 \tau \norm{g_t}_{\nu^y}^2
		+ 2 f_t^\top K g_t
		+ 2 \innerprod{g_t}{L_y^{-1} \nabla_y^* (g_t \nabla K^\top f_t)}_{\nu^y},
	\end{align*}
	and so, since the terms in $K$ cancel out,
	\begin{align}
		\dot{Z}_t
		&= \frac{d}{dt} \left( \norm{f_t}_{H^{-1}_x}^2 + \norm{g_t}_{H^{-1}_y}^2 \right)   \nonumber \\
		&= -2 \tau A_t 
		- 2 \innerprod{f_t}{L_x^{-1} \nabla_x^* (f_t \nabla K g_t)}_{\nu^x}
		+ 2 \innerprod{g_t}{L_y^{-1} \nabla_y^* (g_t \nabla K^\top f_t)}_{\nu^y}.
	\label{eq:MFL-DA:localL2:cancellation}
	\end{align}
	Let us control the last two terms in absolute value. By Cauchy-Schwarz inequality w.r.t.\ the inner product $\innerprod{\cdot}{L_x^{-1} \,\cdot}_{\nu^x}$,
	\begin{align*}
		\abs{\innerprod{f_t}{L_x^{-1} \nabla_x^* (f_t \nabla K g_t)}_{\nu^x}}
		&\leq \sqrt{\innerprod{f_t}{L_x^{-1} f_t}_{\nu^x}}
		~ \sqrt{\innerprod{f_t \nabla K g_t}{\nabla_x L_x^{-1} \nabla_x^* (f_t \nabla K g_t)}_{\nu^x}}.
	\end{align*}
	Note that in the second factor here, $\nabla_x L_x^{-1} \nabla_x^*$ is precisely the orthogonal projector, in $\bm\LLL^2_{\nu^x}$, onto the subspace consisting of gradient fields, and so $0 \preceq \nabla_x L_x^{-1} \nabla_x^* \preceq \id_{\bm\LLL^2_{\nu^x}}$.
	Hence the second factor can be further bounded by 
	\begin{align*}
		\norm{f_t \nabla K g_t}_{\nu^x} \leq \norm{f_t}_{\nu^x} \cdot \sup_{x \in \XXX} \norm{\nabla (K g_t)(x)}
		\leq \norm{f_t}_{\nu^x} \cdot M_{11} \norm{g_t}_{H^{-1}_y},
	\end{align*}
	where the second inequality follows from \autoref{lm:MFLD_L2:localL2:estim_nablaKf_infty}.
	Thus, by the symmetric bound for the third term in \eqref{eq:MFL-DA:localL2:cancellation},
	and since $\forall A, B \in \RR,~ 2 A B \leq A^2+B^2$ and $A+B \leq \sqrt{2} \sqrt{A^2+B^2}$,
	\begin{align*}
		\dot{Z}_t &\leq -2 \tau A_t 
		+ 2 \norm{f_t}_{H^{-1}_x} \cdot \norm{f_t}_{\nu^x} \cdot M_{11} \norm{g_t}_{H^{-1}_y}
		+ 2 \norm{g_t}_{H^{-1}_y} \cdot \norm{g_t}_{\nu^y} \cdot M_{11} \norm{f_t}_{H^{-1}_x} \\
		&= -2 \tau A_t 
		+ 2 M_{11} \norm{f_t}_{H^{-1}_x} \norm{g_t}_{H^{-1}_y} \left( \norm{f_t}_{\nu^x} + \norm{g_t}_{\nu^y} \right) \\
		&\leq -2 \tau A_t + 2 M_{11} Z_t \sqrt{A_t}.
	\end{align*}
	For $A_t$, we have the rough estimates
	\begin{align*}
		\frac{d}{dt} \norm{f_t}^2_{\nu^x}
		&= 2 \innerprod{f_t}{\partial_t f_t}_{\nu^x}
		= -2 \tau \innerprod{f_t}{L_x f_t}_{\nu^x}
		- 2 \innerprod{\nabla f_t}{(f_t+1) \nabla K g_t}_{\nu^x} \\
		&\leq -2 \tau \norm{\nabla f_t}_{\nu^x}^2
		+ 2 \norm{\nabla f_t}_{\nu^x} \cdot (1 + \norm{f_t}_{\nu^x}) \cdot \sup_{\XXX} \norm{\nabla K g_t} \\
		&\leq -2 \tau \norm{\nabla f_t}_{\nu^x}^2
		+ 2 M_{11} \norm{\nabla f_t}_{\nu^x} \norm{g_t}_{H^{-1}_y} (1 + \norm{f_t}_{\nu^x})
	\end{align*}
	where in the last line we used \autoref{lm:MFLD_L2:localL2:estim_nablaKf_infty} again,
	and symmetrically for $\norm{g_t}^2_{\nu^y}$, so
	\begin{align*}
		\dot{A}_t &\leq -2 \tau \left( \norm{\nabla f_t}_{\nu^x}^2 + \norm{\nabla g_t}_{\nu^x}^2 \right)
		+ 2 M_{11} \norm{\nabla f_t}_{\nu^x} \norm{g_t}_{H^{-1}_y} (1 + \norm{f_t}_{\nu^x}) \\
		&\qquad\qquad\qquad\qquad\qquad\qquad
		+ 2 M_{11} \norm{\nabla g_t}_{\nu^y} \norm{f_t}_{H^{-1}_x} (1 + \norm{g_t}_{\nu^y}) \\
		&\leq -2 \tau B_t + 2 M_{11} \sqrt{B_t} \sqrt{Z_t} (1 + \sqrt{A_t})
	\end{align*}
	where we introduce $B_t \coloneqq \norm{\nabla f_t}_{\nu^x}^2 + \norm{\nabla g_t}_{\nu^x}^2$.
    For the second inequality we used that 
    $\forall z, z', a, a', b, b', \allowbreak
    b z' (1+a) + b' z (1+a') \leq (b z' + b' z) (1 + \max\{a,a'\}) \leq \sqrt{b^2 + b^{\prime 2}} \sqrt{z^2 + z^{\prime 2}} (1 + \sqrt{a^2 + a^{\prime 2}})$ by Cauchy-Schwarz inequality.
	
	Thus, denoting $c = \min\{c_{\PI}^x, c_{\PI}^y\}$ and $M = M_{11}$ for concision, we have
	\begin{equation*}
		\begin{aligned}
			\dot{Z}_t &\leq -2 \tau A_t + 2 M Z_t \sqrt{A_t} \\
			\dot{A}_t &\leq -2 \tau B_t + 2 M \sqrt{B_t} \sqrt{Z_t} (1+\sqrt{A_t})
		\end{aligned}
		\qquad\qquad \text{and} \qquad\qquad
		c Z_t \leq A_t \leq c^{-1} B_t.
	\end{equation*}
	This is exactly the same system of inequalities as \eqref{eq:MFLD_L2:localL2:zab_system_quad} in the proof of \autoref{thm:MFLD_L2:localL2:quad}. So we can follow the same steps and obtain the same conclusions as in that proof.
	Namely we get that for any $0< \eps < 1$,
	if $Z_0 \leq \frac{\tau^2 c (\eps/2)^2}{M^2}$, which can be ensured by assuming $A_0 \leq \frac{\tau^2 c^2 (\eps/2)^2}{M^2}$, then
	$Z_t \leq e^{-2 \tau c (1-\eps/2) t} Z_0$ and 
	$A_t \leq \left( 1 + \frac{M^2}{\tau^2 c^2 \eps^2} \right) e^{-2 \tau c (1-\eps) t} A_0$.
\end{proof}


\begin{remark}[Conjectured spectral characterization of the exact rate] \label{rk:MFL-DA:minReSp}
	Similar to \autoref{subsubsec:MFLD_L2:localL2:exact_rate} for MFLD,
	we expect the exact long-time convergence rate of MFL-DA to be equal to $2 \min \{ \Re(\lambda); \lambda \in \spectrum(M) \}$ where $M$ is the operator over $(\LLL^2_{\nu^x} \cap \{1\}^\perp) \times (\LLL^2_{\nu^y} \cap \{1\}^\perp)$ corresponding to the dominant terms in the
	$\LLL^2$ formulation \eqref{eq:MFL-DA:ftgt} of the dynamics:
	\begin{equation*}
		M \begin{pmatrix}
			f \\ g
		\end{pmatrix}
		= \begin{pmatrix}
			\tau L_x f + L_x K g \\
			\tau L_y g - L_y K^\top f
		\end{pmatrix},
		~~\text{or symbolically,}~~
		M = \begin{bmatrix}
			\tau L_x & L_x K \\
			-L_y K^\top & \tau L_y
		\end{bmatrix}.
	\end{equation*}
	Equivalently, instead of $M$ we may consider the symmetric operator over $\overline{\nabla \CCC^\infty_c(\XXX)}^{\,\bm\LLL^2_{\nu^x}} \times \overline{\nabla \CCC^\infty_c(\YYY)}^{\,\bm\LLL^2_{\nu^y}}$
	given by
	$
		\tM = \begin{bmatrix}
			\tau \nabla_x \nabla_x^* & \nabla_x K \nabla_y^* \\
			-\nabla_y K^\top \nabla_x^* & \tau \nabla_y \nabla_y^*
		\end{bmatrix}
	$,
	since $\tM$ is similar to $M$ via conjugation by 
	$\begin{bmatrix}
		\nabla_x^* & \\
		& \nabla_y^*
	\end{bmatrix}$.
	
	
	The intuition that $\tM$ determines the long-time convergence rate of MFL-DA can also be obtained by considering the Wasserstein Hessian of the min-max objective
	$\FFF_\tau(\mu^x, \mu^y) = \iint_{\XXX\times \YYY} k \,\d(\mu^x \otimes \mu^y) + \tau H(\mu^x) - \tau H(\mu^y)$, 
	since it also rewrites
	$\iint_{\XXX \times \YYY} k \,\d((\mu^x-\nu^x) \otimes (\mu^y - \nu^y)) + \tau \KLdiv{\mu^x}{\nu^x} - \tau \KLdiv{\mu^y}{\nu^y} + \cst$
	by the first-order stationarity condition.
	Namely, one can show that
	\begin{align*}
		\Hess_{(\nu^x,\nu^y)} \FFF_\tau\left(
			\begin{pmatrix}
				\Phi^x \\ \Phi^y
			\end{pmatrix},
			\begin{pmatrix}
				\Phi^x \\ \Phi^y
			\end{pmatrix}
		\right)
		&= 2 \iint_{\XXX \times \YYY} \Phi^x(x) \cdot \nabla_x \nabla_y k(x,y) \cdot \Phi^y(y) \,\d\nu^x(x) \d\nu^y(y) \\
		&~~~~ + \tau \Hess_{\nu^x} \KLdiv{\cdot}{\nu^x}(\Phi^x, \Phi^x)
		- \tau \Hess_{\nu^y} \KLdiv{\cdot}{\nu^y}(\Phi^y, \Phi^y) \\
		&= \begin{pmatrix}
			\Phi^x \\ \Phi^y
		\end{pmatrix}^\top
		\underbrace{
			\begin{bmatrix}
				\tau \nabla_x \nabla_x^* & \nabla_x K \nabla_y^* \\
				\nabla_y K^\top \nabla_x^* & -\tau \nabla_y \nabla_y^*
			\end{bmatrix}
		}_{\text{
			$=$
			{\tiny $\begin{bmatrix}
				\id & \\
				& -\id
			\end{bmatrix} $}
			$\tM$
		}}
		\begin{pmatrix}
			\Phi^x \\ \Phi^y
		\end{pmatrix},
	\end{align*}
	where the left-hand side is defined as $\restr{\frac{d^2}{ds^2} \FFF_\tau((\id + s \Phi^x)_\sharp \nu^x, (\id + s \Phi^y)_\sharp \nu^y)}{s=0}$.
%
\end{remark}

\subsection{Consequences for the two-timescale dynamics}

It was shown by \cite{lu2023two} that a two-timescale variant of MFL-DA,
\begin{equation} \label{eq:MFL-DA:2scale}
	\begin{cases}
		\partial_t \mu^x_t = \nabla \cdot \left( \mu^x_t \nabla \int_\YYY k(\cdot, y) \d\mu^y_t(y) \right) + \tau \Delta \mu^x_t \\
		\partial_t \mu^y_t = \Gamma \left[ -\nabla \cdot \left( \mu^y_t \nabla \int_\XXX k(x, \cdot) \d\mu^x_t(x) \right) + \tau \Delta \mu^y_t \right],
	\end{cases}
\end{equation}
converges globally in KL-divergence when the relative timescale $\Gamma$ is large enough or small enough.
\cite{an2025convergence} analyzed the same dynamics using a specialized coupling technique and extended the range of values of $\Gamma$ for which global convergence is guaranteed, this time in $1$-Wasserstein distance.

We note that our convergence analysis also applies to this two-timescale dynamics simply by adapting the definitions of $Z_t, A_t, B_t$ in the proof, and we obtain the following.

\begin{corollary}
	Suppose $k \in \CCC^2(\XXX \times \YYY)$ and let $M_{11} = \sup_{x, y} \norm{\nabla_x \nabla_y k}_{\mathrm{op}}$.
	Further suppose $\nu^x$ and $\nu^y$ satisfy PI with constants $c_{\PI}^x$ resp.\ $c_{\PI}^y$
	and let $c^\Gamma_{\PI} = \min\{ c_{\PI}^x, \Gamma c_{\PI}^y \}$.
	Then for any $0< \eps < 1$,
	if $\chisq{\mu^x_0}{\nu^x} + \chisq{\mu^y_0}{\nu^y} \leq \frac{\tau^2 (c^\Gamma_{\PI})^2 \eps^2}{4 \Gamma M_{11}^2}$ then
	\begin{equation*}
		\forall t \geq 0,~
		\chisq{\mu^x_t}{\nu^x} + \chisq{\mu^y_t}{\nu^y}
		\leq \left( 1 + \frac{\Gamma M^2}{\tau^2 (c^\Gamma_{\PI})^2 \eps^2} \right) e^{-2 \tau c^\Gamma_{\PI} (1-\eps) t}
		\left( \chisq{\mu^x_0}{\nu^x} + \chisq{\mu^y_0}{\nu^y} \right).
	\end{equation*}
\end{corollary}

Thus, adjusting the relative timescale allows to correct for differences in the magnitude of $c_{\PI}^x$ and $c_{\PI}^y$, since it yields a local convergence rate of $2 \tau \min\{ c_{\PI}^x, \Gamma c_{\PI}^y \}$.
Actually this fact is quite natural if one notices that $\Gamma$ can also be interpreted as a spatial rescaling of $\YYY$, i.e., that \eqref{eq:MFL-DA:2scale} is equivalent to the single-timescale dynamics \eqref{eq:MFL-DA:MFL-DA} if we replace $\nabla_y$ by $\Gamma^{-1/2} \nabla_y$.

\begin{proof}
	We reuse the same notations as in the proof of \autoref{thm:MFL-DA:localL2}, except we set
	\begin{equation*}
		Z_t \coloneqq \norm{f_t}_{H^{-1}_x}^2 + \Gamma^{-1} \norm{g_t}_{H^{-1}_y}^2,
		~~~~
		A_t \coloneqq \norm{f_t}_{\nu^x}^2 + \norm{g_t}_{\nu^y}^2,
		~~~~
		B_t \coloneqq \norm{\nabla f_t}_{\nu^x}^2 + \Gamma \norm{\nabla g_t}_{\nu^y}^2
	\end{equation*}
	and $M = \sqrt{\Gamma}\, M_{11}$ and $c = c_{\PI}^\Gamma = \min\{c_{\PI}^x, \Gamma c_{\PI}^y\}$.
	One can check that, after adapting the computations, we again end up with the same system of inequalities \eqref{eq:MFLD_L2:localL2:zab_system_quad}, from where we can follow the same steps as in the proof of \autoref{thm:MFLD_L2:localL2:quad}, and the result follows.
\end{proof}

%

\section[Results for multi-species flows]{Results for multi-species flows} \label{sec:MFG}

In this section we adapt our analysis to general dynamics of the following form, corresponding to multi-species flows with uniform diffusion.
Consider $N$ mappings 
\begin{equation*}
	\forall I \in \{1,...,N\},~
	V_I: \PPP_2(\RR^{d_1}) \times ... \times \PPP_2(\RR^{d_N}) \times \RR^{d_I} \to \RR
\end{equation*}
and write $V_I[\mu^1, ..., \mu^N](z^I)$ for the evaluation of $V_I$ at $(\mu^1, ..., \mu^N, z^I)$,
to be interpreted as a potential over $z^I \in \RR^{d_I}$ dependent on the measures $\mu^1, ..., \mu^N$.
Let $\tau>0$ and suppose the $V_I$'s satisfy appropriate regularity assumptions such that the following system of PDEs is well-posed, where $\mu^I_t \in \PPP_2(\RR^{d_I})$ for each $I$:
\begin{equation} \label{eq:MFG:dyn_mu}
	\forall t \geq 0, \forall I \in \{1,...,N\},~
	\partial_t \mu^I_t = \nabla \cdot \left( \mu^I_t \nabla V_I[\mu^1_t, ..., \mu^N_t] \right) + \tau \Delta \mu^I_t.
\end{equation}
For ease of notation, introduce the shorthands
$\ol\PPP_2 = \PPP_2(\RR^{d_1}) \times ... \times \PPP_2(\RR^{d_n})$
and, generically for any $(\mu^I)_I$, $\ol\mu = (\mu^1, ..., \mu^N) \in \ol\PPP_2$.

We refer to \cite{conger2025monotone} for a discussion of the literature related to such dynamics.
As a brief summary, let us only explain how \eqref{eq:MFG:dyn_mu} can be interpreted
as the time-evolution of a mean-field system with $N$ types (``species'') of agents.
Namely, suppose that there is an infinite number of indistinguishable agents from each species $I$, and that each agent is made to choose some strategy $z^I \in \RR^{d_I}$; each species $I$ can then be described by the probability distribution of the chosen strategies $\mu^I \in \PPP(\RR^{d_I})$.
Moreover, suppose each agent $i$ of species $I$ evolves their strategy $z^{I,i}$ by following the gradient flow of a potential $V_I[\ol\mu]: \RR^{d_I} \to \RR$, dependent on the strategy distributions of all species, with added isotropic noise of variance $\tau$.
Then the time-evolution of the distributions $\mu^1, ..., \mu^N$ is precisely given by \eqref{eq:MFG:dyn_mu}.

In this work, we view \eqref{eq:MFG:dyn_mu} as a dynamical system over $\ol\mu_t \in \ol\PPP_2$, and we study its local convergence to an equilibrium tuple under the following structural assumption.

\begin{assumption} \label{assum:MFG_B}
	The system of PDEs \eqref{eq:MFG:dyn_mu} is well-posed and has an equilibrium tuple $\ol\nu = (\nu^1, ..., \nu^N)$, i.e., such that
	$\forall I \in \{1, ..., N\},~ V_I[\ol\nu] + \tau \log \nu^I = \cst$ on $\RR^{d_I}$.
	Moreover, the functions
	\begin{equation*}
		k_{IJ}(z^I, z^J) = \frac{\delta V_I[\ol\nu](z^I)}{\delta \nu^J(z^J)}
		\qquad
		(I, J \in \{1, ..., N\})
	\end{equation*}
	satisfy that
	for all $f^1 \in \LLL^2_{\nu^1} \cap \{1\}^\perp, ~ ... \, , ~ f^N \in \LLL^2_{\nu^N} \cap \{1\}^\perp$,
	\begin{equation} \label{eq:MFG:assum_kIJ}
		\sum_{I,J}
		\iint_{\RR^{d_I} \times \RR^{d_J}} k_{IJ}(z^I, z^J) ~ f^I(z^I) f^J(z^J) ~\d\nu^I(z^I) \d\nu^J(z^J)
		\geq 
		-\tau_0 \sum_I \int_{\RR^{d_I}} \abs{f^I}^2 \d\nu^I
	\end{equation}
	for some $0 \leq \tau_0 < \tau$.
	Furthermore, for each $I \in \{1, ..., N\}$, $\nu^I$ satisfies PI with a constant $c_{\PI}^I$.
\end{assumption}

\begin{remark}
	We considered a uniform diffusion coefficient $\tau$ in \eqref{eq:MFG:dyn_mu}, but our analyses can be adapted to treat variants of the dynamics with a different coefficient $\tau_I$ for each $I$.
	Likewise, for $\tau_0$, our analysis can be adapted to handle the case where the right-hand side of \eqref{eq:MFG:assum_kIJ} is replaced by $-\sum_I \tau_{0I} \int \abs{f^I}^2 \d\nu^I$, yielding a local convergence rate estimate dependent on $\min_I (\tau_I - \tau_{0I})$.
\end{remark}

\begin{remark}
	In the case where the $V_I$ derive from ``loss'' functionals $F_I: \ol\PPP_2 \to \RR$ via $V_I[\ol\mu](z^I) = \frac{\delta F_I(\ol\mu)}{\delta \mu^I(z^I)}$, the dynamics \eqref{eq:MFG:dyn_mu} can be interpreted as a differentiable game over the Wasserstein space, in the sense of \cite{letcher2019differentiable}.
	With this perspective, the assumption \eqref{eq:MFG:assum_kIJ} corresponds to a lower bound on the symmetric part of the Jacobian of the game. 
\end{remark}

Our results in this setting are formalized in the following two theorems.
Their proofs are very similar to the case of MFLD detailed in \autoref{sec:MFLD_L2}, and are placed in \autoref{subsec:MFG:pfs} below.

\begin{theorem} \label{thm:MFG:gen}
	For each $I$, let $R_I: \ol\PPP_2 \times \RR^{d_I} \to \RR$ be the mapping such that
	\begin{equation*}
		\forall \ol\mu \in \ol\PPP_2,~
		V_I[\ol\mu] = V_I[\ol\nu] + \sum_J \int_{\RR^{d_J}} k_{IJ}(\cdot, z^J) \,\d(\mu_J - \nu_J)(z^J) + R_I[\ol\mu]
		~~\text{over $\RR^{d_I}$}.
	\end{equation*}
	Note that $R_I[\ol\nu]=\cst$ on $\RR^{d_I}$ and $\frac{\delta R_I[\ol\nu]}{\delta \nu^J} = \cst$ on $\RR^{d_I} \times \RR^{d_J}$.
	Assume the following quantities are finite, for each $I \in \{1, ..., N\}$:
	\begin{align*}
		M_{11}^{IJ} &= \sup_{z^I, z^J} \norm{\nabla_{z^I} \nabla_{z^J} \,k_{IJ}}_{\mathrm{op}}, \\
		M_{111}^{IJK} &= \sup_{\substack{\ol\mu \in \ol\PPP_2 \\ z^I, z^J, z^K}}
		\norm{\nabla_{z^I} \nabla_{z^J} \nabla_{z^K}\, \frac{\delta^2 R_I[\ol\mu](z^I)}{\delta \mu^J(z^J) \delta \mu^K(z^K)}}_{\mathrm{op}},
		\quad
		M_{12}^{IJ} = \sup_{\substack{\ol\mu \in \ol\PPP_2 \\ z^I, z^J}}
		\norm{\nabla_{z^I} \nabla^2_{z^J}\, \frac{\delta R_I[\ol\mu](z^I)}{\delta \mu^J(z^J)}}_{\mathrm{op}}.
	\end{align*}
	For concision denote $M_{11} = \sup_{I,J} M_{11}^{IJ}$ and likewise for $M_{111}, M_{12}$, 
	and $c_{\PI} = \min \left\{ c_{\PI}^1, ..., c_{\PI}^N \right\}$.
	Then for any $0<\eps \leq \frac18$, there exist
	$r_0^{-1}, C = \mathrm{poly}\left( \eps^{-1}, (\tau-\tau_0)^{-1}, c_{\PI}^{-1}, M_{11}, M_{111}, M_{12}, N \right)$
	such that if $\sum_I \chisq{\mu^I_0}{\nu^I} \leq r_0$ then
	\begin{equation*}
		\forall t \geq 0,~
		\sum_I \chisq{\mu^I_t}{\nu^I}
		\leq C e^{-2 (\tau-\tau_0) c_{\PI} (1-\eps) t}
		\sum_I \chisq{\mu^I_0}{\nu^I}.
	\end{equation*}
\end{theorem}

In the case where the $V_I$ are linear in $\ol\mu$, we have the following more explicit result.
An example is the polymatrix continuous game setting considered in \cite{lu2025convergence}, as discussed in the next subsection.

\begin{theorem} \label{thm:MFG:quadr}
	Under the setting of the theorem above, if additionally $R_I = 0$ for all $I$, then the constants can be taken equal to
	$r_0 = \frac{(\tau-\tau_0)^2 c_{\PI}^2 \eps^2}{8 N^2 M_{11}^2}$ and
	$C = 1 + \frac{2 N^2 M_{11}^2}{\tau (\tau-\tau_0) c_{\PI}^2 \eps^2}$.
\end{theorem}

\subsection{A sufficient condition: linear monotonicity} \label{subsec:MFG:examples}

To support our assumption \eqref{eq:MFG:assum_kIJ} on the $k_{IJ}$, we now describe a sufficient condition.
When it is satisfied,
we say the dynamics \eqref{eq:MFG:dyn_mu} is \emph{linearly monotone}.
This is the setting considered in the simplified statement of our result in the introduction, \autoref{thm:intro:MFG}.

\begin{proposition} \label{prop:MFG:linearlymonotone}
	In the setup of the dynamics \eqref{eq:MFG:dyn_mu}, suppose the potentials $(V_I)_I$ are such that
	\begin{equation} \label{eq:MFG:linearlymonotone}
		\forall \, \ol\mu, \, \ol\mu' \in \ol\PPP_2,~~
		\sum_I \int_{\RR^{d_I}} 
		\left( V_I[\ol\mu] - V_I[\ol\mu'] \right) ~ \d(\mu^I - \mu^{\prime\, I}) \geq 0.
	\end{equation}
	Then for any $\ol\mu \in \ol\PPP_2$
	and any bounded and compactly supported $s^1, ..., s^N$ such that 
	$\int_{\RR^{d_1}} \d s^1 = ... = \int_{\RR^{d_N}} \d s^N = 0$,
	\begin{equation} \label{eq:MFG:linearlymonotone_2nd_order}
		\sum_{I,J}
		\iint_{\RR^{d_I} \times \RR^{d_J}}
		\frac{\delta V_I[\ol\mu](z^I)}{\delta \mu^J(z^J)} ~
		\d s^I(z^I) \d s^J(z^J)
		\geq 0.
	\end{equation}
	In particular, the assumption \eqref{eq:MFG:assum_kIJ} in \autoref{assum:MFG_B} is automatically satisfied with $\tau_0=0$
	(provided that the $\nabla_{z^I} \nabla_{z^J} \,k_{IJ}$ are uniformly bounded).
\end{proposition}

\begin{proof}
	Let $\ol\mu \in \ol\PPP_2$ and $\ols = (s^1, ..., s^N)$ with $s^J$ bounded and compactly supported and $\int_{\RR^{d_J}} \d s^J = 0$ for all $J$.
	Denote $\ol\mu_\delta = \ol\mu + \delta \ol s \in \ol\PPP_2$ for any $\delta \in \RR$.
	Then by definition of the first variation,
	\begin{equation*}
		\frac{1}{\delta^2} \sum_I \int_{\RR^{d_I}} 
		\left( V_I[\ol\mu_\delta] - V_I[\ol\mu] \right) ~ \d(\mu^I_\delta - \mu^I)
		~\to~ 
		\sum_{I,J} \iint_{\RR^{d_I} \times \RR^{d_J}}
		\frac{\delta V_I[\ol\mu](z^I)}{\delta \mu^J(z^J)} ~
		\d s^I(z^I) \d s^J(z^J)
		~~~~\text{as}~~~~
		\delta \to 0.
	\end{equation*}
	Hence the announced implication.
	
	In particular,
	\eqref{eq:MFG:assum_kIJ} can then be verified for all $f^1, ..., f^N$ with $f^J \in \LLL^2_{\nu^J} \cap \{1\}^\perp$ for each $J$, by considering $s^J = f^J \nu^J$ and by a density argument to remove the boundedness and compact support requirements.
	Indeed the left-hand side of \eqref{eq:MFG:assum_kIJ} is continuous in the $f^J$ w.r.t.\ the $\LLL^2$ norms, by a similar argument as in the proof of \autoref{prop:MFLD_L2:localL2:tau0=tau_unstable}, provided that the $\nabla_{z^I} \nabla_{z^J} \,k_{IJ}$ are uniformly bounded.
\end{proof}

\begin{remark}
	The converse implication \eqref{eq:MFG:linearlymonotone_2nd_order} $\implies$ \eqref{eq:MFG:linearlymonotone} is also true under appropriate regularity of the $V_I$. This can be shown by applying the mean-value theorem to 
	$g: \delta \mapsto \sum_I \int_{\RR^{d_I}} 
	\big( V_I[\ol\mu + \delta (\ol\mu'-\ol\mu)] \allowbreak - V_I[\ol\mu] \big) \allowbreak ~ \d(\mu^{\prime\, I} - \mu^I)$.
	Indeed $g(1)$ equals the left-hand side of \eqref{eq:MFG:linearlymonotone} and $g(0) = 0$, so assuming that $g$ is differentiable, there exists $0<\xi<1$ such that $g(1) = g'(\xi)$. 
	Now by formal computation, one can expect $g'(\xi)$ to be precisely equal to the left-hand side of \eqref{eq:MFG:linearlymonotone_2nd_order} 
	with $s^I = \mu^{\prime\, I} - \mu^I$
	and with ``$\ol\mu$'' replaced by $\ol\mu + \xi (\ol\mu'-\ol\mu)$.
	Determining the appropriate assumptions on the $V_I$ to justify this reasoning rigorously is left to future work.
\end{remark}

\begin{remark}
	For $N=1$, linear monotonicity is also known as Lasry-Lions monotonicity in the mean-field game theory literature.
	Its connections to displacement monotonicity and other related conditions were investigated by \cite{gangbo2022global,gangbo2022mean,meszaros2024mean,graber2023monotonicity}.
\end{remark}

MFLD with a globally linearly convex $F$ analyzed in \autoref{sec:MFLD_L2}, as well as the MFL-DA dynamics from \autoref{sec:MFL-DA}, are examples of instances of \eqref{eq:MFG:dyn_mu} that satisfy linear monotonicity, with $N=1$ and $N=2$ respectively.
Next we describe another example that appeared in the literature, with arbitrary $N$.

\paragraph{$N$-player pairwise-zero-sum polymatrix continuous games.}
Consider the case where the $V_I$ derive from functionals $F_I: \ol\PPP_2 \to \RR$ via
\begin{equation*}
	\forall I,~
	\forall \ol\mu \in \ol\PPP_2,~
	\forall z^I \in \RR^{d_I},~~
	V_I[\ol\mu](z^I) = \frac{\delta F_I(\ol\mu)}{\delta \mu^I(z^I)}.
\end{equation*}
In a game theory context, $F_I$ can be interpreted as the loss of the $I$-th player in a $N$-player game with mixed strategy sets $\PPP_2(\RR^{d_1}), ..., \PPP_2(\RR^{d_N})$.
Further suppose there exist $k_{IJ}: \RR^{d_I} \times \RR^{d_J} \to \RR$ such that
\begin{equation*}
	\forall I,~
	F_I(\ol\mu) = \sum_J \iint_{\RR^{d_I} \times \RR^{d_J}} k_{IJ}(z^I, z^J) \, \d\mu^I(z^I) \d\mu^J(z^J),
\end{equation*}
and that it holds
\begin{equation*}
	\forall I,~
	k_{II} = 0
	~~~~\text{and}~~~~
	\forall I \neq J,~
	k_{IJ}(z^I, z^J) = -k_{JI}(z^J, z^I).
\end{equation*}
That is, following the terminology of \cite{cai2016zerosum}, we are presented with a polymatrix continuous game which is pairwisely zero-sum.
In this context, \cite{lu2025convergence} analyzed the global convergence of a variant of \eqref{eq:MFG:dyn_mu} that uses time-averaged and exponentially discounted gradients (over tori instead of Euclidean spaces).

In this setting, the linear monotonicity condition \eqref{eq:MFG:linearlymonotone} holds, as
\begin{align*}
	\forall \, \ol\mu, \, \ol\mu',~~
	&\sum_I \int_{\RR^{d_I}}
	\left( 
		\sum_J \int_{\RR^{d_J}} k(z^I, z^J) ~ \d(\mu^J - \mu^{\prime\, J})(z^J)
	\right)
	\d(\mu^I - \mu^{\prime\, I})(z^I) \\
	&= \sum_{I,J} \int_{\RR^{d_I} \times \RR^{d_J}}
	k(z^I, z^J)
	\, \d(\mu^I - \mu^{\prime\, I})(z^I)
	\, \d(\mu^J - \mu^{\prime\, J})(z^J)
	= 0.
\end{align*}
Moreover, for the dynamics over tori, the existence of an equilibrium tuple is proved in \cite[Prop.~2.2]{lu2025convergence}, and the PI assumption is ensured by the boundedness of the $k_{IJ}$ and by the Holley-Stroock criterion (see \cite[Lemma~2.1]{lu2023two}), so that our \autoref{assum:MFG_B} is satisfied.
Also note that $V_I[\ol\mu](z^I) = \frac{\delta F_I(\ol\mu)}{\delta \mu^I(z^I)} = \sum_J \int_{\RR^{d_J}} k_{IJ}(z^I, z^J) \,\d\mu^J(z^J)$ is linear in $\ol\mu$ for each $I$, so the mappings $R_I$ in \autoref{thm:MFG:gen} are identically zero, and so our finer-grained result \autoref{thm:MFG:quadr} applies.

\subsection{Proof of \autoref{thm:MFG:gen} and \autoref{thm:MFG:quadr}} \label{subsec:MFG:pfs}

The proofs of \autoref{thm:MFG:gen} and \autoref{thm:MFG:quadr} follow essentially the same steps as the corresponding results for MFLD, \autoref{thm:MFLD_L2:localL2:gen_simple} and \autoref{thm:MFLD_L2:localL2:quad}.

Introduce the (unbounded) operators
$\nabla_I^* = -\frac{1}{\nu^I} \nabla \cdot (\nu^I \bullet)$ from $\bm\LLL^2_{\nu^I}$ to $\LLL^2_{\nu^I}$ and
$L_I = -\frac{1}{\nu^I} \nabla \cdot (\nu^I \nabla \bullet) = \nabla^*_I \nabla$ over $\LLL^2_{\nu^I}$,
and denote $\norm{f^I}_{H^{-1}_I}^2 = \innerprod{f^I}{L_I^{-1} f^I}_{\nu^I}$ for all $f^I \in \LLL^2_{\nu^I} \cap \{1\}^\perp$.

\begin{proof}
	By definition of $R_I$, the dynamics \eqref{eq:MFG:dyn_mu} rewrites, in terms of 
	$f_t^I = \frac{\d\mu_t^I}{\d\nu^I}-1 \in \LLL^2_{\nu^I} \cap \{1\}^\perp$,
	\begin{align*}
		\partial_t \mu_t^I &= \nabla \cdot (\mu_t^I \nabla V_I[\ol\mu_t]) + \tau \Delta \mu_t^I \\
		&= \tau \nabla \cdot \left( \mu_t^I \nabla \log \frac{\mu_t^I}{\nu^I} \right)
		+ \sum_J \nabla \cdot \left( \mu_t^I \nabla \int_{\RR^{d_J}} k_{IJ}(\cdot, z^J) \,\d(\mu_t^J - \nu^J) \right)
		+ \nabla \cdot (\mu_t^I \nabla R_I[\ol\mu)t]) \\
		\partial_t f_t^I &= -\tau L_I f_t^I
		- \sum_J \nabla_I^* \left( (f_t^I+1) \nabla K_{IJ} f_t^J \right)
		- \nabla_I^* \left( (f_t^I+1) \nabla R_I[\ol\mu_t] \right) \\
		&= -\tau L_I f_t^I
		- \sum_J L_I K_{IJ} f_t^J
		- \sum_J \nabla_I^* \left( f_t^I \nabla K_{IJ} f_t^J \right)
		- \nabla_I^* \left( (f_t^I+1) \nabla R_I[\ol\mu_t] \right)
	\end{align*}
	where $K_{IJ}: \LLL^2_{\nu^J} \to \LLL^2_{\nu^I}$ is the operator defined by
	$(K_{IJ} f^J)(z^I) = \int_{\RR^{d_J}} k_{IJ}(z^J, z^J) \, f^J(z^J) \d\nu^J(z^J)$.
	Consequently,
	\begin{multline*}
		\frac{d}{dt} \norm{f_t^I}_{H^{-1}_I}^2
		= 2 \innerprod{f_t^I}{L_I^{-1} \partial_t f_t^I}_{\nu^I}
		= -2 \tau \norm{f_t^I}_{\nu^I}^2
		- 2 \sum_J \innerprod{f_t^I}{K_{IJ} f_t^J}_{\nu^I} \\
		- 2 \sum_J \innerprod{f_t^I}{L_I^{-1} \nabla_I^* \left( f_t^I \nabla K_{IJ} f_t^J \right)}_{\nu^I}
		- 2 \innerprod{f_t^I}{L_I^{-1} \nabla_I^* \left( (f_t^I+1) \nabla R_I[\ol\mu_t] \right)}_{\nu^I}.
	\end{multline*}
	For the third term, note that for each $J$, by a Cauchy-Schwarz inequality w.r.t.\ $\innerprod{\cdot}{L_I^{-1} \cdot}_{\nu^I}$,
	\begin{align*}
		\abs{\innerprod{f_t^I}{L_I^{-1} \nabla_I^* \left( f_t^I \nabla K_{IJ} f_t^J \right)}}
		&\leq \norm{f_t^I}_{H^{-1}_I}
		\cdot \sqrt{\innerprod{f^I_t \nabla K_{IJ} f_t^J}{\nabla L^{-1} \nabla^*_I \left( f^I_t \nabla K_{IJ} f_t^J \right)}_{\nu^I}} \\
		&\leq \norm{f_t^I}_{H^{-1}_I}
		\cdot \norm{f_t^I \nabla K_{IJ} f_t^J}_{\nu^I} \\
		&\leq \norm{f_t^I}_{H^{-1}_I}
		\cdot \norm{f_t^I}_{\nu^I} \cdot \sup_{\RR^{d_I}} \norm{\nabla K_{IJ} f_t^J} \\
		&\leq \norm{f_t^I}_{H^{-1}_I}
		\cdot \norm{f_t^I}_{\nu^I} \cdot M_{11} \norm{f_t^J}_{H^{-1}_J}
	\end{align*}
	where in the second inequality we used that $0 \preceq \nabla_I L_I^{-1} \nabla^*_I \preceq \id_{\bm\LLL^2_{\nu^I}}$ as the orthogonal projector in $\bm\LLL^2_{\nu^I}$ onto the subspace consisting of gradient fields,
	and the last inequality follows from \autoref{lm:MFLD_L2:localL2:estim_nablaKf_infty}.
	Likewise the fourth term can be bounded by noting that
	\begin{align*}
		\abs{\innerprod{f_t^I}{L_I^{-1} \nabla_I^* \left( (f_t^I+1) \nabla R_I[\ol\mu_t] \right)}_{\nu^I}}
		&\leq \norm{f_t^I}_{H^{-1}_I}
		\cdot \norm{(f_t^I+1) \nabla R_I[\ol\mu_t]}_{\nu^I} \\
		&\leq \norm{f_t^I}_{H^{-1}_I}
		\cdot (1 + \norm{f_t^I}_{\nu^I})
		\cdot \sup_{\RR^{d_I}} \norm{\nabla R_I[\ol\mu_t]} \\
		&\leq \norm{f_t^I}_{H^{-1}_I}
		\cdot (1 + \norm{f_t^I}_{\nu^I})
		\cdot \frac12 (M_{12} + N M_{111}) \sum_J W_2^2(\mu^J_t, \nu^J_t) \\
		&\leq \norm{f_t^I}_{H^{-1}_I}
		\cdot (1 + \norm{f_t^I}_{\nu^I})
		\cdot M_R \sum_J \chisq{\mu^J_t}{\nu^J_t}
	\end{align*}
	where $M_R = (M_{12} + N M_{111}) c_{\PI}^{-1}$,
	by \autoref{lm:MFG:pfs:estim_nablaR_infty} below
	and by the estimate from \cite{liu2020poincare}.
	By summing these inequalities,
	we thus have
	\begin{multline*}
		\frac{d}{dt} \sum_I \norm{f_t^I}_{H^{-1}_I}^2
		\leq -2 \tau \sum_I \norm{f_t^I}_{\nu^I}^2 
		- 2 \sum_{I,J} \innerprod{f_t^I}{K_{IJ} f_t^J}_{\nu^I} \\
		 + 2 \sum_{I,J} \norm{f_t^I}_{H^{-1}_I}
		\cdot \norm{f_t^I}_{\nu^I} \cdot M_{11} \norm{f_t}_{H^{-1}_J}
		+ 2 \sum_I \norm{f_t^I}_{H^{-1}_I}
		\cdot (1 + \norm{f_t^I}_{\nu^I})
		\cdot M_R \sum_J \norm{f^J_t}_{\nu^J}^2
	\end{multline*}
	and the first two terms are bounded by 
	$-2 (\tau-\tau_0) \sum_I \norm{f_t^I}_{\nu^I}^2$
	by assumption \eqref{eq:MFG:assum_kIJ}.
	
	We also have the rough estimates for the $\LLL^2$ norms
	\begin{align*}
		\frac{d}{dt} \norm{f_t^I}_{\nu^I}^2
		&= -2\tau \norm{\nabla f_t^I}_{\nu^I}^2
		- 2 \sum_J \innerprod{\nabla f_t^I}{(f_t^I+1) \nabla K_{IJ} f_t^J}_{\nu^I}
		- 2 \innerprod{\nabla f_t^I}{(f_t^I+1) \nabla R_I[\ol\mu_t]}_{\nu^I} \\
		&\leq -2\tau \norm{\nabla f_t^I}_{\nu^I}^2
		+ 2 \norm{\nabla f_t^I}_{\nu^I} \cdot (1 + \norm{f_t^I}_{\nu^I}) \cdot 
		\left[ \sum_J \sup_{\RR^{d_I}} \norm{\nabla K_{IJ} f_t^J} + \sup_{\RR^{d_I}} \norm{\nabla R_I[\ol\mu_t]} \right] \\
		&\leq -2\tau \norm{\nabla f_t^I}_{\nu^I}^2
		+ 2 \norm{\nabla f_t^I}_{\nu^I} \cdot (1 + \norm{f_t^I}_{\nu^I}) \cdot 
		\sum_J \left( M_{11} \norm{f_t^J}_{H^{-1}_J} + M_R \norm{f_t^J}_{\nu^J}^2 \right).
	\end{align*}
	
	Denoting for concision
	\begin{align} \label{eq:MFG:pfs:zabZAB}
		z_t^I &= \norm{f_t^I}_{H^{-1}_I}^2,
		&
		a_t^I &= \norm{f_t^I}_{\nu^I}^2,
		&
		b_t^I &= \norm{\nabla f_t^I}_{\nu^I}^2, \\
		\text{and}~~~~
		Z_t &= \sum_I z_t^I,
		&
		A_t &= \sum_I a_t^I,
		&
		B_t &= \sum_I b_t^I,
	\end{align}
	the two above estimates rewrite
	\begin{align*}
		\dot{Z}_t
		&\leq -2(\tau-\tau_0) A_t
		+ 2 M_{11} \left( \sum_I \sqrt{z_t^I ~ a_t^I} \right) \left( \sum_J \sqrt{z_t^J} \right)
		+ 2 M_R \left( \sum_I \sqrt{z_t^I} \left(1 + \sqrt{a_t^I} \right) \right) A_t \\
		\dot{A}_t
		&\leq -2\tau B_t
		+ 2 \left( \sum_I \sqrt{b_t^I} \left( 1 + \sqrt{a_t^I} \right) \right)
		\left( M_{11} \sum_J \sqrt{z_t^J} + M_R A_t \right).
	\end{align*}
	Using Cauchy-Schwarz inequalities in $\RR^N$, we find that
	\begin{align*}
		\dot{Z}_t
		&\leq -2(\tau-\tau_0) A_t
		+ 2 M_{11} \sqrt{N} Z_t \sqrt{A_t} 
		+ 2 M_R \sqrt{Z_t} A_t \cdot \sqrt{2} (\sqrt{N} + \sqrt{A_t}) \\
		&\leq -2(\tau-\tau_0) A_t
		+ 2 \sqrt{N} M_{11} \cdot Z_t \sqrt{A_t} 
		+ 2 \sqrt{2N} M_R \cdot \sqrt{Z_t} \, A_t (1 + \sqrt{A_t}) \\
		\dot{A}_t
		&\leq -2\tau B_t + 2 \sqrt{B_t} \cdot \sqrt{2} (\sqrt{N} + \sqrt{A_t}) \cdot \left( M_{11} \sqrt{N} \sqrt{Z_t} + M_R A_t \right) \\
		&\leq -2\tau B_t + 2 \sqrt{B_t} (1 + \sqrt{A_t}) \left( \sqrt{2} \, N M_{11} \cdot \sqrt{Z_t} + \sqrt{2} \, N M_R \cdot A_t \right).
	\end{align*}
	This is the same system of inequalities as \eqref{eq:MFLD_L2:localL2:zab_system_gen} in the proof of \autoref{thm:MFLD_L2:localL2:gen_full},
	except that ``$M_{11}$'' is replaced by $\sqrt{2} N M_{11}$ and ``$M_R$'' by $\sqrt{2} N M_R$.
	We also still have that
	$c_{\PI} Z_t \leq A_t \leq c_{\PI}^{-1} B_t$.
	So we can follow the same steps and obtain the same conclusions as in that proof. 
	Namely we get that for any $0 < \eps \leq \frac18$,
	denoting $M = \sqrt{2} N (M_{11} \vee M_R)$
	and $\gamma = \frac{\tau (\tau-\tau_0) c_{\PI} \eps^2}{128 \, M^2}$,
	if $A_0 \leq (\gamma + 1/c_{\PI})^{-1} \cdot 
	\left( 2^{-20} (\tau-\tau_0)^4 c_{\PI}^3 \eps^4 M^{-4} \wedge 1 \right)$,
	then
	$A_t \leq C e^{-2(\tau-\tau_0) c_{\PI} (1-\eps) t} A_0$
	where $C = 1 + \frac{128 M^2}{\tau (\tau-\tau_0) c_{\PI}^2 \eps^2}$.
	This implies the statement of \autoref{thm:MFG:gen}.
\end{proof}

\begin{lemma} \label{lm:MFG:pfs:estim_nablaR_infty}
	Under the conditions of \autoref{thm:MFG:gen}, for any $\ol\mu \in \ol\PPP_2$ and any $I \in \{1, ..., N\}$,
	\begin{equation*}
		\sup_{\RR^{d_I}} \norm{\nabla R_I[\ol\mu]}
		\leq \frac12 (M_{12} + N M_{111}) \sum_J W_2^2(\mu^J, \nu^J).
	\end{equation*}
\end{lemma}

\begin{proof}
	The proof is very similar to that of \autoref{lm:MFLD_L2:localL2:estim_nablaR'_infty}.
	
	Fix $I \in \{1, ..., N\}$.
	Since $\nu^J \in \PPP_2^{\AC}(\RR^{d_J})$ for each $J$, by Brenier's theorem, there exist optimal transport maps $T^J$ such that $\mu^J = T^J_\sharp \nu^J$.
	Let $\Psi^J = T^J - \id$ and $\mu^J_s = (\id + s \Psi^J)_\sharp \nu^J$ for all $0 \leq s \leq 1$, so that $\mu^J_0 = \nu^J$ and $\mu^J_1 = \mu^J$.
	Also denote $\Psi^J_s = \Psi^J \circ (\id + s \Psi^J)^{-1}$ for all $0 \leq s < 1$, so that $\partial_s \mu^J_s = -\nabla \cdot (\mu^J_s \Psi^J_s)$,
	and note that $\int \norm{\Psi^J_s}^2 \d\mu^J_s = W_2^2(\mu^J, \nu^J)$ for all $s$.
	
	Fix $z^I \in \RR^{d_I}$.
	By explicit computations, one has that
	\begin{align*}
		\frac{d}{ds} \nabla R_I[\ol\mu_s](z^I) 
		&= \nabla_{z^I} \sum_J \int_{\RR^{d_J}} \d\mu^J_s(z^J)~ \Psi^J_s(z^J)^\top \nabla_{z^J} \frac{\delta R_I[\ol\mu_s](z^I)}{\delta \mu^J(z^J)} \\
		\frac{d^2}{ds^2} \nabla R_I[\ol\mu_s](x)
		&= \nabla_{z^I} \sum_J \int_{\RR^{d_J}} \d\mu^J_s(z^J)~ \Psi^J_s(z^J)^\top \nabla_{z^J}^2 \frac{\delta R_I[\ol\mu_s](z^I)}{\delta \mu^J(z^J)} \, \Psi^J_s(z^J) \\
		& \hspace{-9pt} 
		+ \nabla_{z^I} \sum_{J,K} \iint_{\RR^{d_J} \times \RR^{d_K}} \d\mu^J_s(z^J) \d\mu^K_s(z^K)~ \Psi^J_s(z^J)^\top \nabla_{z^J} \nabla_{z^K} \frac{\delta^2 R_I[\ol\mu_s](z^I)}{\delta \mu^J(z^J) \delta \mu^K(z^K)} \, \Psi^K_s(z^K).
	\end{align*}
	In particular, for all $s \in [0,1)$,
	\begingroup \allowdisplaybreaks
	\begin{align*}
		\norm{ \frac{d^2}{ds^2} \nabla R_I[\ol\mu_s](z^I) }
		&\leq \sum_J \int_{\RR^{d_J}}
		\d\mu^J_s(z^J)~ \norm{\Psi^J_s(z^J)}^2
		\cdot \norm{ \nabla_{z^I} \nabla_{z^J}^2 \frac{\delta R_I[\ol\mu_s](z^I)}{\delta \mu^J(z^J)} }_{\mathrm{op}} \\*
		&~~~~ + \sum_{J,K} \iint_{\RR^{d_J} \times \RR^{d_K}} 
		\d\mu^J_s(z^J) \d\mu^K_s(z^K)~
		\norm{\Psi^J_s(z^J)} \cdot \norm{\Psi^K_s(z^K)} \\*
		&\hspace{18em}
		\cdot \norm{ \nabla_{z^I} \nabla_{z^J} \nabla_{z^K} \frac{\delta^2 R_I[\ol\mu_s](z^I)}{\delta \mu^J(z^J) \delta \mu^K(z^K)} }_{\mathrm{op}} \\
		&\leq M_{12} \sum_J W_2^2(\mu^J, \nu^J)
		+ M_{111} \sum_{J,K} W_2(\mu^J, \nu^J) W_2(\mu^K, \nu^K) \\
		&\leq (M_{12} + N M_{111}) \sum_J W_2^2(\mu^J, \nu^J).
	\end{align*}
	\endgroup
	So by a second-order Taylor expansion of $s \mapsto \nabla R[\ol\mu^s](z^I)$,
	since $R_I[\ol\nu]=\cst$ on $\RR^{d_I}$ and $\frac{\delta R_I[\ol\nu]}{\delta \nu^J} = \cst$ on $\RR^{d_I} \times \RR^{d_J}$,
	\begin{equation*}
		\norm{\nabla R_I[\ol\mu](z^I)}
		= \norm{0 + 0 + \int_0^1 \d s ~(1-s)
		\frac{d^2}{ds^2} \nabla R_I[\ol\mu_s](z^I)}
		\leq \frac12 (M_{12} + N M_{111}) \sum_J W_2^2(\mu^J, \nu^J),
	\end{equation*}
	as announced.
\end{proof}

\begin{proof}[Proof of \autoref{thm:MFG:quadr}]
	The beginning of the proof is identical to the one for the general case, specialized to $R = 0$.
	One can check that, denoting
	$z_t^I, a_t^I, b_t^I, Z_t, A_t, B_t$ as in \eqref{eq:MFG:pfs:zabZAB},
	we have
	\begin{align*}
		\dot{Z}_t
		&\leq -2(\tau-\tau_0) A_t
		+ 2 M_{11} \left( \sum_I \sqrt{z_t^I ~ a_t^I} \right) \left( \sum_J \sqrt{z_t^J} \right) \\
		\dot{A}_t
		&\leq -2\tau B_t
		+ 2 \left( \sum_I \sqrt{b_t^I} \left( 1 + \sqrt{a_t^I} \right) \right)
		M_{11} \sum_J \sqrt{z_t^J}.
	\end{align*}
	Using Cauchy-Schwarz inequalities in $\RR^N$, we find that
	\begin{align*}
		\dot{Z}_t
		&\leq -2(\tau-\tau_0) A_t
		+ 2 \sqrt{N} M_{11} \cdot Z_t \sqrt{A_t} \\
		\dot{A}_t
		&\leq -2\tau B_t + 2 \sqrt{2} \, N M_{11} \cdot \sqrt{B_t} (1 + \sqrt{A_t}) \cdot \sqrt{Z_t}.
	\end{align*}
	This is the same system of inequalities as \eqref{eq:MFLD_L2:localL2:zab_system_quad} in the proof of \autoref{thm:MFLD_L2:localL2:quad}, except that ``$M_{11}$'' is replaced by $\sqrt{2} N M_{11}$. 
	We also still have that
	$c_{\PI} Z_t \leq A_t \leq c_{\PI}^{-1} B_t$.
	So we can follow the same steps and obtain the same conclusions as in that proof.
	Namely we get that for any $0<\eps<1$,
	if $Z_0 \leq \frac{(\tau-\tau_0)^2 c_{\PI} (\eps/2)^2}{2 N^2 M_{11}^2}$,
	which can be ensured by assuming 
	$A_0 \leq \frac{(\tau-\tau_0)^2 c_{\PI}^2 (\eps/2)^2}{2 N^2 M_{11}^2}$,
	then
	$Z_t \leq e^{-2 (\tau-\tau_0) c_{\PI} (1-\eps/2) t} Z_0$
	and $A_t \leq \left( 1 + \frac{2 N^2 M_{11}^2}{\tau (\tau-\tau_0) c_{\PI}^2 \eps^2} \right) e^{-2(\tau-\tau_0) c_{\PI} (1-\eps) t} A_0$.
	This is precisely the announced statement of \autoref{thm:MFG:quadr}.
\end{proof}

\section{Conclusion} \label{sec:ccl}

We analyzed the local convergence of mean-field Langevin type dynamics, in $\chi^2$-divergence, under linear convexity type assumptions.
For the most common setting, the Wasserstein gradient flow of $F_\tau = F + \tau H$ with $F$ displacement-smooth,
under (relaxations of) linear convexity of $F$,
we showed local and long-time convergence with a tight rate estimate of $2 \tau$ times the Poincar\'e inequality constant of the stationary distribution.
Contrary to prior quantitative convergence analyses, we do not require $F$ displacement-convex, nor $\tau$ large, nor a uniform LSI condition.
We then adapted our analysis to the mean-field Langevin descent-ascent dynamics, which is the Wasserstein gradient descent-ascent flow of $(\mu^x, \mu^y) \mapsto \iint_{\TT^{d_x} \times \TT^{d_y}} k \,\d(\mu^x \otimes \mu^y) + \tau H(\mu^x) - \tau H(\mu^y)$.
Its long-time convergence behavior (for nonconvex-nonconcave $k$ and without assuming $\tau$ large) is an open problem, which our result reduces to whether the dynamics eventually passes through a $\chi^2$ neighborhood of the equilibrium.
Finally, we applied our analysis to the general setting of multi-species flows with diffusion, assuming (relaxations of) a linear monotonicity condition which generalizes Lasry-Lions monotonicity.

We end by highlighting some open directions for future work.
\begin{itemize}
	\item Our results provide lower estimates on the asymptotic rate of exponential convergence during the final phase of the aforementioned dynamics.
	As mentioned in \autoref{subsubsec:MFLD_L2:localL2:exact_rate} and \autoref{rk:MFL-DA:minReSp}, we expect that the exact asymptotic rate can be characterized
	as the spectral abscissa of a certain unbounded operator.
	However, rigorously confirming this intuition is left as a technical question.
	\item In our main assumption, \autoref{assum:MFLD_A}, we required the ``local weak linear convexity'' parameter $\tau_0$ to lie in $[0, \tau)$.
	As mentioned in
	\autoref{subsubsec:MFLD_localL2:range_of_tau0}, restricting $\tau_0$ to $[0,\tau]$ is without loss of generality,
	but our analysis does not cover the critical case $\tau = \tau_0$.
	The exploration of this case is left for future work.
\end{itemize}

\printbibliography
\addcontentsline{toc}{section}{\refname} 

\ifextended%
\newpage
\appendix
\phantomsection
\addcontentsline{toc}{section}{APPENDIX}

\section[Otto calculus interpretation of PI]{Otto calculus interpretation of Poincar\'e inequality} \label{sec:apx_OL}

This appendix is devoted to remarks and formal computations that can provide intuition for our results, in the spirit of Otto calculus.
It can be read independently of the rest of the paper.

Throughout this appendix, let $\nu \propto e^{-V(x)} \d x \in \PPP_2(\RR^d)$ with $V$ of class $\CCC^2$ and satisfying appropriate growth conditions. 
Denote the generator of the associated overdamped Langevin dynamics by
\begin{equation*}
	L f = -\Delta f + \nabla V \cdot \nabla f
	= -\frac{1}{\nu} \nabla \cdot (\nu \nabla f).
\end{equation*}
That is, the time-marginals $\mu_t$ of the overdamped Langevin dynamics evolve according to $\partial_t f_t = -L f_t$ with $\mu_t = f_t \nu$.
Note that for any sufficiently regular $f, g: \RR^d \to \RR$, we have 
$\innerprod{f}{Lg}_\nu = \innerprod{\nabla f}{\nabla g}_\nu$ by integration by parts, where $\innerprod{f}{g}_\nu = \int f \cdot g \,\d\nu$. 
So we can write formally
$L = \nabla^* \nabla$
where $*$ denotes adjoints in~$\LLL^2_\nu$.
Accordingly, we define an operator $\nabla^*$, mapping vector fields to scalar fields, by
\begin{equation*}
	\nabla^* \Phi = -\nabla \cdot \Phi + \nabla V \cdot \Phi = -\frac1\nu \nabla \cdot (\nu \Phi).
\end{equation*}
In the context of Stein's method in statistics,
$-\nabla^*$ is sometimes called the Langevin-Stein operator.

Recall that we say $\nu$ satisfies Poincar\'e inequality (PI) with constant $c_{\PI}$ if
\begin{equation*}
	\forall f ~\text{s.t.} \int f \,\d\nu=0,~~
	c_{\PI} \int \abs{f}^2 \d\nu \leq \int \norm{\nabla f}^2 \d\nu.
\end{equation*}
Equivalently, $L \succeq c_{\PI} \id$ in $\LLL^2_\nu \cap \{1\}^\perp$. In other words, PI is equivalent to the symmetric operator $L$ having a positive spectral gap.

\subsection{\texorpdfstring{$\nabla^*$}{nabla*} as a formal bijection between Wasserstein and Fisher-Rao tangent spaces}

In the Wasserstein geometry, the tangent space of $\PPP_2(\RR^d)$ at a measure $\nu$ is
$
	T_\nu \PPP_2(\RR^d) = \overline{\nabla \CCC^\infty_c}^{\,\bm\LLL^2_\nu}
$
i.e., the subspace of $\bm\LLL^2_\nu$ consisting of gradient fields.
In the Fisher-Rao geometry, the tangent space at $\nu$ is
$\left\{ s \in \MMM(\RR^d); \int \d s = 0 ~\text{and}~ \int \nu^{-1} \abs{s}^2 \d x = \int \abs{\frac{\d s}{\d \nu}}^2 \d\nu < \infty \right\}$,
and the dual tangent space is $\LLL^2_\nu \cap \{1\}^\perp$, with the primal-dual correspondence $s = f \nu$.

Morally, if $\nu$ satisfies PI,
then $\nabla^*$ establishes a bijection from $T_\nu \PPP_2(\RR^d)$ to $\LLL^2_\nu \cap \{1\}^\perp$, with invertibility being a consequence of the spectral gap property of $L = \nabla^* \nabla$:
formally $(\nabla^*)^{-1} = \nabla L^{-1}$.
However $\nabla^*$ is not well-defined as an operator between these two spaces,
due to regularity issues.%
\footnote{A possible formalization via Fr\'echet spaces is given in \cite[Sec.~3.3]{lafferty1988density} in the case of $\nu$ having a $\CCC^\infty$-smooth and compactly supported density w.r.t.\ the Lebesgue measure.}
Yet this idea is quite useful as a heuristic.
For example it intervened implicitly in our proof of \autoref{prop:MFLD_L2:localL2:tau0=tau_unstable}, where effectively, in the notations of that Proposition, we used $\nabla L^{-1} f \in T_\nu \PPP_2(\RR^d)$ as a descent direction for $F_\tau$ in the Wasserstein geometry --- which was the right geometry to work in, given the available regularity assumptions on $F$.

This idea also underlies the proof of the following standard fact \cite[Lemma~1]{barthe2013invariances} (see also \cite[Sec.~2.2.2 ``The Brascamp-Lieb Inequality'']{chewi2024book}),
which we present as an illustration.
This lemma will also be reused in the next subsection.

\begin{lemma} \label{lm:apx_OL:hormanderL2}
	Suppose $V = -\log \frac{\d\nu}{\d x}$ is $\CCC^\infty$-smooth.
	Consider any $\CCC^\infty$-smooth $A: \RR^d \to \RR^{d \times d}$ such that, for all $x$, the matrix $A(x)$ is symmetric positive-definite.
	Then the following statements are equivalent:
	\begin{enumerate}[label=(\roman*)]
		\item For any gradient field $\Phi \in \CCC^\infty_c(\RR^d, \RR^d)$,~
		$\int (\nabla^* \Phi)^2 \d\nu \geq \int \Phi^\top A \Phi \,\d\nu$.
		\item For any $f \in \CCC^\infty_c(\RR^d)$ such that $\int f \, \d\nu=0$,~
		$\int \abs{f}^2 \d\nu \leq \int \nabla f^\top A^{-1} \nabla f \,\d\nu$.
	\end{enumerate}
\end{lemma}
\vspace{0.5em}

\begin{proof}
	In this proof, for any pre-Hilbert spaces $E, F$, we call two operators $B: E \to F$ and $B^*: F \to E$ adjoint to each other if it holds $\forall e, f,~ \innerprod{Be}{f} = \innerprod{e}{B^* f}$.
	
	Let $G = \left\{ \Phi \in \CCC^\infty_c(\RR^d, \RR^d) ~\text{s.t.\ $\Phi$ is a gradient field} \right\}$
	and $S = \left\{ f \in \CCC^\infty_c(\RR^d) ~\text{s.t.} \int f\,\d\nu = 0 \right\}$,
	viewed as pre-Hilbert spaces equipped with the inner products $\innerprod{\Phi_1}{\Phi_2}_\nu = \int \Phi_1^\top \Phi_2 \,\d\nu$ and $\innerprod{f_1}{f_2}_\nu = \int f_1 f_2 \,\d\nu$ respectively.
	Since $V$ is $\CCC^\infty$, then $\nabla^*$ maps $G$ into $S$.
	Moreover in the other direction, the inverse by $L$ of any $f \in S$ is well-defined and belongs to $\CCC^\infty_c$ \cite[Thm.~3 of Sec.~6.3]{evans2010pde}, so $\nabla^*$ is invertible from $G$ to $S$ with $(\nabla^*)^{-1} = \nabla L^{-1}$.
	Also note that $\nabla: S \to G$ is adjoint to $\nabla^*$, and is invertible with $\nabla^{-1} = L^{-1} \nabla^*$.
	
	Denote with abuse of notation by $A: G \to G$ the operator $\Phi \mapsto A(\cdot) \Phi(\cdot)$, and likewise for $A^{-1}, A^\half, A^{-\half}$.
	Statement $(i)$ can be reformulated as $(\nabla^*)^* \nabla^* = \nabla \nabla^* \succeq A$ in $G$, 
	and statement $(ii)$ as $\nabla^* A^{-1} \nabla \succeq \id$ in $S$.
	So the claimed equivalence can be reformulated as
	\begin{equation*}
		(A^{-\half} \nabla) (A^{-\half} \nabla)^* \succeq \id_{\tG} 
		\iff 
		(A^{-\half} \nabla)^* (A^{-\half} \nabla) \succeq \id_S
	\end{equation*}
	where $\tG = A^{\half} G$,
	and $A^{-\half} \nabla: S \to \tG$ 
	and its adjoint $(A^{-\half} \nabla)^* = \nabla^* A^{-\half}$
	are invertible.
	
	Now it is a general fact that for any pre-Hilbert spaces $E, F$ and any adjoint operators $B: E \to F$, $B^*: F \to E$ that are invertible,
	we have the equivalence $B^* B \succeq \id_E \iff B B^* \succeq \id_F$.
	Indeed, if $B^* B \succeq \id_E$, then for any $f \in F$, denoting $e = B^{-1} f$, we have
	$\norm{f}^2 = 2 \innerprod{B^* f}{e} - \innerprod{Be}{Be} \leq 2 \norm{B^* f} \norm{e} - \norm{e}^2 \leq \norm{B^* f}^2$.
	The other implication follows similarly.
\end{proof}

\begin{remark}
	We believe that the assumption that $V$ is $\CCC^\infty$-smooth can be relaxed to $\CCC^2$-smoothness, but we were unable to find a suitable reference for the equivalence of $(i)$ and $(ii)$ under this setting. For the direction $(i) \implies (ii)$, this is precisely the content of \cite[Lemma~1]{barthe2013invariances}.
\end{remark}

\begin{remark}
	One may ask whether the statements $(i), (ii)$ are also equivalent to
	{\itshape
		\begin{itemize}
			\item[(i')] For any \emph{vector} field $\Phi \in \CCC^\infty_c(\RR^d, \RR^d)$,~
			$\int (\nabla^* \Phi)^2 \d\nu \geq \int \Phi^\top A \Phi \,\d\nu$.
		\end{itemize}
	}
	\noindent
	This is not the case, as one can see by considering the orthogonal decomposition $\bm\LLL^2_\nu = \Ima \nabla \oplus \Ker \nabla^*$.
	Indeed for $\Phi \in \Ker \nabla^* \setminus \{0\}$, the left-hand side is zero and the right-hand side is positive.%
	\footnote{To show the existence of a $\CCC^\infty$-smooth and compactly supported element of $\Ker \nabla^*$, it suffices to take $\Phi = \nu^{-1} w$ for a $\CCC^\infty$-smooth, compactly supported, and divergence-free vector field $w$.
	For example one can take 
	$w(x) = \xi(\norm{x}) A x$ for a bump function $\xi: \RR \to \RR$ and an antisymmetric matrix $A$.}
\end{remark}


\subsection{PI as non-degeneracy of the Hessian of \texorpdfstring{$\KLdiv{\cdot}{\nu}$}{KL(.|nu)} at optimum}

The seminal work of \cite{otto2000generalization} promoted the view that functional inequalities such as PI or LSI can be interpreted, via Otto calculus, as conditions on $\FFF(\mu) = \KLdiv{\mu}{\nu}$. 
Let us summarize the correspondences:
\begin{itemize}[itemsep=0pt]
	\item LSI for $\nu$ $\leftrightarrow$ Polyak-Lojasiewicz inequality for $\FFF$, with the same constant;
	\item Talagrand T2 inequality for $\nu$ $\leftrightarrow$ quadratic growth property for $\FFF$, with the same constant;
	\item HWI inequality for $\nu$ $\leftrightarrow$ (a consequence of) star-strong convexity for $\FFF$ \cite[Def~5.1]{guille2021study};
	\item PI for $\nu$ $\leftrightarrow$ Hessian of $\FFF$ at optimum lower-bounded by $c_{\PI}$;
	\item Curvature-Dimension condition $\mathrm{CD}(\alpha, \infty)$ $\leftrightarrow$ Hessian of $\FFF$ uniformly lower-bounded by $\alpha$.
	The direct implication is classical, and the converse follows from the work of \cite{erbar2015equivalence}, as explained in \cite[Eq.~(14)]{clerc2023variational}.
\end{itemize}

Here ``Hessian'' refers to the following notion, where we stress the requirement that $\Phi$ is a gradient field.
\begin{definition}[{\cite[Chapter~15]{villani2008optimal}}]
	For $\FFF: \PPP_2(\RR^d) \to \RR$, the Wasserstein Hessian of $\FFF$ at an absolutely continuous measure $\mu$ is (if it exists) the symmetric bilinear form $\Hess_\mu \FFF$ over $T_\mu \PPP_2(\RR^d)$ such that
	$
		\forall \Phi \in T_\mu \PPP_2(\RR^d),~
		\Hess_\mu \FFF(\Phi, \Phi) = \restr{\frac{d^2}{ds^2}}{s=0} \FFF((\id + s \Phi)_\sharp \mu)
	$.
	
	For $\FFF(\mu) = \KLdiv{\mu}{\nu}$, the Wasserstein Hessian can be computed formally to be 
	\begin{equation*}
		\Hess_\mu \KLdiv{\cdot}{\nu}(\Phi, \Phi) = \int \Gamma_2(\Phi, \Phi) \,\d\mu
		~~~\text{where}~~~
		\Gamma_2(\Phi, \Phi) = \trace((\nabla \Phi)^2) + \Phi^\top \nabla^2 V \Phi.
	\end{equation*}
	Note that this quantity is finite only for $\Phi$ smooth enough ($\Phi \in T_\mu \PPP_2(\RR^d)$ is not sufficient).
\end{definition}

All of these correspondences are classical \cite{otto2000generalization}, except perhaps for the one for PI, which we now show.
The proof hinges on the following identity, showing that the Wasserstein Hessian of $\KLdiv{\cdot}{\nu}$ at $\nu$ itself has two different nice forms.
It appeared, e.g.,\ in \cite[Eq.~(10)]{barthe2013invariances}, for $\Phi$ being a gradient field.
The fact that it actually holds for $\Phi$ being any vector field appears to be new, but note that
$\trace((\nabla \Phi)^2)$ may be negative if $\Phi$ is not a gradient field.

\begin{lemma} \label{lm:apx_OL:hessopt_two_forms}
	For $\nu \propto e^{-V(x)} \d x$ and $\FFF(\mu) = \KLdiv{\mu}{\nu}$, we have formally, for any vector field~$\Phi$,
	\begin{equation*}
		\Hess_\nu \FFF(\Phi, \Phi)
		= \int \Gamma_2(\Phi,\Phi) \,\d\nu
		= \int (\nabla^* \Phi)^2 \,\d\nu
	\end{equation*}
	where $\Gamma_2(\Phi, \Phi) = \trace((\nabla \Phi)^2) + \Phi^\top \nabla^2 V \Phi$
	and $\nabla^* \Phi = -\nabla \cdot \Phi + \nabla V \cdot \Phi = -\frac1\nu \nabla \cdot (\nu \Phi)$.
\end{lemma}

\begin{proof}
	This can be shown by explicit computations using integration by part and Bochner's formula. 
	
	Alternatively, here is a novel (to our knowledge) interpolation-based proof. Consider any curve $(\mu_t)_t$ with $\partial_t \mu_t = -\nabla \cdot (\mu_t \Phi_t)$ and $\mu_0 = \nu$ and $\Phi_0 = \Phi$, \emph{with no assumption on $(\Phi_t)_t$}. 
	Then we can compute $\restr{\frac{d^2}{dt^2} \KLdiv{\mu_t}{\nu}}{t=0}$ in two different ways. From a Wasserstein geometry perspective we obtain
	\cite{wang2025higher}
	\begin{equation*}
		\restr{\frac{d^2}{dt^2} \FFF(\mu_t)}{t=0}
		= \Hess_\nu \FFF(\Phi_0, \Phi_0)
		+ \int (\cD \Phi)_0 \cdot \nabla \FFF'[\mu_0] \,\d\mu_0
		= \int \Gamma_2(\Phi, \Phi) \,\d\nu + 0
	\end{equation*}
	since $\mu_0 = \nu$, where $(\cD \Phi)_t = \partial_t \Phi_t + (\nabla \Phi_t) \Phi_t$ is the Wasserstein acceleration. 
	From a measure-space perspective we obtain
	\begin{align*}
		\restr{\frac{d^2}{dt^2} \KLdiv{\mu_t}{\nu}}{t=0}
		&= \innerprod{\restr{\partial_t \mu_t}{t=0}}{\bm{H}_{\mu_0} \restr{\partial_t \mu_t}{t=0}}
		+ \innerprod{\restr{\partial_{tt}^2 \mu_t}{t=0}}{\log \frac{\mu_0}{\nu}} \\
		&= \int \frac{1}{\mu_0} \left( \restr{\partial_t \mu_t}{t=0} \right)^2 + 0 \\
		&= \int \frac{1}{\nu} \abs{\nabla \cdot (\nu \Phi)}^2 \\
		&= \int \abs{\frac{1}{\nu} \nabla \cdot (\nu \Phi)}^2 \d\nu
		= \int \abs{\nabla^* \Phi}^2 \d\nu
	\end{align*}
	where $\bm{H}_{\mu}$ denotes the measure-space Hessian of $\KLdiv{\cdot}{\nu}$ at $\mu$.
	The second term on the first line is zero again because of $\mu_0=\nu$.
	The lemma follows by equating the two above computations.
\end{proof}

Substituting the identity of \autoref{lm:apx_OL:hessopt_two_forms} into the equivalence of \autoref{lm:apx_OL:hormanderL2}
leads to the following.

\begin{theorem}
	For any $\nu \in \PPP_2(\RR^d)$, the following conditions are equivalent (up to regularity considerations):
	\begin{equation*}
		\forall f ~\text{s.t.}~ \int f \,\d\nu=0,~~~~
		\int \abs{f}^2 \d\nu
		\leq 
		c_{\PI}^{-1} \int \norm{\nabla f}^2 \d\nu
	\end{equation*}
	and
	\begin{equation*}
		\forall \Phi \in T_\nu \PPP_2(\RR^d) = \overline{\nabla \CCC^\infty_c}^{\,\bm\LLL^2_\nu},~~~~
		\Hess_\nu \KLdiv{\cdot}{\nu}(\Phi, \Phi) 
		\geq 
		c_{\mathrm{PI}} \int \norm{\Phi}^2 \d\nu.
	\end{equation*}
	In words, $\nu$ satisfying PI with constant $c_{\PI}$ is equivalent to the Wasserstein Hessian of $\KLdiv{\cdot}{\nu}$ at optimum being lower-bounded by $c_{\PI}$.
\end{theorem}

This equivalence was previously explicitly remarked in \cite[Lemma~35]{duncan2023geometry}. More precisely the result stated in that reference is a regularized version of the equivalence, but it can be recovered by taking a limit ($k(x, y) \to \delta(y-x)$) or by simple adaptations of the arguments.

\section{A local convergence guarantee in KL-divergence}
\label{sec:apx_MFLD_localKL}

Since LSI for $\nu$ is equivalent to global contraction in $\KLdiv{\cdot}{\nu}$ of the overdamped Langevin dynamics associated to $\nu$, and since the drift term $\nabla F'[\mu_t]$ in the MFLD PDE \eqref{eq:intro:MFLD} is morally close to $\nabla F'[\nu] = -\tau \nabla \log \nu$ when $\mu_t$ is close to a stationary measure $\nu$, then one could expect MFLD to converge locally in $\KLdiv{\cdot}{\nu}$ with a rate $2 \tau c_{\LSI}$.
This intuition was one of the initial motivations of our work, but we were unable to show it under satisfactory assumptions.

We only remark that if one is willing to make  different assumptions than standard, then a form of local convergence in $\KLdiv{\cdot}{\nu}$ can be obtained by a natural variation on the global convergence analysis from \cite{chizat2022mean,nitanda2022convex}.
This is the object of the next proposition,
which is
in part extracted from the proof of \cite[Prop.~5.1]{wang2024mean}.

\begin{proposition} \label{prop:apx_MFLD_localKL:MFLD_localKL}
	Under \autoref{assum:MFLD_A} with $\tau_0=0$,
	additionally assume $\nu$ satisfies LSI,
	and denote by
	$c_{\LSI}, c_{\mathrm{T2}}$ its optimal LSI resp.\ Talagrand inequality constants.
	Further suppose
	$M_{10} = \sup_{\tmu \in \PPP_2(\RR^d), x,y \in \RR^d}~ \norm{\nabla_x F''[\tmu](x,y)} <\infty$
	and let $C = \tau^{-1} M_{10} \sqrt{2/c_{\mathrm{T2}}}$.
	Then for any $0 < \eps \leq \frac18$,
	if
	$\KLdiv{\mu_0}{\nu} + 2 C \sqrt{\KLdiv{\mu_0}{\nu}} \leq C^{-2} \eps^2 / 10$, then the MFLD \eqref{eq:intro:MFLD} satisfies
	\begin{equation*}
		\forall t \geq 0,~
		\KLdiv{\mu_t}{\nu}
		\leq e^{-2 \tau c_{\LSI} (1-\eps) t} \left[ \KLdiv{\mu_0}{\nu} + 2 C \sqrt{\KLdiv{\mu_0}{\nu}} \right].
	\end{equation*}
%
	Alternatively, if $M_{00} = \sup_{\tmu,x,y} \abs{F''[\tmu](x,y)} < \infty$, the previous sentence holds with $C = \sqrt{2} \tau^{-1} M_{00}$.
\end{proposition}

\begin{proof}
	For any $\mu \in \PPP_2(\RR^d)$, denoting by $\hmu \propto \exp\left( -\frac1\tau F'[\mu] \right) \d x$ the associated proximal Gibbs distribution, we have that uniformly on $\RR^d$,
	\begin{align}
		\abs{\log \hmu - \log \nu + \cst} 
		= \frac1\tau \abs{F'[\mu] - F'[\nu]}
		&\leq \tau^{-1} M_{10} W_2(\mu, \nu)   \nonumber \\
		&\leq \tau^{-1} M_{10} \sqrt{2/c_{\mathrm{T2}}} \sqrt{\KLdiv{\mu}{\nu}}   \nonumber \\
		&\leq \tau^{-3/2} M_{10} \sqrt{2/c_{\mathrm{T2}}} \sqrt{F_\tau(\mu) - F_\tau(\nu)}.
	\label{eq:apx_MFLD_localKL:T2W2_pinskerTV}
	\end{align}
	Here the first inequality follows from applying \cite[Lemma~D.8]{wang2024mean} to $\mu \mapsto F'[\mu](x)$ for arbitrary fixed $x$,
	and the third inequality uses the lower half of the entropy sandwich \cite[Lemma~3.4]{chizat2022mean}.

	Consequently, by the Holley-Stroock criterion, $\hmu$ satisfies LSI with a constant at least 
	$c(\mu) = \exp\left( -\tau^{-3/2} M_{10} \sqrt{2/c_{\mathrm{T2}}} \sqrt{F_\tau(\mu) - F_\tau(\nu)} \right) c_{\LSI}$.
	On the other hand, $t \mapsto F_\tau(\mu_t) - F_\tau(\nu)$ is non-increasing since MFLD is the WGF of $F_\tau$.
	So for any $\mu_0$ such that
	\begin{equation}
		F_\tau(\mu_0) - F_\tau(\nu) \leq \tau^3 M_{10}^{-2} (c_{\mathrm{T2}}/2) \left[ -\log(1-\eps) \right]^2,
	\end{equation}
	the $\hmu_t$ satisfy LSI with constant at least $(1-\eps) c_{\LSI}$ uniformly for all $t \geq 0$, and so by applying the main result of \cite[Theorem~3.2]{chizat2022mean} we get the local convergence bound in function value
	\begin{equation}
		\forall t \geq 0,~
		\tau \KLdiv{\mu_t}{\nu} \leq F_\tau(\mu_t)-F_\tau(\nu)
		\leq e^{-2 \tau c_{\LSI} (1-\eps) t} \left( F_\tau(\mu_0)-F_\tau(\nu) \right).
	\end{equation}

	It remains to upper-bound $F_\tau(\mu_0)-F_\tau(\nu)$. By the upper half of the entropy sandwich lemma \cite[Lemma~3.4]{chizat2018global}, we have for any $\mu$,
	\begin{equation}
		F_\tau(\mu)-F_\tau(\nu)
		\leq \tau \KLdiv{\mu}{\hmu}
		= \tau \KLdiv{\mu}{\nu} + \tau \int \d\mu \left( \log \nu - \log \hmu \right).
	\end{equation}
	Now by definition and by the bound $\frac1\tau \abs{F'[\mu] - F'[\nu]} \leq \tau^{-1} M_{10} \sqrt{2/c_{\mathrm{T2}}} \sqrt{\KLdiv{\mu}{\nu}} \eqqcolon B(\mu)$ derived above,
	\begin{align}
		\log \nu - \log \hmu
		&= \frac1\tau (F'[\mu] - F'[\nu])
		+ \log \int \exp\left( -\frac1\tau F'[\mu](y) \right) \d y
		- \log \int \exp\left( -\frac1\tau F'[\nu](y) \right) \d y \\
		&\leq B(\mu)
		+ \log \int \exp\left( -\frac1\tau F'[\nu](y) + B(\mu) \right) \d y
		- \log \int \exp\left( -\frac1\tau F'[\nu](y) \right) \d y \\
		&\leq 2 B(\mu).
	\end{align}
	Hence $F_\tau(\mu) - F_\tau(\nu) \leq \tau \KLdiv{\mu}{\hmu} \leq \tau \KLdiv{\mu}{\nu} + 2 \tau B(\mu)$.
	By combining this estimate with the convergence bound in function value shown above, we obtain that for any $\eps>0$, if $\mu_0$ is such that
	$\KLdiv{\mu_0}{\nu} + \tau^{-1} M_{10} \sqrt{2/c_{\mathrm{T2}}} \sqrt{\KLdiv{\mu_0}{\nu}} \leq \tau^2 M_{10}^{-2} (c_{\mathrm{T2}}/2) \left[ -\log(1-\eps) \right]^2$, then
	\begin{equation}
		\forall t \geq 0,~
		\KLdiv{\mu_t}{\nu}
		\leq e^{-2 \tau c_{\LSI} (1-\eps) t} \left[ \KLdiv{\mu_0}{\nu} + 2 \tau^{-1} M_{10} \sqrt{2/c_{\mathrm{T2}}} \sqrt{\KLdiv{\mu_0}{\nu}} \right].
	\end{equation}
	The first part of the proposition follows, since $\left[ -\log (1-\eps) \right]^2 \geq \frac{\eps^2}{10}$ for $\eps \leq \frac18$.
	
	For the second part of the proposition, where we assume that $M_{00} = \sup_{x,y,\tmu} \abs{F''[\tmu](x,y)}$ is finite (instead of $M_{10}$), 
	note that uniformly on $\RR^d$,
	\begin{align}
		\abs{\log \hmu - \log \nu + \cst} 
		= \frac1\tau \abs{F'[\mu] - F'[\nu]}
		&\leq \tau^{-1} M_{00} \norm{\mu - \nu}_{\mathrm{TV}} \\
		&\leq \tau^{-1} M_{00} \sqrt{2} \sqrt{\KLdiv{\mu}{\nu}} \\
		&\leq \tau^{-3/2} M_{00} \sqrt{2} \sqrt{F_\tau(\mu) - F_\tau(\nu)}
	\end{align}
	by Pinsker's inequality.
	The result then follows by reasoning similarly as for the first part, except $M_{10} \sqrt{2/c_{\mathrm{T2}}}$ is replaced by $\sqrt{2} M_{00}$ in \eqref{eq:apx_MFLD_localKL:T2W2_pinskerTV}.
\end{proof}

The condition 
$M_{00} = \sup_{\tmu,x,y} \abs{F''[\tmu](x,y)} < \infty$
can be verified for the two-layer neural network training setting considered in \cite{chizat2022mean,nitanda2022convex}, provided that
the loss $\ell(y_i, \cdot)$ is smooth and 
the activation $\sigma(\cdot^\top x_i)$ is bounded on $\RR^d$, uniformly for all $(x_i, y_i)$ in the training set.
However, we note that both $\sup_{\tmu,x,y} \abs{F''[\tmu](x,y)}$ and $\sup_{\tmu,x,y} \norm{\nabla_x F''[\tmu](x,y)}$ are unbounded in some other cases of interest, such as the MFLD of quadratic $F$ over $\RR^d$ with a square-distance interaction kernel $k(x,x') = \norm{x-x'}^2$ in \eqref{eq:intro:quadr_F}.

Furthermore, even in cases where both \autoref{prop:apx_MFLD_localKL:MFLD_localKL} and our main result \autoref{thm:MFG:gen} apply, we note that the rate guaranteed by the latter theorem is faster, since the optimal LSI resp.\ PI constants of the stationary measure are always ordered as $c_{\LSI} \leq c_{\PI}$.
The discrepancy between the rate estimates can be large for low temperatures $\tau$,
as the work of \cite[Sec.~2.4]{menz2014poincare} indicates that for $\nu \propto e^{-V(x)/\tau} \d x$, the constants typically scale as $c_{\LSI} \asymp \tau\, c_{\PI}$ when $\tau \to 0$.

\fi

\end{document}